\documentclass[preprint]{elsarticle}
\usepackage[numbers]{natbib}
\usepackage[colorlinks,citecolor=blue,urlcolor=blue]{hyperref}
\usepackage{graphicx}
\usepackage{amsmath,amsfonts,amssymb,amsopn,amscd,amsthm}
\usepackage{bbm,wasysym,stmaryrd}
\usepackage{subfigure}

\graphicspath{{Figures}}

\newcommand{\R}{\mathbbm{R}}
\newcommand{\ta}{\theta_{\alpha}}

\renewcommand{\L}{\mathbbm{L}}

\newcommand{\N}{\mathbbm{N}}

\newcommand{\erf}{\textrm{erf}}

\DeclareMathOperator{\eqlaw}{\stackrel{\mathcal{L}}{=}} 
\DeclareMathOperator{\eqdef}{\stackrel{\mathcal{\text{def}}}{=}}

\newcommand{\derpart}[2]{ \frac{\partial #1}{\partial #2} } 
\newcommand{\der}[2]{ \frac{\text{d} #1}{\text{d} #2} }  
\newcommand{\Exp}{\mathbbm{E}} 
\newcommand{\Ex}[1]{\mathbbm{E}\left [#1 \right]} 

\newcommand{\M}{\mathcal{M}}
\newcommand{\C}{\mathcal{C}}

\theoremstyle{plain}
\newtheorem{theorem}{Theorem}

\newtheorem{proposition}[theorem]{Proposition}
\newtheorem{lemma}[theorem]{Lemma}

\theoremstyle{definition}

\newenvironment{remark}{\par {\noindent \it \sc Remark.} \small \it } {}

\journal{Physica D}

\begin{document}
	\begin{frontmatter}
		\title{Mean-Field equations for stochastic firing-rate neural fields with delays: derivation and  noise-induced transitions}

		\author{Jonathan Touboul}
		\ead{jonathan.touboul@college-de-france.fr}
		\address{The Mathematical Neuroscience Lab, College de France, Center for interdisciplinary research in Biology (CIRB), CNRS UMR 7241, INSERM U1050, UPMC ED 158, MEMOLIFE PSL* \& \\
		BANG Laboratory, INRIA Paris\\
		11 place Marcelin Berthelot, 75005 Paris}

\begin{abstract}
In this manuscript we analyze the collective behavior of mean-field limits of large-scale, spatially extended stochastic neuronal networks with delays. Rigorously, the asymptotic regime of such systems is characterized by a very intricate stochastic delayed integro-differential McKean-Vlasov equation that remain impenetrable, leaving the stochastic collective dynamics of such networks poorly understood. In order to study these macroscopic dynamics, we analyze networks of firing-rate neurons, i.e. with linear intrinsic dynamics and sigmoidal interactions. In that case, we prove that the solution of the mean-field equation is Gaussian, hence characterized by its two first moments, and that these two quantities satisfy a set of coupled delayed integro-differential equations. These equations are similar to usual neural field equations, and incorporate noise levels as a parameter, allowing analysis of noise-induced transitions. We identify through bifurcation analysis several qualitative transitions due to noise in the mean-field limit. In particular, stabilization of spatially homogeneous solutions, synchronized oscillations, bumps, chaotic dynamics, wave or bump splitting are exhibited and arise from static or dynamic Turing-Hopf bifurcations. These surprising phenomena allow further exploring the role of noise in the nervous system.
\end{abstract}

\begin{keyword} 
noise, neural fields, collective dynamics,  bifurcations, Turing instabilities.
\end{keyword}

\end{frontmatter}

\pagestyle{myheadings}
\thispagestyle{plain}
\markboth{J. TOUBOUL}{Stochastic Neural Fields Dynamics}

{\bf Update:} This paper was updated to take into account measurability issues that may arise when considering that individual synapses have independent fluctuations. Actually, the mean-field dynamics that are analyzed in the manuscript correspond to a slightly distinct microscopic model than the one initially indicated, and that we updated here; the results on the dynamics and bifurcations are unchanged. We thank F. Delarue for noting this point.

\section{Introduction}
The activity of the brain is often characterized by large-scale macroscopic states resulting of the structured interaction of a very large number of neurons. This interaction yield meaningful signals accessible from non-invasive imaging techniques (EEG/MEG/Optical Imaging) and often used for diagnosis. Finer analysis of the brain's constitution identifies anatomical structures, such as the cortical columns, composed of the order of few thousands to one hundred thousand neurons belonging to a few different species, in charge of specific functions, sharing the same input and strongly interconnected.  Neurons composing these columns manifest highly complex behaviors often characterized by the intense presence of noise. They communicate through the emission of action potentials delivered after a specific delay due to the transport of information through axons at a finite speed and to the synaptic transmission. These delays have a clear role in shaping the neuronal activity, as established by different authors (see e.g.~\cite{roxin-brunel-etal:05,coombes-laing:11,roxin-montbrio:11,series-fregnac:02}). Two paradigmatic examples of this organization are the primary visual cortex of certain mammals and the rat's barrel cortex. In the primary visual cortex V1, neurons can be divided into orientation preference columns responding to specific orientations in visual stimuli, forming specific patchy connections \cite{hubel-wiesel-etal:78,bosking-zhang-etal:97}. Similarly, the rat's barrel cortex presents a clear columnar organization, and neurons responding to the sensory information of a particular whisker anatomically gather into barrels~\cite{woosley:69,kandel-schwartz-etal:00}. Several relevant brain states relying on the coordinated behaviors of large neural assemblies recently raised the interest of physiologists and computational neuroscientists, among which we shall cite the rapid complex answers to specific stimuli~\cite{thorpe-delorme-etal:01}, decorrelated activity~\cite{ecker-berens-etal:10,renart-de-la-rocha-etal:10}, large scale oscillations~\cite{buszaki:06}, synchronization~\cite{tabareau-slotine:10,izhikevich-polychronization:06}, and spatio-temporal pattern formation~\cite{ermentrout-cowan:80,coombes-owen:05,spreng-grady:10}. 

\emph{Neural fields} are intermediate-scale (mesoscopic) descriptions of neural networks activity. At this level of description, neurons gather into spatially localized homogeneous structures, the neural populations, containing sufficiently many neurons so that averaging effects occur, and that have small enough spatial extension to resolve quite fine topological or functional structure of the brain and its activity. Neural fields dynamics are almost exclusively studied through the use of heuristic continuum limits ever since the seminal work of Wilson, Cowan and Amari \cite{amari:72,amari:77,wilson-cowan:72,wilson-cowan:73}. In this model, the activity is represented through a macroscopic variable, the population-averaged firing rate, that is generally assumed to be deterministic. This approach successfully accounted for a number of phenomena, for instance for the problem of spatio-temporal pattern formation in spatially extended models, in relationship with visual hallucinations for instance (see e.g.~\cite{coombes-owen:05,ermentrout:98,ermentrout-cowan:79,laing-troy-etal:02}). However, these approaches neglect the presence of noise at neural field scale, implicitly making the assumption that the prominent noisy structure of individual neurons activity vanishes in the limit of large networks. Deterministic models describing neural field's activity are relevant approximations of the large networks activity provided that noise does not induce qualitative transitions. However, increasingly many researchers tend to consider that the different intrinsic or extrinsic noise sources produce a meaningful signal and conveys important information~\cite{rolls-deco:10}, and it is hence of high interest to analyze the effects of noise on the dynamics of neural fields, taking into the anatomical structure and noisy nature of neurons' activity.

Relating Wilson and Cowan type of models to the dynamics of stochastic neuronal networks is therefore a deep question in the field of computational neuroscience, and has been recently the subject of intense research. One of the main difficulties is to find relevant descriptions of the collective dynamics using suitable models of neuronal activity, and in particular including noise at the cellular level. Sparsely connected neural networks for integrate-and-fire neurons were analyzed in that view, and evidence regimes where the activity is uncorrelated~\cite{abbott-van-vreeswijk:93,amit-brunel:97,brunel-hakim:99}. In this case, the emergent global activity of the population in the limit of an infinite number of neurons in that case is deterministic, and evolves according to a mean-field firing rate equation. A distinct approach proposed to analyze a discrete-time Markovian model governing the firing dynamics of the neurons in the network, where the transition probability satisfies a master equation~\cite{buice-cowan:07,elboustani-destexhe:09}, and were developed for spatially extended networks in~\cite{bressloff:09}, a notable exception in this domain. In the limit where the number of the neurons is infinite, the behavior of the system reduces to standard Wilson and Cowan equations, and finite-size effects produce a stochastic perturbation that qualitatively modifies the solution of the global system~\cite{touboul-ermentrout:10}. Most of these approaches hold in some parameter regions, are based on statistical physics tools such as path integrals and Van-Kampen expansions, and the analysis of their dynamics involve a moment expansion and truncation. In these different limits, noise cancels out through averaging effects in the macroscopic descriptions. 

In this manuscript, following~\cite{touboulNeuralfields:11}, we consider the limit of interacting neural networks in the presence of external noise and stochastic synapses. In that view, neural fields are a particular limit of a set of interacting nonlinear stochastic processes with space-dependent interactions and propagation delays. The approach is evocative of statistical fluid mechanics and more generally interacting particle systems, a widely studied problem in mathematical physics chiefly motivated by thermodynamics or fluid dynamics questions. 
The network equations, analogous in the thermodynamic domain to the Newton equations on the free movement of particles in a gas, is shown in the neural field limit (the thermodynamic limit) to satisfy the propagation of chaos\footnote{Here the term \emph{chaos} is understood here in the statistical physics sense as the Boltzmann's molecular chaos ("Sto\ss zahlansatz"), corresponding to the independence between the velocities of two different particles before they collide. This is very different from the notion of chaos in deterministic dynamical systems, and in particular what we term chaos in the spatio-temporal patterns found in section~\ref{sec:TwoLayer}.} property ensuring that the state of neurons are independent provided that the initial conditions were also independent. The probability distribution of the state of a typical neuron is solution of an intricate implicit equation on the space of stochastic processes, an integro-differential McKean-Vlasov equation. This equation can be expressed as a nonlocal partial differential equation on probability distributions, and would correspond to the Boltzmann equation in the thermodynamics analogy (or in different settings, Vlasov-Poisson, Landau, \dots). The study of such equations is generally very complex to perform and prevents from understanding the dynamics of such neural fields. In the statistical fluid mechanics domain, a particularly successful method has been to derive macroscopic descriptions of observable quantities at a local thermodynamic equilibrium (where the speeds have a Gaussian distribution) such as the local density, macroscopic local velocity and local temperature fields. These are, in the thermodynamics case, given by Euler or Navier-Stokes equations. Nonlinear phenomena related to the dynamics of such fluids, as for instance turbulence and the formation of vortexes, is then evidenced in these equations.

Unfortunately, in the case of neural fields, there is \emph{a priori} no local Gaussian equilibria, and characterizing the solutions of the system is still an open problem. The central idea of this manuscript consists in instantiating a simplified yet relevant model of neuron, the firing-rate model. In that case, we demonstrate that the system has Gaussian attractive solutions (local equilibria). In the present case, we show that the characteristics of these Gaussian local equilibria, namely the two first moments, satisfy a closed system of deterministic delayed integro-differential equations. This exact reduction allows bringing the stochastic mean-field problem in the setting of usual studies of neural fields. This new set of equations, derived from a mathematically rigorous analysis of the limits of stochastic interacting neurons, has the nice property to be compatible with the usual Wilson and Cowan equations in the zero noise limit. Qualitative effects of noise on the dynamics of neural fields activity will hence be reached from the thus derived set of infinite-dimensional equations. 

The plan of the paper is as follows. In section~\ref{sec:MathSetting} we introduce the neural network model under investigation, and to this end recall a few  previous results on mean-field limits of spatially extended neural networks with propagation delays. We present the limit equation, as the number of neurons goes to infinity, of general neural networks in the neural-field limit, and discuss the existence and uniqueness of spatially homogeneous solutions (in law). It is also in this section that we use these results in the case of networks composed of firing-rate neurons. We show that the solution of this equation has Gaussian solutions and characterize the evolution of this distribution through deterministic equations on its mean and covariance. Well-posedness of these equations is analyzed, and most of the proofs of the new theoretical results are provided in the appendices A, B and C. Noise-induced transitions are analyzed in section~\ref{sec:Analysis}: we first deal with a one-population network in section~\ref{sec:OneLayer} and analyze the complex bifurcation structure observed when delays and noise are simultaneously varied. Section~\ref{sec:TwoLayer} focuses on a more realistic two-layers network account for an excitatory/inhibitory structure. We show in particular the destabilization of spatially homogeneous solutions, transitions towards perfectly synchronized oscillations through dynamic Turing-Hopf bifurcations involving chaotic spatio-temporal structures. These results are discussed from a biological viewpoint in the conclusion. 

\section{Mean-field equations for firing-rate neural networks}\label{sec:MathSetting}
In this section, we set up the context of the study, review some general results and apply these to networks of firing-rate neurons in order to derive our system of nonlinear delayed integro-differential equations the analysis of which will be the core of the next section.

\subsection{Mathematical Background}\label{sec:MathBack}
We consider networks composed of $N$ neurons falling into $P(N)$ homogeneous populations corresponding to specific locations on the neural field $\Gamma$\footnote{$\Gamma$ can either represent the physical space or a functional space (or both).}. The state of each neuron is described by a $d$-dimensional variable, belonging to a space denoted $E$, that generally describes the cell's voltage and different related ionic concentrations. Populations are distributed on $\Gamma$ according to a measure $\lambda$, a sum of Dirac measures in the finite population case (the case where $P(N)$ remains finite in the limit $N\to\infty$), or a distribution with regular density with respect to Lebesgue's measure in the continuous case. The $P(N)$ populations have their locations $r_{\alpha}$ independently drawn in the law $\lambda$, and we denote by $N_{\alpha}(N)$ (or simply $N_{\alpha}$) the number of neurons in population $\alpha$ in the network of size $N$. The interconnection between a neuron $i$ of population $\alpha$ located at $r_{\alpha}\in \Gamma$ and a neuron of population $\beta$ at location $r_{\beta}$ is characterized by a synaptic coefficient $J(r_{\alpha},r_{\beta})$ and a propagation delay $\tau(r_{\alpha}, r_{\beta})$ depending on the distance between the populations and the propagation speed (see Fig.~\ref{fig:Neurons}). 

\begin{figure}[htbp]
	\centering
		\includegraphics[width=.5\textwidth]{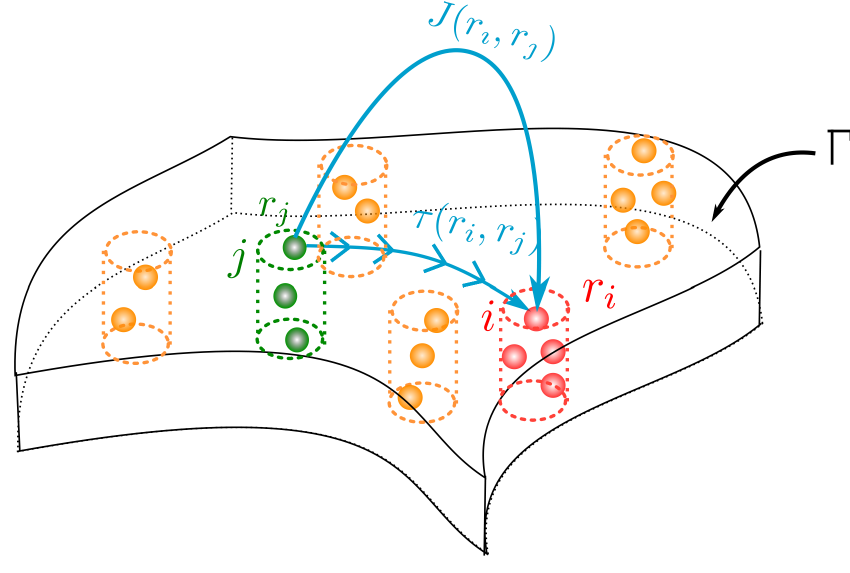}
	\caption{A typical architecture of neural field: cylinders represent neural populations as cortical columns spanning across the cortex. Neuron $j$ in the green population at location $r_{j}\in \Gamma$ communicates with neuron $i$ in the red population at location $r_i$, and the information sent is received with a delay $\tau(r_i,r_j)$ depending on the spatial location of each neuron, and with a synaptic weight $J(r_i,r_j)$ also depending on the neural populations of $i$ and $j$. }
	\label{fig:Neurons}
\end{figure}

We work in a complete probability space $(\Omega, \mathcal{F},\mathbbm{P})$ satisfying the usual conditions endowed with a filtration $\big(\mathcal{F}_t\big)_t$. Neuron $i$ in population $\alpha$ satisfies the equations:

\begin{multline}\label{eq:NetworkSpace}
	\displaystyle{d\,X^{i,N}_t=G(r_{\alpha},t,X^{i,N}_t) \, dt + \frac{{\lambda(\Gamma)}}{P(N)}\sum_{\gamma=1}^{P(N)} \sum_{j=1}^{N_{\gamma}} \frac{1}{N_{\gamma}} \int_{-\tau}^0 b(r_{\alpha},r_{\gamma},X^{i,N}_t,X^{j,N}_{t+s})d\eta(r_{\alpha},r_{\gamma},s) \,dt} \\
	\displaystyle{+ g(r_{\alpha},t,X^{i,N}_t) dW^{i}_t
	+ \frac{{\lambda(\Gamma)}}{P(N)}\sum_{\gamma=1}^{P(N)} \sum_{j=1}^{N_{\gamma}} \frac{1}{N_{\gamma}} \int_{-\tau}^0 \beta(r_{\alpha},r_{\gamma},X^{i,N}_t,X^{j,N}_{t+s})d\mu(r_{\alpha},r_{\gamma},s) dB^{i}_t.}
\end{multline}
In this equation, the free dynamics of each neuron is governed by a drift function $G:\Gamma\times\R\times E \mapsto E$ accounting for the intrinsic dynamics and deterministic inputs of neuron $i$. Stochastic effects are characterized through the diffusion matrix $g:\Gamma\times\R\times E \mapsto \R^{d\times m}$ and the  $m$-dimensional independent adapted Brownian motions $(W^{i}_t)_{i\in\N}$. The transmission delays are taken into account through the signed finite measures $\eta$ and $\mu$. Interactions between neurons are of two types: deterministic interactions are governed by the function $b$ and stochastic variations of the synaptic weights through the function $\beta$ and the $d \times d$ independent adapted Brownian motions\footnote{{\bf Update: } the Brownian motions are not a function of the population location $r_{\gamma}$.} $(B^i_t)_{i\in\N}$. Due to the presence of delays, each component of the initial conditions is a continuous function belonging to $\C=C([-\tau,0],E)$.

There is a critical competition between the number of populations and the total number of neurons: in order for averaging effects to occur in the neural field, a very large number of neurons in each population is required, competing with the total number of populations. The \emph{neural field regime} described in~\cite{touboulNeuralfields:11} assumes that:
\begin{equation}\label{eq:PopulationEstimate}
	\mathbbm{e}(N)\eqdef\frac 1 {P(N)} \sum_{\gamma=1}^{P(N)} \frac 1 {N_{\gamma}(N)} \mathop{\longrightarrow}\limits_{N\to\infty} 0
\end{equation}
In the finite population case, it corresponds to the assumption that the number of neurons in every population tends to infinity. For continuous neural fields, this assumption is slightly less stringent and allows a few populations to have a finite number of neurons (these populations will be of measure $0$). This regime is relevant for modeling neural fields, since at this scale neuronal populations contain a number of neurons orders of magnitude larger than the number of populations (see e.g.~\cite{changeux:06}).

Under mild regularity conditions on the functions governing the network equation, it was proved in~\cite{touboulNeuralfields:11} that the system enjoys the propagation of chaos property. This means that if the initial conditions of the neurons were independent and identically distributed at the level of each population (such initial conditions are termed \emph{chaotic}), neurons remain independent during the evolution. In details, for $l\in \N^*$\footnote{In the manuscript, we classically denote by $\N$ the set of natural integers, $\mathbbm{Z}$ the set of integers and $\R$ the set of real numbers. A star exponent indicates that $0$ is not taken into account, and a $+$ indicates that only positive quantities are considered. } neurons indexed by $\{i_1,\cdots,i_l\}$, the law of $(X^{i_1,N}_t, \cdots, X^{i_l,N}_t, -\tau\leq t \leq T)$ are independent processes when $N\to\infty$. In that limit, $X^{i}_t$ has the law of $\bar{X}_t(r)$, the unique solution of the mean-field equations:
	\begin{multline}\label{eq:MFESpaceSummary}
		\displaystyle{d\,\bar{X}_t(r)=G(r,t,\bar{X}_t(r)) \, dt + \int_{\Gamma} \int_{-\tau}^0 \Exp_{{Z}}[b(r,r',\bar{X}_t(r),{Z}_{t+s}(r'))]d\eta(r,r',s) \, \lambda(r') dr' \,dt} \\
		\displaystyle{+ g(r,t,\bar{X}_t(r)) dW_t
		+ \int_{\Gamma} \int_{-\tau}^0 \Exp_Z[\beta(r,r',\bar{X}_t(r),{Z}_{t+s}(r'))]d\mu(r,r',s) \lambda(r')\,dr' dB_t.}
	\end{multline}
	In this equation, $(Z_t(r))$ denotes a process independent of $\bar{X}_t(r)$ with the same law, and $\Exp_{Z}$ denotes the expectation with respect to the process $Z$. The processes $(W_t)_{{t\geq 0}}$ and $(B_t)_{t\geq 0}$ are independent Brownian motions\footnote{In order to emphasize the propagation of chaos property, we initially defined the equations with spatially chaotic Brownian motions, which does not modify the law of the solution, as we discussed in J. Touboul, Journal of Statistical Physics 56 (3), p. 546-573 (2014)}. Rigorously, the law of $\bar{X}_t(r)$ is given by the solution of:
	\begin{multline*}
		\displaystyle{d\,\bar{X}_t(r)=G(r,t,\bar{X}_t(r)) \, dt + \mathcal{E}_{r'} \left [\int_{-\tau}^0 \Exp_{{Z}}[b(r,r',\bar{X}_t(r),{Z}_{t+s}(r'))]d\eta(r,r',s) \,dt\right]} \\
		\displaystyle{+ g(r,t,\bar{X}_t(r)) dW_t
		+ \mathcal{E}_{r'} \left[ \int_{-\tau}^0 \Exp_Z[\beta(r,r',\bar{X}_t(r),{Z}_{t+s}(r'))]d\mu(r,r',s)\right]  dB_t,}
	\end{multline*}
	where $\mathcal{E}_{r'}$ is the expectation with respect to the distribution of the population locations over $\Gamma$ with distribution $\lambda(\cdot)/\lambda(\Gamma)$. The initial conditions are given by a process $\zeta_t(r)$ belonging to the space of mappings of $\Gamma$ with values in $\C$.
In particular, considering a case with a finite number $P$ of populations located at positions $(r_{\alpha},\alpha\in\{1,\ldots,P\}) \in \Gamma^P$, the mean-field equations read:
	\begin{multline*}
		\displaystyle{d\,\bar{X}_t(r_{\alpha})=G(r_{\alpha},t,\bar{X}_t(r_{\alpha})) \, dt + \sum_{\gamma=1}^P \int_{-\tau}^0 \Exp_{{Z}}[b(r_{\alpha},r_{\gamma},\bar{X}_t(r_{\alpha}), {Z}_{t+s}(r_{\gamma}))] d\eta(r_{\alpha},r_{\gamma},s) \,dt} \\
			\displaystyle{+ g(r_{\alpha},t,\bar{X}_t(r_{\alpha})) dW_t
			+ \sum_{\gamma=1}^P \int_{-\tau}^0 \Exp_{{Z}}[\beta(r_{\alpha},r_{\gamma},\bar{X}_t(r_{\alpha}),{Z}_{t+s}(r_{\gamma}) )]d\mu(r_{\alpha},r_{\gamma},s)  dB_t .}
	\end{multline*}	
	
These equations can appear very formal due to the generality of the models considered and to the mathematical approach developed of~\cite{touboulNeuralfields:11}. In particular, even if the mean-field approach reduces an infinite system of interacting diffusion processes into a single well-posed equation, a major issue one faces is the concrete identification, characterization and simulation of the solution of the mean-field equations (see e.g.~\cite{BFFT,talay-vaillant:03}). In our spatial and delayed setting, this concern is even more true. \\
In order to qualitatively characterize the dynamics of these equations, we start by investigating spatially homogeneous solutions in law, before instantiating a particular neuronal dynamics.

\subsection{Spatially Homogeneous solutions and synchronization in law}
Spatially homogeneous solutions, sometimes called synchronized solutions, correspond to neural fields regimes in which neurons manifest the same behavior after a transient phase. A similar phenomenon, called polychronization, corresponds to the fact that the neural field form a few homogeneous clusters (see e.g.~\cite{izhikevich-polychronization:06}). Our stochastic setting suggests to extend these notion to synchronization or polychronization \emph{in law}: i.e. probability distributions, after a transient phase, synchronize (or polychronize). The propagation of chaos property implies that individual neurons belonging to the same population are synchronized in law since their probability distribution are solution of the same equation with the same initial condition. A more complex question is the synchronization or polychronization in law across populations. The following proposition provides a simple sufficient condition for the existence and uniqueness of spatially homogeneous solutions in law.

\begin{proposition}\label{pro:Synchro}
	Assume that the distribution of the initial condition $(\zeta^0_t(r), -\tau\leq t \leq 0)$ is chaotic and independent of $r$, and that the functions $G(r,t,x)$ and $g(r,t,x)$ do not depend on $r$. Moreover, if the law of the quantities:
	\begin{equation}\label{eq:TotalSynchro}
		\begin{cases}
			B(r,x,\varphi)\eqdef\int_{\Gamma}\int_{-\tau}^0  b(r,r',x,\varphi(u)) d\eta(r,r',u)\lambda(r')dr'\quad \text{and}\\ 		
			H(r,\psi,\varphi)\eqdef\int_{\Gamma} \int_0^t \int_{-\tau}^0 \beta(r,r',\psi(s),\varphi(s+u)) d\mu(r,r',u) dB_{s} \lambda(r')dr'
		\end{cases}
	\end{equation}
	do not depend on $r$ for any $(\psi,\varphi)$ measurable functions, then the solution of the mean-field equation~\eqref{eq:MFESpaceSummary} is spatially homogeneous in law.	The law of the spatially homogeneous is solution of the implicit equation:
	\begin{multline}\label{eq:SynchronizedLaw}
		\displaystyle{X_t (r_0)=\zeta_0(r_0) + \int_0^t ds \Big( G(r_0,s,X_s(r_0)) + \Exp_{Z} [B(r_0,X_s(r_0),Z_{(\cdot)}(r_0))]\Big)}\\
	 \displaystyle{+\Exp_{Z} [ H(r_0,X_{(\cdot)}(r_0), Z_{(\cdot)}(r_0) )] + \int_0^t dW_s g(r_0,s,X_s(r_0)),}
	\end{multline}
	where $Z\eqlaw X$ and is independent of $X$. This equation has a unique solution. 
\end{proposition}

The proof of this proposition is performed in appendix~\ref{append:ExistUniqueSynchro}. The demonstration of the existence of stationary solutions uses the characterization of the solutions as the limit of Picard iterates. This technique, usually used in proofs of existence and uniqueness of solutions, shows that starting from an arbitrary process, the iterates of a particular function $\Phi$ converge towards the solution of the equation. The idea is to use the degree of freedom we have on the process initializing the recursion. We show that the set of spatially homogeneous processes is invariant under the iteration of $\Phi$. Choosing a spatially homogeneous initial process produces a sequence of spatially homogeneous processes, hence the limit of this sequence (which is the unique solution of the mean-field equations) is spatially homogeneous.

	Let us focus more specifically on the assumptions of the proposition. First of all, spatially homogeneous drift and diffusion functions are obvious necessities in order for the neurons to have identical responses. The condition~\eqref{eq:TotalSynchro} ensures that the global input a neuron receives is independent of its spatial location. These are hence relatively strong assumptions. However, this sufficient condition is sometimes also necessary. This is for instance the case of the system analyzed in section~\ref{sec:Analysis}. Indeed, considering $b(r,r',x,y)=J(r,r')S(y)$ and $\beta(r,r',x,y)=\sigma(r,r')S(y)$, then the existence of a spatially homogeneous solution in law implies the fact that:
	\[r\mapsto \int_{\Gamma} \int_{-\tau}^0 J(r,r') \Exp[{S(\bar{X}_t(r'))}]d\eta(r,r',s) \, \lambda(r') dr'\]
	is constant (and similarly for the synaptic case). For $\bar{X}$ a spatially homogeneous solution, the term $\Exp[{S(\bar{X}_t(r'))}]$ does not depend on $r'$. If this term is not zero (which is generally the case since we consider positive sigmoids), this implies that necessarily:
	\[r\mapsto \int_{\Gamma} \int_{-\tau}^0 J(r,r') d\eta(r,r',s) \, \lambda(r') dr'\]
	is constant, hence ~\eqref{eq:TotalSynchro} is satisfied. 
	
	In other particular cases, it can occur that spatially homogeneous solutions exist even if the condition of the proposition is not satisfied, for instance in the latter case when $\Exp[{S(\bar{X}_t(r')}]=0$.

Analogous ideas allow extending such conditions for polychronization in law. It would amount ensuring that there exists a partition of $\Gamma$ into different clusters such that each cluster receiving the same input from the others.

\subsection{Mean-field equations for firing-rate models}\label{sec:FiringRateModels}
We now apply the theory exposed to revisit from a probabilistic viewpoint Wilson and Cowan firing-rate approach~\cite{wilson-cowan:72,wilson-cowan:73,venkov-coombes-etal:07}. 
In that model, the state of each neuron is described by a scalar quantity representing the voltage of each neuron, assumed to have a linear intrinsic dynamics $G(r,t,x)=-x/\theta(r) +I(r,t)$ where $\theta(r)$ is the characteristic time of the membrane potential and $I(r,t)$ the deterministic external input received by the neurons. They receive noisy input driven by a Brownian motion, with a diffusion coefficient $g(r,t)=\Lambda(r,t)$, and interact through their mean firing rate assumed to be a sigmoidal transform of the voltage variable. The delayed interactions are written as a sum over all neurons of a synaptic coefficient only depending on the populations the interacting neurons belong to and a sigmoidal transform of the pre-synaptic neuron (the one sending a current), $b(r,r',x,y)=J(r,r')S(r',y)$. The functions $S(r,x)$ are assumed uniformly bounded and uniformly Lipschitz-continuous with respect to $x$. We also consider that the interconnection weights are noisy, and specify the function $\beta(r,r',x,y)=\sigma(r,r') S(r',y)$. This model is a relatively simple case of the general study reviewed in section~\ref{sec:MathBack} and clearly satisfies the regularity conditions of~\cite{touboulNeuralfields:11}. In order to simplify further our analysis, we will assume that the delay measures $\eta_{\alpha\gamma}(u)$ and $\mu_{\alpha,\gamma}(u)$ (respectively $\eta(r,r',u)$ and $\mu(r,r',u)$) are Dirac measures at fixed times $\tau(r,r')$, which generally can be considered to be $\Vert r-r'\Vert_{\Gamma}/c + d$ where $c$ would the transport velocity in the axons and $d$ the typical time of the synapse. 

The equation of the dynamics of neuron $i$ of population located at $r_{\alpha} \in \Gamma$ in the network with $N$ neurons and $P(N)$ populations reads:
\begin{multline}\label{eq:NetworkFiringRateSpace}
\displaystyle{dV^{i,N}(t) = \Bigg( -\frac 1 {\theta(r_{\alpha})}  V^{i,N}(t) + I(r_{\alpha}, t) + \sum_{\gamma=1}^{P(N)} J({r_{\alpha},r_{\gamma}}) \; \frac{1}{N_{\gamma}} \sum_{j=1}^{N_{\gamma}} S(r_{\gamma},V^{j,N}(t-\tau(r_{\alpha},r_{\gamma}))) \Bigg) \, dt}\\
\displaystyle{+ \Lambda(r_{\alpha},t) dW^{i}_t + \sum_{\gamma=1}^{P(N)} \sigma(r_{\alpha}, r_{\gamma}) \Bigg(\frac{1}{N_\gamma} \sum_{j=1}^{N_{\gamma}} S(r_{\gamma},V^{j,N}(t-\tau(r_{\alpha},r_{\gamma})))\Bigg) \,dB^{i}_{t}.}
\end{multline}
In section~\ref{sec:TwoLayer} we will consider neural fields composed of different layers (for instance excitatory and inhibitory neurons). The results presented in this theoretical analysis readily extend to such multiple layers neural fields. 

The results summarized in section~\ref{sec:MathBack} readily apply to the present case. In details, for $T>0$ a fixed time horizon, assuming that the initial conditions of the network equations are chaotic, then activity $V^{i,N}$ of neuron $i$ in population $\alpha$ converges in law towards the process $\bar{V}(r_{\alpha})$ where $(\bar{V}_t(r))$ is the unique solution of the mean-field equation:
\begin{multline}\label{eq:MFEFiringRateSpace}
	\displaystyle{d\bar{V}_t(r)=\left(-\frac 1 {\theta(r)}  \bar{V}_t(r)+I(r,t)+\int_{\Gamma}  J({r,r'}) \Exp[S(r',\bar{V}_{t-\tau(r,r')}(r'))]\lambda(r')dr'\right)dt}\\
	\displaystyle{+\Lambda(r,t) dW_t+ \int_{\Gamma} \sigma({r,r'}) \Exp[S(r',\bar{V}_{t-\tau(r,r')}(r'))] \lambda(r')dr'dB_t.}
\end{multline}
In this equation, the processes $(W_t)$ and $(B_t)$ are independent Brownian motions, and the expectation term is the probabilistic expectation under the distribution of $\bar{V}$. Moreover, the propagation of chaos property applies, i.e. in the limit $N\to\infty$, finite sets of neurons are independent.

This result characterizes the behavior of the system up to a finite time $T$. When considering the problem of stability and bifurcations of stationary or periodic solutions, this limitation is problematic. Indeed, the property does not ensure that the asymptotic regimes of the mean-field equations accurately correspond to solutions of the network equation, which would be ensured by a uniform propagation of chaos property (see~\cite{mischler-mouhot:11}). This is the object of the following theorem.

\begin{theorem}[Uniform propagation of chaos]\label{thm:UniformPropaChaos}
	If the Lipschitz constant of the sigmoid $S$ is small enough, then the convergence of the network equations towards the mean-field equations is uniform in time, i.e. there exists a constant $C>0$ such that for all $T>0$, for all $i\in\N$ a neuron at location $r\in\Gamma$,
	\[\sup_{0\leq t\leq T} \Exp[ \vert V^i_t - \bar{V}^i_t \vert ] \leq \frac {C}{\sqrt{N}}\]
	for $\bar{V}^i_t$ the a particular process with law $\bar{V}_t(r)$ termed the coupled process. In this inequality, $C$ only depends on the parameters of the system and is independent of $T$.
\end{theorem}

A more precise quantitative statement on the parameters and the proof of this Theorem are provided in~\ref{append:UniformPropagation}. The proof is based on a thorough control of the convergence, making a fundamental use of the linearity of the equation.

Now that the convergence towards the mean-field equations~\eqref{eq:MFEFiringRateSpace} has been quantified, we characterize the solutions of these equations.

\begin{theorem}\label{pro:GaussianSpace}
	If the initial condition ${V}^0(r)$ is a Gaussian chaotic process, the solution of the mean-field equations \eqref{eq:MFEFiringRateSpace} with initial conditions ${V}^0(r)$ is Gaussian for all time. Let us denote by $\mu(r,t)$ its mean and by $v(r,t)$ its variance. The term $\Ex{S(r,{V}_t(r)})$ is a function of $r$, $\mu(r,t)$ and $v(r,t)$ only, denoted $f(r,\mu,v)$. We have:
	\begin{equation}\label{eq:DDEIntegroDiff}
			\begin{cases}
				\displaystyle{\derpart{\mu}{t}(r,t)=-\frac 1 {\theta(r)} \mu(r,t) + \int_{\Gamma}  J(r,r') f\big( r,\mu(r',t-\tau(r,r')),v(r',t-\tau(r,r'))\big)\lambda(r')dr'+I(r,t)}\\
				\\
				\displaystyle{\derpart{v}{t}(r,t)=-\frac 2 {\theta(r)} \,v(r,t) + \int_{\Gamma}  \sigma(r,r')^2 f\big (r,\mu(r',t-\tau(r,r')),v(r',t-\tau(r,r'))\big)^2\lambda(r')^2 dr' +\Lambda^2(r,t)}
			\end{cases}
		\end{equation}
		with initial condition $\mu(r,t) =\Ex{{V}^0_t(r)}$ and $v(r,t) =\Exp{[({V}^0_t(r)-\mu(r,t))^2]}$ for $t\in [-\tau, 0]$ and $r\in\Gamma$. 
\end{theorem}

\begin{proof}
Using the variation of constant formula, it is easy to show that the unique solution of the mean-field equations \eqref{eq:MFEFiringRateSpace} with initial condition $V^0$ satisfies the implicit equation:
		\begin{multline}\label{eq:SoluMFESpace}
		\displaystyle{V_t(r)=e^{-\frac{t}{\theta(r)}} V^0_0(r)+e^{-\frac{t}{\theta(r)}}\Bigg(\int_0^t e^{\frac{s}{\theta(r)}} \Big(I(r,s)+\int_{\Gamma} J({r,r'}) \Exp[S(r',{V}_{t-\tau(r,r')}(r'))]\lambda(r')dr'\Big)  
		ds}\\
		\displaystyle{+\int_0^t e^{\frac{s}{\theta(r)}} \Lambda(r,s) dW_s+ \int_{\Gamma} \sigma({r,r'}) \int_0^t e^{\frac{s}{\theta(r)}} \Exp[S(r',{V}_{s-\tau(r_,r')}(r'))] \lambda(r')dr'dB_s\Bigg).}
		\end{multline}
It is clear from this formulation that the righthand side is a Gaussian process as the sum of a deterministic function and stochastic integrals of deterministic functions with respect to Brownian motions, and hence $V_t(r)$ also is\footnote{This property can also be proved by using the classical characterization of the solution of the mean-field equation as the limit of the iteration of the map $\Phi$ as proved in~\cite{touboulNeuralfields:11} and used in~\ref{append:ExistUniqueSynchro}. Picard's iterations are initialized with a Gaussian process, and it is very simple to show that the space of Gaussian processes is invariant under $\Phi$. ${V}$ will hence be defined as the limit of a sequence of Gaussian processes, hence Gaussian itself.}.

For $X$ a Gaussian process with mean $\mu$ and variance $v$, the term $\Ex{S(r,X)}$ is a function of $f(r,\mu,v)=\int_{\R} S(r,x\sqrt{v}+\mu)Dx$
with $Dx=e^{-x^2/2}/\sqrt{2\pi}$ the standard Gaussian distribution. Taking the expectation and the variance of the process given by the implicit equation~\ref{eq:SoluMFESpace}, we obtain the following equations:
\begin{multline*}
	\displaystyle{\mu(r,t)=e^{-\frac{t}{\theta(r)}}\bigg(\mu(r,0)+\int_0^t e^{\frac{s}{\theta(r)}} \Big ( I(r,s) }\\
			 \displaystyle{+\int_{\Gamma} J({r,r'}) f(r',\mu(r',s-\tau(r,r')),v(r',s-\tau(r,r')))\lambda(r')dr' \Big) 
			ds\bigg)}
\end{multline*}
and 
\begin{multline*}
			\displaystyle{v(r,t)=	e^{-\frac{2t}{\theta(r)}}\bigg(v(r,0)+\int_0^t e^{\frac{2s}{\theta(r)}}\Big( \Lambda(r,s)^2} \\ 
	 \displaystyle{+\int_{\Gamma}  \sigma^2({r,r'}) f^2(r',\mu(r',s-\tau(r,r')),v(r',s-\tau(r,r')))\lambda^2(r')dr'\Big) 
		ds\bigg),}
\end{multline*}
which are equivalent to system~\eqref{eq:DDEIntegroDiff}. 
\end{proof}

\begin{remark}
	\begin{itemize}
		\item Formula~\eqref{eq:SoluMFESpace} shows that the initial condition exponentially vanishes. Thanks to the uniform propagation of chaos (Theorem~\ref{thm:UniformPropaChaos}), any choice of chaotic initial condition will approach the thus described Gaussian solution.
		\item The mean and variance characterize the law of $V$ since the covariance $C(r,r',t_1,t_2)$ of ${V}_{t_1}(r)$ and ${V}_{t_2}(r')$ is a simple function of $\mu(r,t)$ and $v(r,t)$. Indeed, for $r\neq r'$, the covariance is null because of the independence of the initial conditions and of the Brownian motions involved at two different space locations. For $r=r'$:
				\begin{multline*}
					\displaystyle{C(r,r,t_1,t_2) = e^{-(\frac{t_1+t_2}{\theta(r)})} v(r,0)
					+ \int_0^{t_1\wedge t_2} e^{\frac{2 s}{\theta(r)}} \Lambda(r,s)^2\,ds}\\
					\displaystyle{+ \int_0^{t_1\wedge t_2} e^{\frac{2 s}{\theta(r)}}\int_{\Gamma}\lambda(r')^2\,dr' \sigma(r,r')^2 f^2(r',\mu(r',s-\tau(r,r')),v(r',s-\tau(r,r'))\,ds }
				\end{multline*}
				\item If $S(r,x) = \erf(g(r) x+h(r))$, the function $f(r,\mu,v)$ takes the simple form (see~\cite{touboul-hermann-faugeras:11}):
				\begin{equation*}
					f(r,\mu,v) = \erf\left(\frac{g(r)\, \mu + h(r)}{\sqrt{1+g^{2}(r)v}}\right).
				\end{equation*}
				This simple expression motivates the choice of $\erf$ sigmoids in section~\ref{sec:Analysis}.
				\item In a finite-population case, or for spatially homogeneous solutions, the equations are the following delayed differential equations:
				\begin{equation}\label{eq:DDE}
				\begin{cases}
					\displaystyle{\dot{\mu}_{\alpha}(t)=-\frac 1 {\ta} \mu_{\alpha}(t) + \sum_{\beta=1}^P J_{\alpha\beta}f_{\beta}\big(\mu_{\beta}(t-\tau_{\alpha\beta}),v_{\beta}(t-\tau_{\alpha\beta})\big)+I_{\alpha}(t)} & \alpha=1\ldots P\\
					\displaystyle{\dot{v}_{\alpha}(t)=-\frac 2 {\ta} \,v_{\alpha}(t) + \sum_{\beta=1}^P \sigma_{\alpha\beta}^2 f_{\beta}^2\big(\mu_{\beta}(t-\tau_{\alpha\beta}),v_{\beta}(t-\tau_{\alpha\beta})\big) +\lambda^2_{\alpha}(t)}  & \alpha=1\ldots P
				\end{cases}
				\end{equation}
				(we used indexes to label the populations instead of their spatial location).
	\end{itemize}
\end{remark}

\begin{theorem}[Well posedness of the moment equations]\label{thm:ExistenceUniquenessMoments}
	Under non-degeneracy conditions on the process, there exists a unique solution to the moment equations~\eqref{eq:DDEIntegroDiff} and \eqref{eq:DDE}.
\end{theorem}
The precise statement of this theorem as well as the rigorous proof are developed in appendix~\ref{append:ExistenceUniqueness}. 

One of the main interest of the theorem~\ref{pro:GaussianSpace} is to rigorously describe the stochastic dynamics of the complex stochastic mean-field equations through two coupled integro-differential equations. A very interesting feature of the system~\eqref{eq:DDEIntegroDiff} is that it is compatible with usual Wilson and Cowan equations in the zero noise limit. The system precisely quantifies the fact that when in the presence of noise, mean and covariance interact in a nonlinear fashion. In these equations, noise appears as a {parameter} of a deterministic dynamical system, which allows characterizing the dynamics of the stochastic equations using the well developed bifurcation theory in Hilbert spaces.

\section{Noise-induced transitions}\label{sec:Analysis}

In this section we analyze the dynamics of the mean-field described by~\eqref{eq:DDEIntegroDiff}. We particularly focus on the effects of noise on the solutions. We first analyze a one-layer network, in a case where analytical study is possible, before addressing the more relevant case of two layers networks involving an excitatory and an inhibitory populations.

\subsection{Noise-induced stabilization of spatially homogeneous regimes in a one-layer system}\label{sec:OneLayer}
In this section, we consider the case of a single-layer neural field, whose mean and standard deviation satisfy equation~\eqref{eq:DDEIntegroDiff}, and investigate the stability of spatially homogeneous regimes as a function of noise levels. We consider a one-dimensional neural field distributed homogeneously on $\Gamma=\mathbbm{S}^1$ (i.e. $\Gamma=[0,1]$ with periodic boundary conditions and $\lambda(dx)=dx$). This choice is motivated by functional neural fields modeling orientation preference. In order to simplify the analysis, we assume that the connectivity functions $J(r,r')$ and $\sigma(r,r')$ only depend on the distance $\vert r-r'\vert$, that $\theta(r)$ is constant and that $S(r,x)$ only depends on $x$. We further assume that $\Lambda(r,t)$ is a constant equal to simply noted $\Lambda$. Since the neural field is periodic and the connectivity functions are convolutional, spatially homogeneous solutions exist by direct application of Proposition~\ref{pro:Synchro}. The spatially homogeneous state satisfies the equations:
\[
\begin{cases}
	\der{\mu}{t}=-\frac{\mu}{\theta} + \int_0^1 J(r) f(\mu(t-\tau(r)),v(t-\tau(r)))\,dr + I\\
	\der{v}{t}=-\frac{2\,v}{\theta} + \int_0^1 \sigma^2(r) f^2(\mu(t-\tau(r)),v(t-\tau(r)))\,dr + \Lambda^2\\
\end{cases}
\]

Let us denote $\mathcal{J}$ (resp. $\tilde{\sigma}^2$) the integral $\int_0^1 J(r,r')dr'$ (resp. $\int_0^1 \sigma^2(r,r')dr'$), assumed finite (these obviously do not depend on $r$).  Taking $S(r,x)=\erf(gx)$, we have seen that $f(r,x,y)$ is equal to $\erf(gx/\sqrt{1+g^2y})$. Hence $F_0\eqdef f(r,0,v)$ does not depend on $v$. Let us further set $I=-\mathcal{J} F_0$. In that case, $\mu(r,t)\equiv 0$ for any $(r,t)\in\Gamma\times\R^+$ is a solution of the mean equation whatever the standard deviation $v$.
Possible spatially homogeneous equilibria $(\bar{\mu},\bar{v})$ are solution of the equations:
\[
\begin{cases}
	-\frac{\bar{\mu}}{{\theta}} + \mathcal{J} f(\bar{\mu},\bar{v}) + I &=0\\
	-\frac{2\,\bar{v}}{\theta} + \tilde{\sigma}^2 f^2(\bar{\mu},\bar{v}) + \Lambda^2&=0\\
\end{cases}
\]
and a trivial solution is given by $(\mu_0,v_0)\eqdef \left(0, \frac{\theta}{2} (\tilde{\sigma}^2\,F_0^2 +\Lambda^2)\right)$.

Let us start by considering the non-delayed case. It is easy to demonstrate that if the slope $g$ is small enough (precisely $g<\sqrt{2\pi/(\theta^2\mathcal{J}^2-2\pi v_0)}$), $(\mu_0,v_0)$ is stable and is the unique fixed point of the system. For larger values of $g$, the system presents two distinct regimes as a function of the noise parameters: for large values of the parameters $\sigma$ and $\Lambda$, $(\mu_0,v_0)$ is the unique spatially homogeneous equilibrium and it is stable, and for smaller values of the noise parameters, $(\mu_0,v_0)$ is unstable and there are two additional equilibria denoted $(\mu_1,v_1)$ and $(\mu_2,v_2)$. Moreover, the system undergoes a pitchfork bifurcation along the ellipse:
\[\tilde{\sigma}^2\,F_0^2 +\Lambda^2 = \left(\frac{\mathcal{J}^2\theta^2g^2}{2\pi}-1\right)\frac {2}{\theta g^2},\] 
and on this line the two fixed points $(\mu_1,v_1)$ and $(\mu_2,v_2)$ collapse on $(\mu_0,v_0)$ and disappear. 

This simple analytical study provides a first example of the influence of noise levels on the nature and stability of spatially homogeneous equilibria of the mean-field equations. Let us now consider the full spatially extended system with delays and address the problem of pattern formation beyond a Turing instability around the spatially homogeneous steady state $(\mu_0,v_0)$. To this purpose, we analyze the linear stability of this fixed point (see e.g.~\cite{bressloff:96,hutt-bestehorn:03,venkov-coombes-etal:07}), which amounts characterizing the eigenvalues of the linearized equations around this fixed point. Since $f(x,y)=\erf(gx/\sqrt{1+g^2y})$, we clearly have:
\[\begin{cases}
	\derpart{f}{x}\vert_{(\mu_0,v_0)}=\frac{g}{\sqrt{1+g^2v_0}} \frac{1}{\sqrt{2\pi}}\eqdef F'_0\\
	\derpart{f}{y}\vert_{(\mu_0,v_0)}=\frac{-g^3\mu_0}{\sqrt{\pi(1+g^2v_0)^3}} e^{-\frac{\mu_0^2g^2}{1+g^2v_0}} =0.
\end{cases}\]
The linearized equations around $(\mu_0,v_0)$ hence read:
\[
\begin{cases}
	\derpart{A}{t}(r,t) &= -\frac {1} {\theta} A(r,t) + F'_0\,\int_{\Gamma} J(r'-r)A(r',t-\tau(r'-r))\,dr'\\
	\derpart{B}{t}(r,t) &= -\frac {2} {\theta} B(r,t) + 2F_0\,F'_0\,\int_{\Gamma} \sigma^2(r'-r)A(r',t-\tau(r'-r))\,dr'
\end{cases}
\]
 
Since the integral operators are convolutions on $\mathbbm{S}^1$, they are diagonalizable on the Fourier basis. Let us consider perturbations of the equilibrium of the form $A_{\nu,k}(r,t)=\Re(e^{\nu\,t+2\pi\,k\,r})$ with $\nu=\mathbf{i}\omega + l$, and leave $B(r,t)$ unspecified. We denote by $(a_k(\nu))$ and $(b_k(\nu))$ the Fourier coefficients of the functions $J(r)e^{-\nu\tau(r)}$ and $\sigma(r)e^{-\nu\tau(r)}$:
\[
\begin{cases}
	a_k(\nu) &= \int_{\Gamma} J(r')e^{-\nu\,\tau(r')}e^{-2\mathbf{i}\pi k\,r'}\,dr'\\
	b_k(\nu) &= \int_{\Gamma} \sigma^2(r')e^{-\nu\,\tau(r')}e^{-2\mathbf{i}\pi k\,r'}\,dr'
\end{cases}.
\]
The functions $A_{\nu,k}$ are eigenfunctions for the first equation the linearized system provided that $\nu$ satisfies the relationship:
\[\nu=\nu_k\eqdef -\frac 1 {\theta} +F_0' a_k,\]
defining the \emph{dispersion relationship} of the system. The characteristic roots of the linearized equations are given by the eigenvalues of the matrix:
\[\left(\begin{array}{ll}-\frac 1 \theta +F_0' a_k & 0\\ 2F_0F_0' b_k & -2/\theta \end{array}\right)\]
which are exactly $\{-2/\theta, -1/\theta + F_0' a_k, k\in \N\}$. In particular, this shows that no instability can occur on the standard deviation equation, and the whole stability of the homogeneous fixed point only depends on the Fourier coefficients $J(r)e^{-\tau(r)}$ and $F_0'$. Explicitly, in the original parameters, the spectrum is hence composed of the eigenvalues $\{-\frac{2}{\theta}, -\frac{1}{\theta} + a_k \; \frac{g}{\sqrt{2\pi(1+g^2v_0)}} , k\in \N\}$ with $v_0=\frac{\theta}{2} (\tilde{\sigma}F_0^2 +\Lambda^2)$. Let $a_M$ be the Fourier coefficient with largest real part\footnote{Since $J$ is integrable, this maximum necessarily exists because of Parcheval-Plancherel theorem.}. The solution $(0,v_0)$ is hence linearly stable as soon as 
\begin{equation}\label{eq:BifurcOneLayer}
	\Re\left(-\frac{1}{\theta}+a_M \frac{g}{\sqrt{2\pi (1+g^2v_0)}} \right)<0.
\end{equation}

An instability occurs when the characteristic roots such that $\nu$ has a positive real part. A Turing bifurcation point is defined by the fact that there exists an integer $k$ such that $\Re(\nu_k)=0$. It is said to be \emph{static} if at this point $\Im(\nu_k)=0$, and \emph{dynamic} if $\Im(\nu_k)=\omega_k\neq 0$. In that latter case, the instability is called Turing-Hopf bifurcation, and generates a global pattern with wavenumber $k$ moving coherently at speed $\omega_k/k$ as a periodic wavetrain. If the maximum of $\lambda_k$ is reached for $k=0$, a spatially homogeneous state is excited.

Formula~\eqref{eq:BifurcOneLayer} precisely quantifies the stabilization effect of the noise. Indeed, if $-1/\theta + \Re(a_M) g/\sqrt{2\pi}<0$, then all the eigenvalues are negative whatever $v_0$ and hence the solution $(0,v_0)$ is stable whatever the noise connectivity matrix $\sigma$ and the additive noise $\Lambda$. If now $-1/\theta + \Re(a_M) g/\sqrt{2\pi}>0$, then for $\Lambda$ and $\sigma$ small, the fixed point $(0,v_0)$ is unstable. When $\sigma$ or $\Lambda$ are increase, the fixed point will gain stability, since the maximal eigenvalue tends to $-1/\theta$ when $v_0$ goes to infinity.

In order to further identify the presence of Turing-Hopf instabilities, we choose an exponential connectivity function $J(r)=e^{-\vert r \vert/s}$ for some $s>0$, and $\tau(r) = \frac{\vert r \vert}{c} + \tau_d$. In that case, we have:
\[\frac{e^{-\nu\tau_d}(1-e^{-(\frac 1 s + \frac {\nu} c)})}{\frac 1 s + \frac {\nu} c+\mathbf{i}2\pi k }  \]
The related characteristic equation (or dispersion relationship ) reads:
\[\nu+\frac 1 {\theta} = F_0' \frac{e^{-\nu\tau_d}(1-e^{-(\frac 1 s + \frac {\nu} c)})}{\frac 1 s + \frac {\nu} c+\mathbf{i}2\pi k } \]
These equations are relatively complicated to solve analytically in that general form\footnote{A recent technique allows computing numerically the Hopf bifurcations lines and involves relatively complex formulation (see~\cite{veltz:11})}. However, when considering purely synaptic delay case (corresponding formally to $c=\infty$, i.e. disregarding the transport phenomenon), we can compute in closed form the curves of Turing-Hopf instabilities. In details, Turing instabilities arise when there exists $(\nu,k)$ satisfying the dispersion relationship and such that $\nu$ is purely imaginary: $\nu=\mathbf{i}\omega_k$. Equating modulus and argument in the dispersion relationship, we obtain the Turing instability curves. These only exist for parameters such that:
$\frac{F_0'^2 (1-e^{-1/s})^2}{1/s^2 + 4\pi^2k^2} \geq \frac 1 {\theta^2}$, i.e. when the noise levels are small enough so that:
\[\tilde{\sigma}^2\,F_0^2 +\Lambda^2 \leq \frac{\theta^2}{2\pi} \frac{(1-e^{-1/s})^2}{1/s^2+4\pi^2k^2}-\frac 1 {g^2}\eqdef \Lambda^*\]and in that case, we have:
\[\omega_k= \sqrt{\frac{F_0'^2 (1-e^{-\frac 1 s})^2}{\frac 1 {s^2} + 4\pi^2k^2} - \frac 1 {\theta^2}}.\]
An instability hence arises for parameters such that:
\[
	\tau_d=\frac {1}{\omega_k} \Big(-\arctan(\theta \omega_k) -\arctan(2\pi k s) +2\pi m\Big).
\]
for some $m\in\mathbbm{Z}$. 

In Figure~\ref{fig:DelaysBifs} we plotted the Turing instability curves for the mode $k=0$ and around the spatially homogeneous equilibrium $(\mu_0,v_0)$. We observe that Turing instability only occur when delays are non-null, and for noise parameters corresponding to the regime where the system presents $3$ spatially homogeneous equilibria. These bifurcations hence occur around the unstable spatially homogeneous solution, hence do not affect the stability of the fixed point (the system has a positive real eigenvalue). However, these instabilities produce transients regimes characterized by oscillations at the frequency $w_0$ close to the unstable spatially homogeneous solution. Numerical simulations of the neural field equations show that these oscillations progressively vanish and the system converges either towards spatially homogeneous or not depending on the initial condition. Moreover, since these instabilities arise on the mode $k=0$, oscillations in the neural field are synchronized. These instabilities are displayed in Fig.~\ref{fig:OscillationDelay} and ~\ref{fig:OscillationDelayTwo}. A movie of the solution (Supplementary Material), shows a surprising behavior of the solution that tends to stabilize around the unstable fixed point for relatively large periods of time. 

In order to exhibit non-spatially homogeneous solutions corresponding to wavenumbers strictly greater than $1$, we choose specific initial conditions as follows. We have seen that for small values of the noise parameters and large values of the slope $g$, the fully synchronized system presented two different stable equilibria that we denoted $(\mu_1,v_1)$ and $(\mu_2,v_2)$. If all initial conditions belong to the attraction basin of the same spatially homogeneous equilibrium, then the only mode to be excited corresponds to $k=0$, and the neural field stabilizes on a constant mode. If the initial condition belong to the attraction basin of the two different stable fixed points, higher modes are excited. As examples, setting the initial condition to:
\begin{equation}\label{eq:IC1}
	\mu(r,t)=\begin{cases}
		1 & r\in [0,0.25],\;t\in [-\tau,0]\\
		-1 & r\in [0.75,1],\;t\in [-\tau,0]\\
		0 & \text{otherwise}
	\end{cases}
\end{equation}
we excite a non-constant mode $k=1$ as illustrated in Fig.~\ref{fig:NonConstantMode}, and setting the initial condition to:
\begin{equation}\label{eq:IC2}
	\mu(r,t)=\begin{cases}
		1 & r\in [0,0.05] \cup [0.5,0.55],\;\;t\in [-\tau,0]\\
		-1 & r\in [0.25,0.3]\cup [0.75,0.8],\;t\in [-\tau,0]\\
		0 & \text{otherwise}
\end{cases}
\end{equation}
we excite the mode $k=2$, as illustrated in Fig.~\ref{fig:FourClusters}. These spatially periodic solutions persist when considering delays, and transient oscillations superimpose to this dynamics (Fig. ~\ref{fig:OscillationDelayTwo}). As noise is increased, this mode looses stability in favor of the mode $k=1$ first, before this mode looses again stability in favor of the spatially homogeneous solution $(\mu_0,v_0)$ as expected from the analysis of the stability of that fixed point. Higher modes prove relatively unstable, and illustrated in Fig.~\ref{fig:8clust} for an initial condition corresponding to a wavenumber $k=4$. After a short transient, the system stabilizes on a stationary solution corresponding to the mode $k=1$.

\begin{figure}
	\centering
		\subfigure[Bifurcations of the synchronized system]{\includegraphics[width=.3\textwidth]{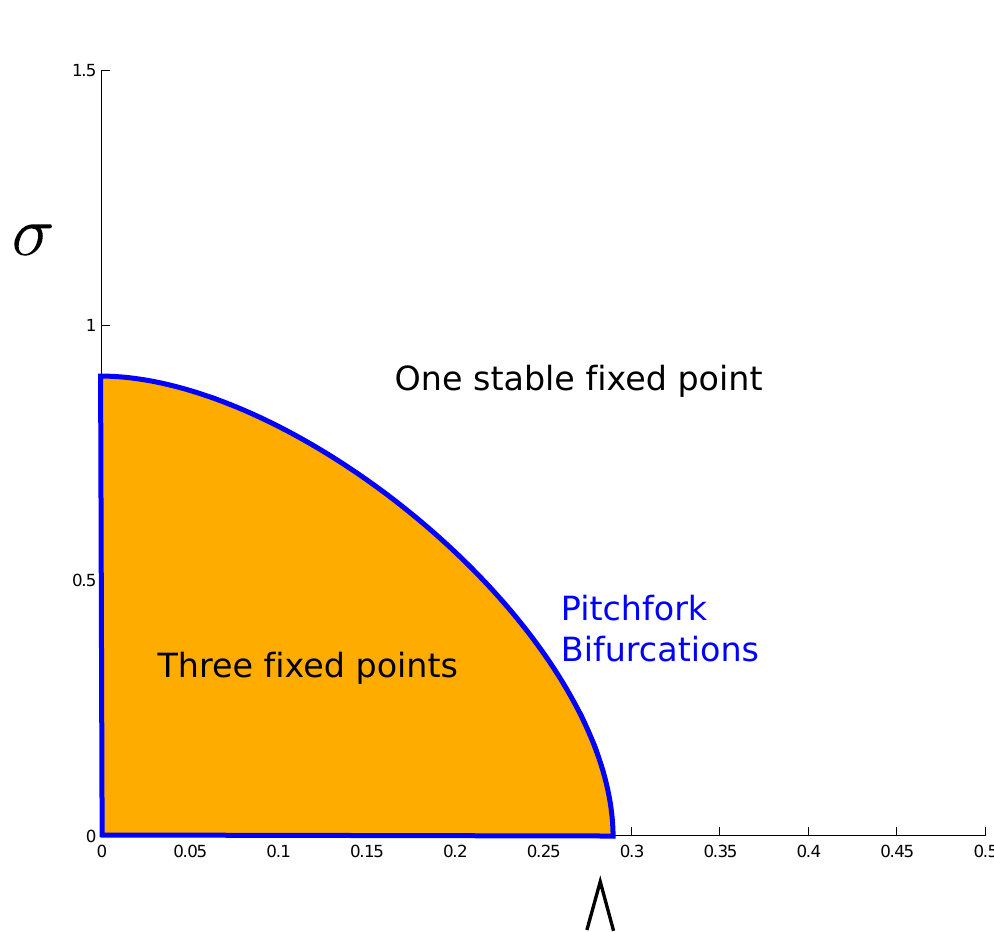}}\quad
		\subfigure[Bifurcations diagram : noise and delay]{\includegraphics[width=.3\textwidth]{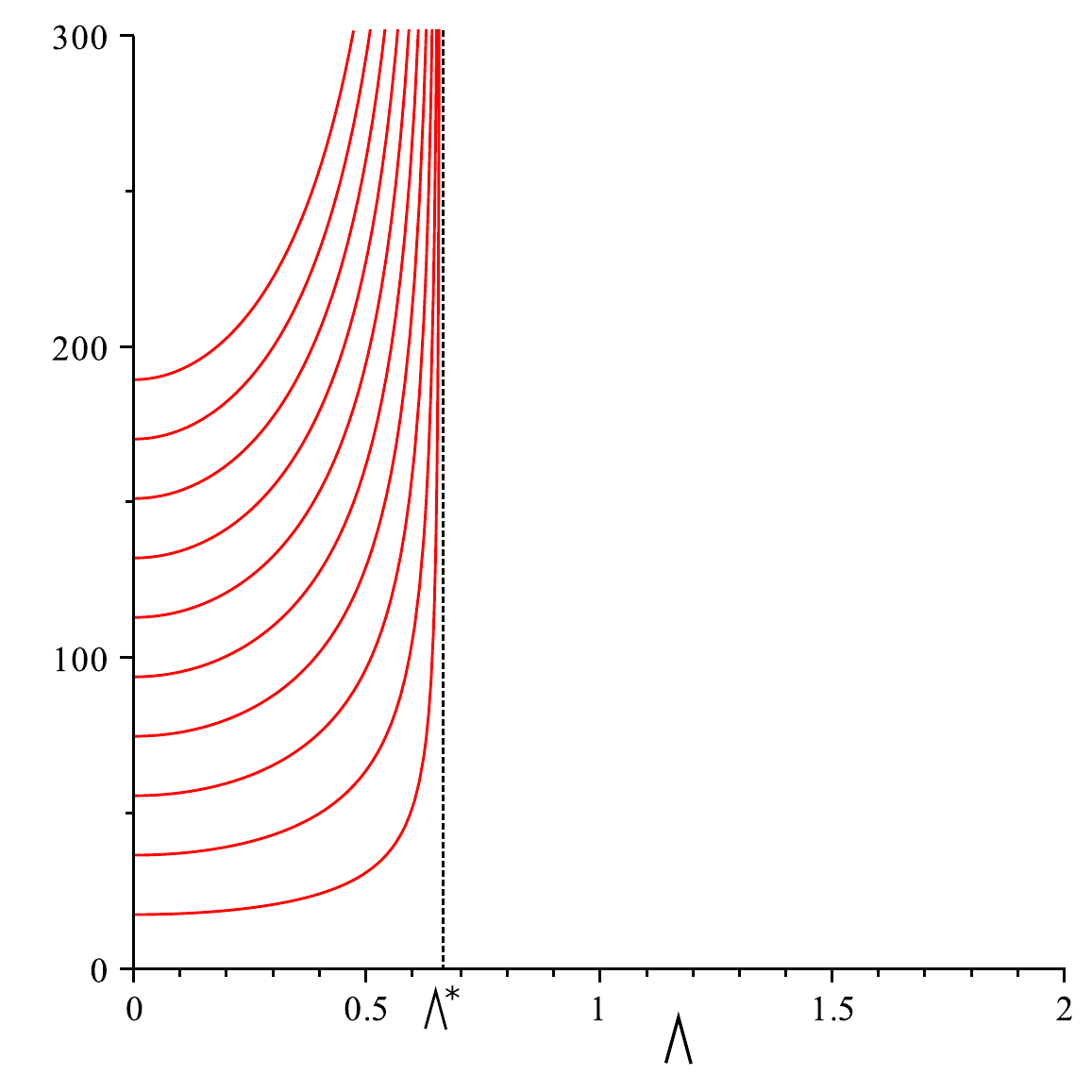}\label{fig:DelaysBifs}}\quad
		\subfigure[$\tau=20$, homogeneous initial conditions]{\includegraphics[width=.3\textwidth]{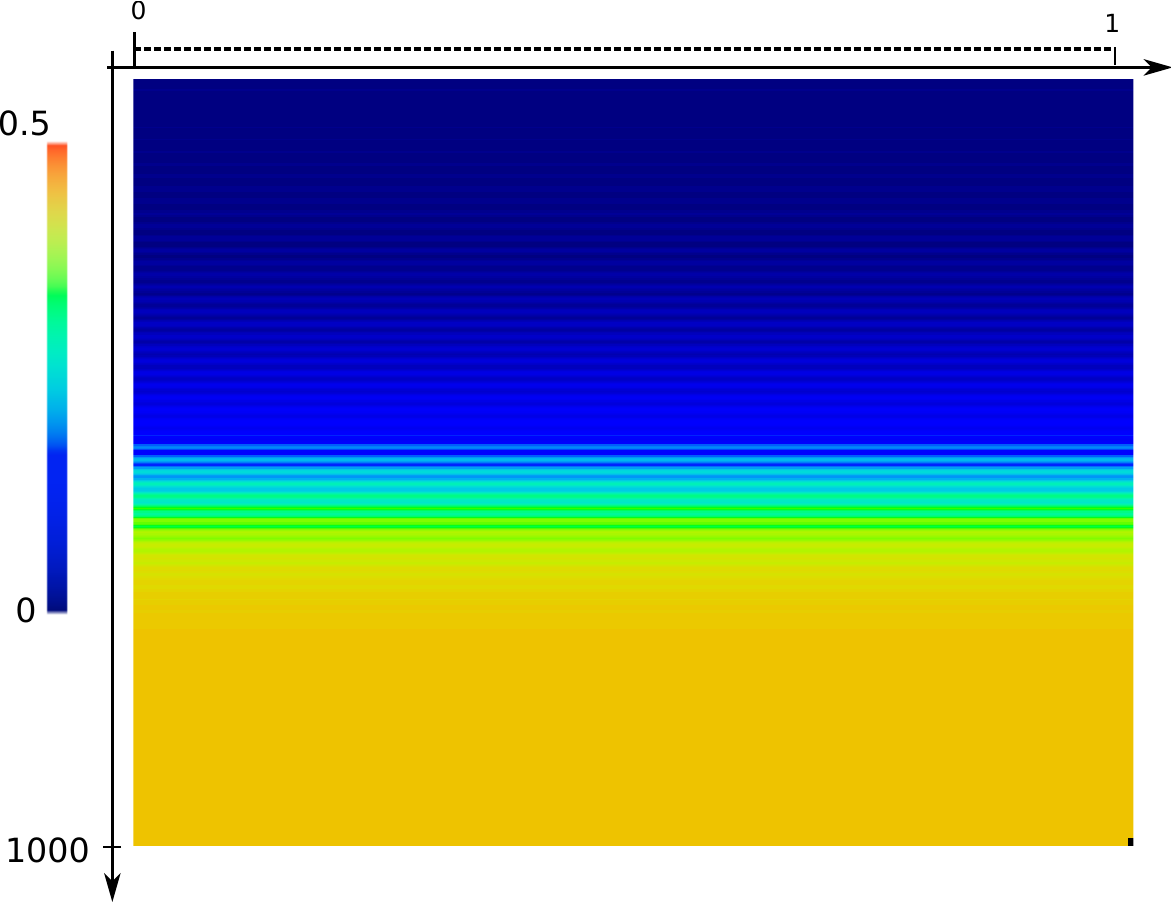}\label{fig:OscillationDelay}}\\
		\subfigure[$\tau=20$, IC1]{\includegraphics[width=.3\textwidth]{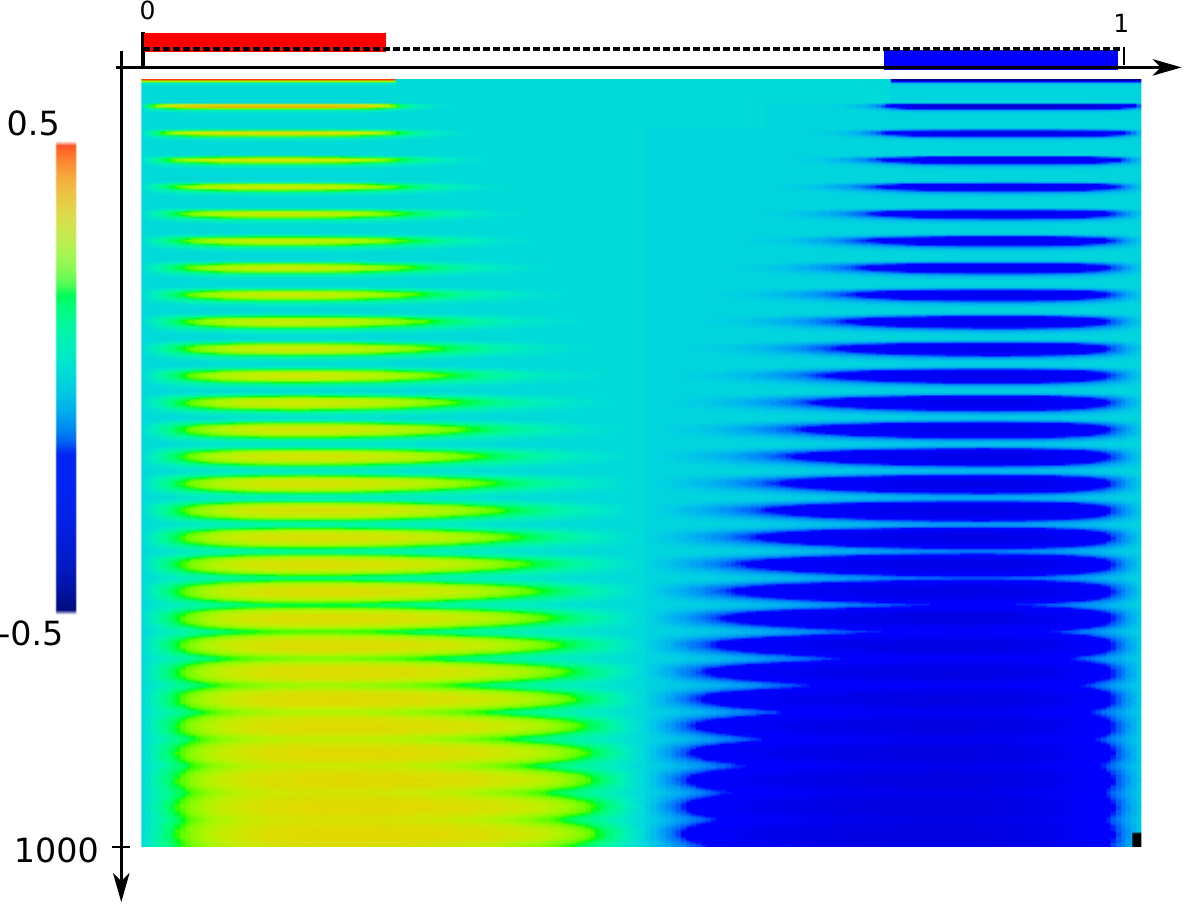}\label{fig:OscillationsDelaysNonHomo}\label{fig:OscillationDelayTwo}}\quad
		\subfigure[$\sigma\equiv 0$ and $\Lambda\equiv 0.1$, IC1]{\includegraphics[width=.3\textwidth]{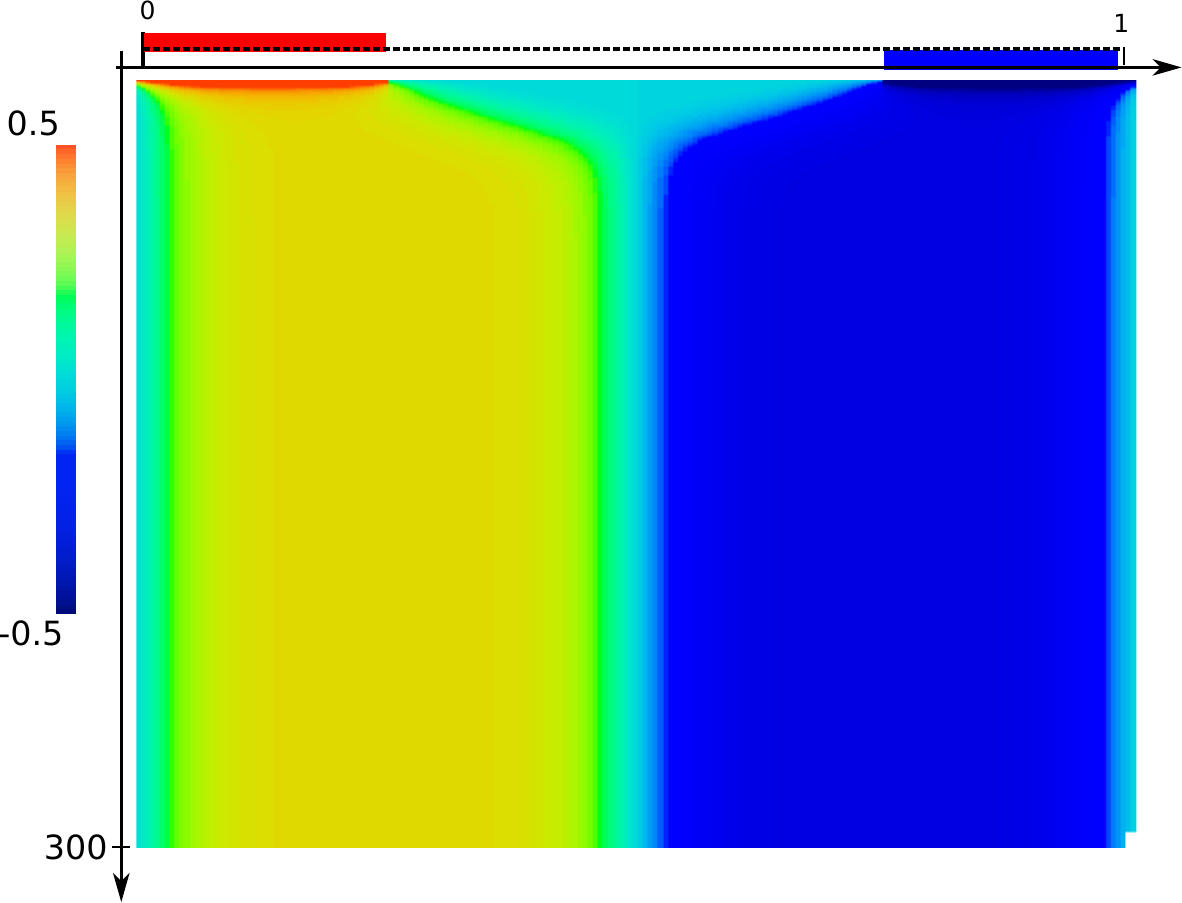}\label{fig:NonConstantMode}}\quad
		\subfigure[$\sigma\equiv 1$ and $\Lambda\equiv 0.1$, IC1]{\includegraphics[width=.3\textwidth]{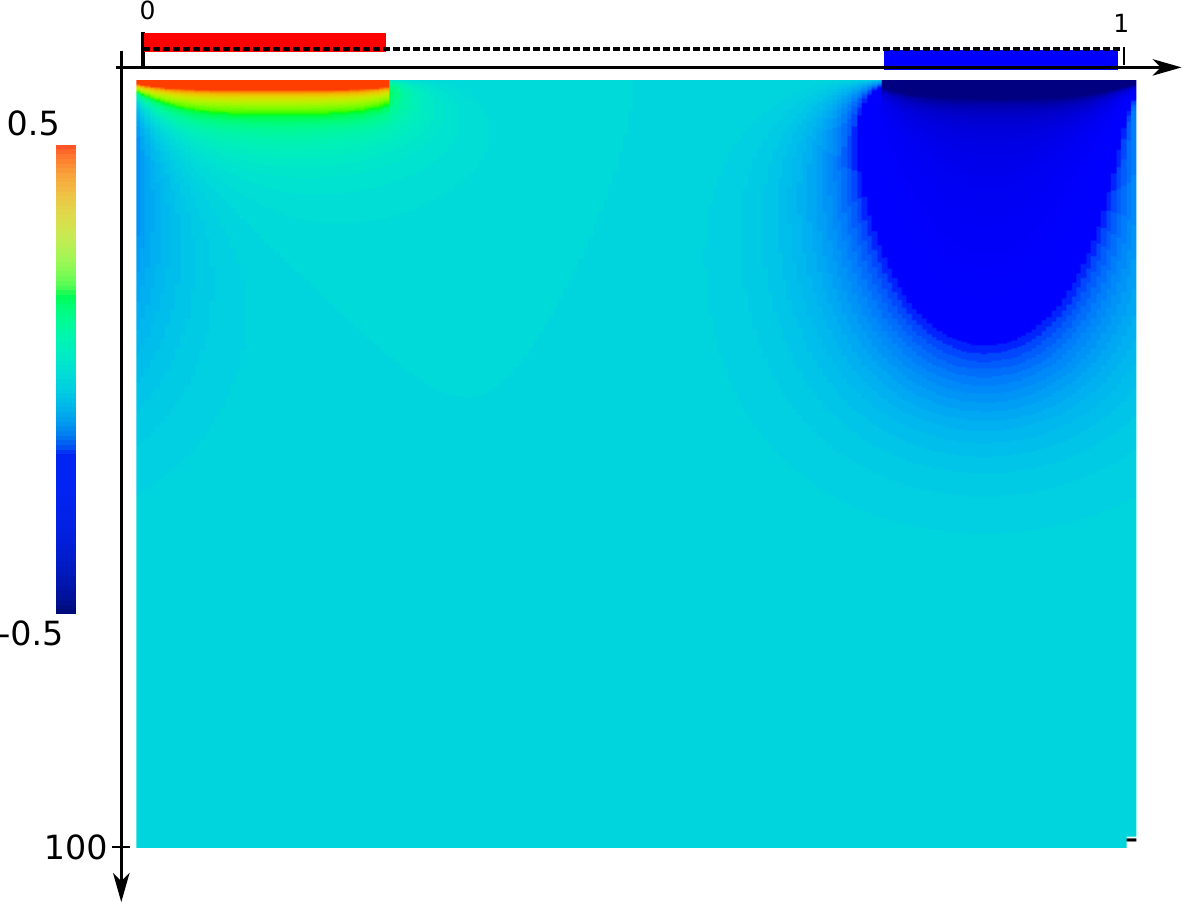}}\\
		\subfigure[$\sigma\equiv 0$ and $\Lambda\equiv 0.1$, IC2]{\includegraphics[width=.3\textwidth]{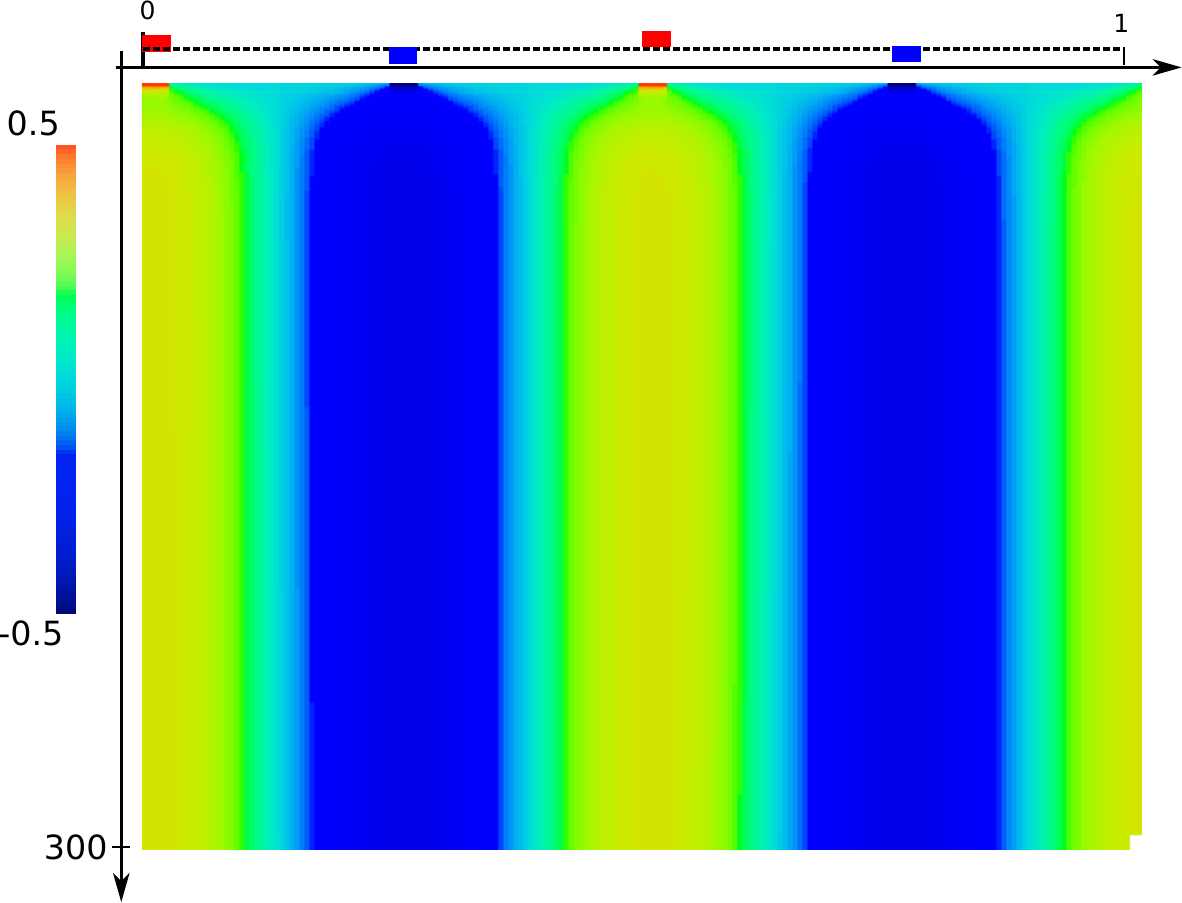}\label{fig:FourClusters}}
		\subfigure[$\sigma\equiv 0$ and $\Lambda\equiv 0.2$, IC2]{\includegraphics[width=.3\textwidth]{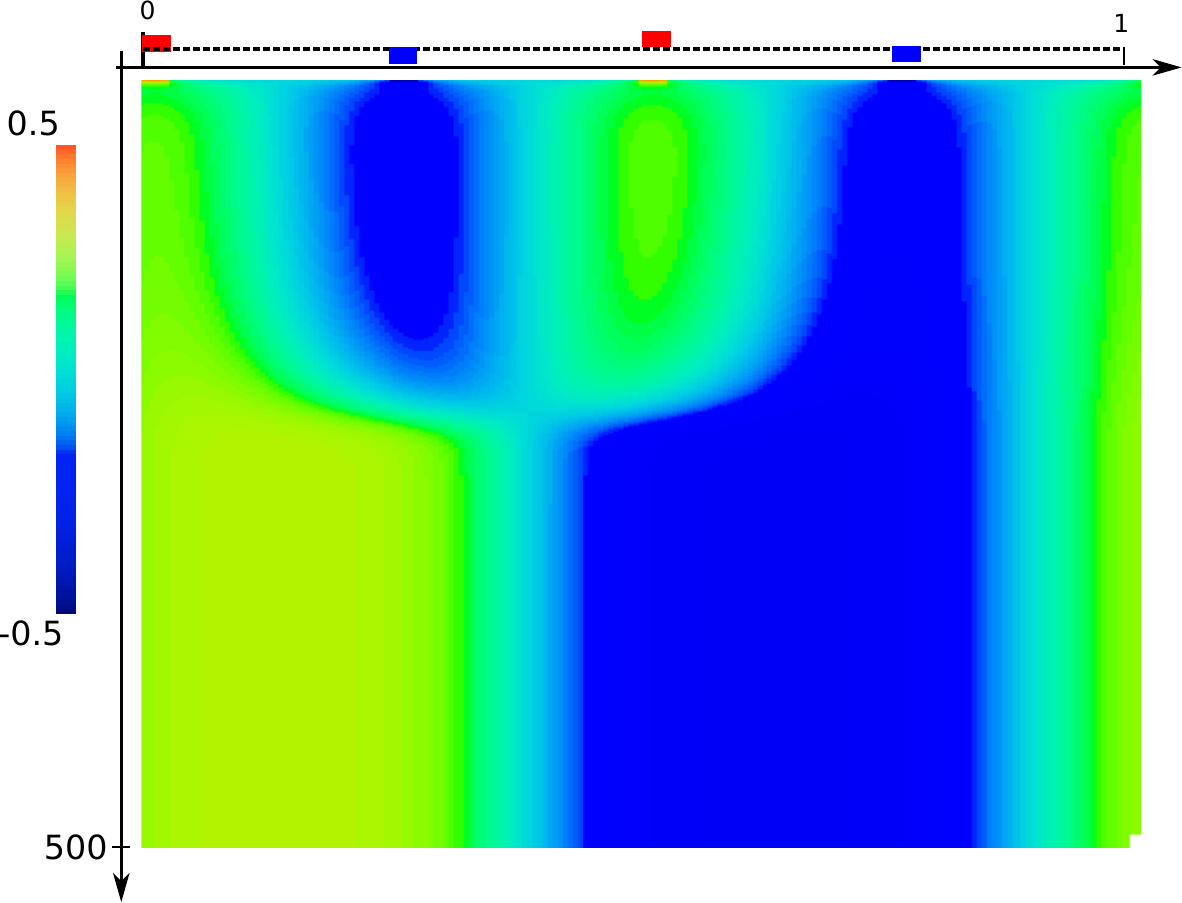}}
		\subfigure[$\sigma\equiv 0$ and $\Lambda\equiv 0.1$]{\includegraphics[width=.3\textwidth]{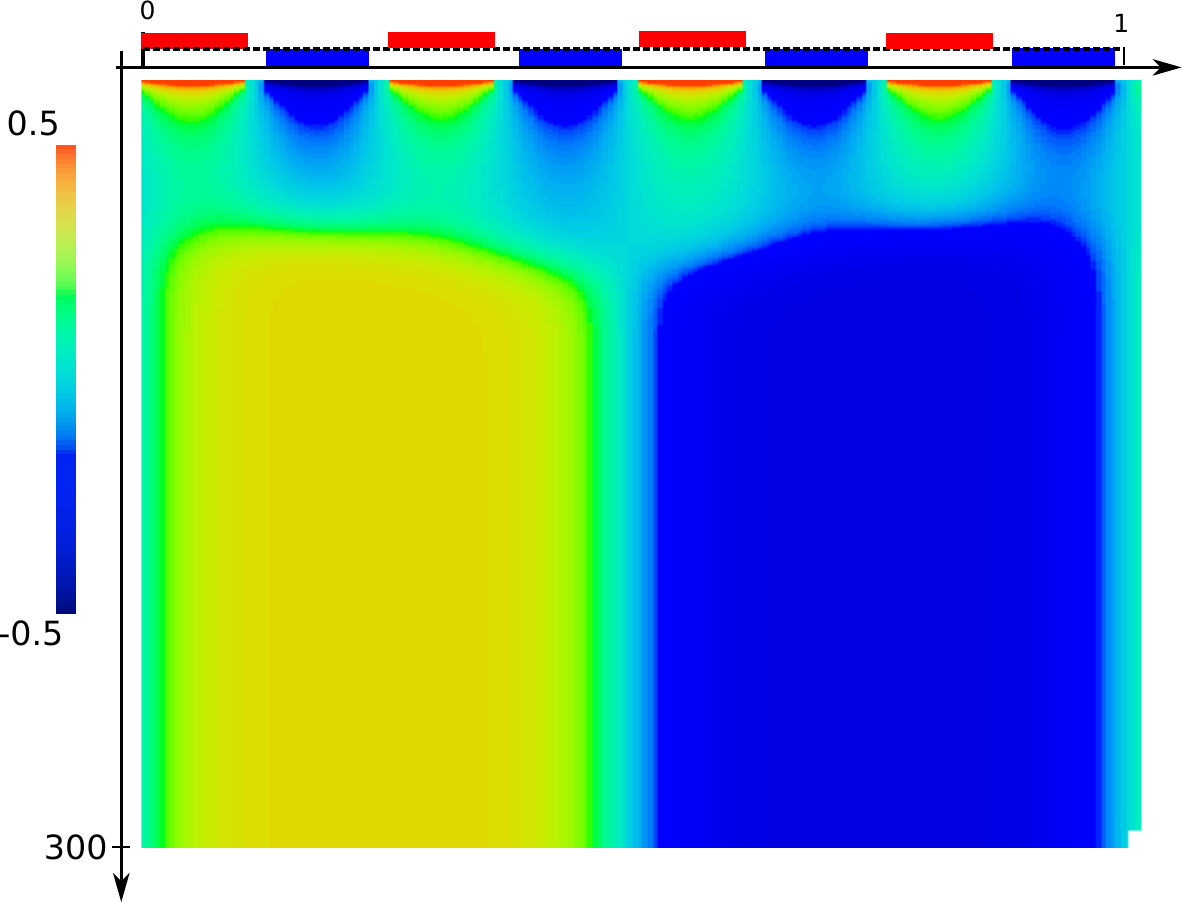}\label{fig:8clust}}
	\caption{Stabilization by noise of the spatially homogeneous state $(\mu_0,v_0)$. (a) shows the bifurcation diagram as a function of $\sigma$ and $\Lambda$ of the fully synchronized state, (b): Turing-Hopf codimension two bifurcations. (c) and (d) : $\tau=20$ illustrates the effect of delays producing transient oscillations, for homogeneous initial conditions (c) or IC1 (d). (e): no delay, initial conditions IC1 and small noise levels: non-spatially homogeneous state (mode $k=1$). (f): as noise is increased (here $\sigma$) the spatially homogeneous state $(\mu_0,v_0)$ is stabilized and attractive. (g): initial conditions IC2 shows mode $k=2$ excited, and as $\Lambda$ is increased (h), this state looses stability in favor of the mode $k=1$ before $(\mu_0,v_0)$ is stabilized. (i) shows the instability of the mode related to $k=4$. }
	\label{fig:OneLayer}
\end{figure}

\medskip
The simplified one-layer case chosen for simplicity of analytical exploration presents a limited set of spatio-temporal behaviors: only stationary solutions are found, either spatially homogeneous or characterized by a few typical modes. In particular, no wave or oscillatory activity was observed. In order to go beyond these phenomena, we now turn to the study of two layers neural fields.

\subsection{Dynamic Turing Patterns and noise-induced synchronization in a two-layers system}\label{sec:TwoLayer}

We consider in this section a more biologically relevant system composed of two-layers: an excitatory (modeling pyramidal neurons) and an inhibitory layer (modeling interneurons). This set up can be easily cast in our framework. Following the same lines, we can show that the propagation of chaos property applies and that the network equations uniformly converge towards an integro-differential McKean Vlasov equation. Firing-rate models have Gaussian solutions, whose mean and standard deviation will satisfy a coupled system of delayed integro-differential equations.

In details, let us label $1$ the excitatory and $2$ the inhibitory layer. Interconnections from the inhibitory or excitatory layers are exponentially-shaped: $J_{a}(r,r') = K_a\,e^{-\vert r-r'\vert/s_a}$ for $a=1$ or $2$, with $K_a^{-1}=\mathcal{J}_a=\int_{0}^1J_{a}(r,r')dr'$. The choice of the typical spatial extension of the kernel $s_a$ is a modeling issue. When considering $\Gamma$ as a functional space (e.g. orientation columns), excitation is local and inhibition is more distal, which motivates a choice $s_1>s_2$. If $\Gamma$ models the anatomical location of each neuron, inhibitory connections are globally characterized by shorter axons than excitatory axons, corresponding to $s_2>s_1$. 

We define the type function of a neuron (excitatory or inhibitory) by the function $\nu(i)\in\{1,2\}$. The connectivity kernels $J_{\nu}$ are multiplied by a typical connectivity coefficient between the different populations, $w_{\nu\nu'}$ for the deterministic interactions.
Similarly, the noise interaction kernels depend on the type of connection and are denoted $\sigma_{\nu,\nu'}(r,r')$ and are modeled by $\sigma_{\nu,\nu'}(r,r')=\sigma J_{\nu'}(r,r')$. The typical time constants $\theta_{\nu}(r)$ of all neurons are again considered constant and chosen as our time unit (i.e. $\theta=1$). In a network composed of $N_{\gamma\nu}$ neurons of type $\nu$ in the population located at $r_{\gamma}\in\Gamma$, the equation of neuron $i$ of type $\nu(i)=a\in\{1,2\}$, in population $\alpha$ at location $r_{\alpha}$ reads:
\begin{multline*}
\displaystyle{dV^{i}(t) = \Bigg( -V^{i}(t) + I_{a}(r_{\alpha}, t) + \sum_{\gamma=1}^{P(N)} \sum_{\nu=1}^2 \frac{1}{N_{\gamma\nu}} \sum_{j,\,p(j)=\gamma, \nu(j)=\nu} w_{a\nu} J_{\nu}({r_{\alpha},r_{\gamma}}) \; S(r_{\gamma},V^j(t-\tau(r_{\alpha},r_{\gamma}))) \Bigg) \, dt} \\
\displaystyle{+ \Bigg(\sum_{\gamma=1}^{P(N)} \sum_{\nu=1}^2 \frac{1}{N_{\gamma\nu}} \sum_{j,\,p(j)=\gamma, \nu(j)=\nu} \sigma J_{\nu}({r_{\alpha},r_{\gamma}}) \; S(r_{\gamma},V^j(t-\tau(r_{\alpha},r_{\gamma}))) \Bigg) \, dB^{\alpha,a,\gamma,\nu}_t+\Lambda(r_{\alpha},t) dW^{i}_t}.
\end{multline*}
The Gaussian attractive solutions have a mean and standard deviation that satisfy the delayed integro-differential equations:
\begin{equation}\label{eq:DDEIntegroDiff2Pops}
		\begin{cases}
			\derpart{\mu_1}{t}(r,t) & =-\mu_1(r,t) + \int_{\Gamma} \Big \{w_{11} J_{1}(r,r') f(r,\mu_1(r',t-\tau(r,r')),v_1(r',t-\tau(r,r')))\\
			 &\quad +w_{12}J_{2}(r,r') f(r,\mu_2(r',t-\tau(r,r')),v_2(r',t-\tau(r,r')))\Big\}\lambda(r')dr'+I_1(r,t)\\
			\derpart{\mu_2}{t}(r,t) &=-\mu_2(r,t) + \int_{\Gamma} \Big\{w_{21}J_{1}(r,r') f(r,\mu_1(r',t-\tau(r,r')),v_1(r',t-\tau(r,r')))\\
			& \quad +w_{22}J_{2}(r,r') f(r,\mu_2(r',t-\tau(r,r')),v_2(r',t-\tau(r,r')))\Big\}\lambda(r')dr' +I_2(r,t)\\
			\derpart{v_1}{t}(r,t) & =- 2\,v_1(r,t) + \sigma^2\int_{\Gamma} \Big\{ J_{1}^2(r,r') f^2(r,\mu_1(r',t-\tau(r,r')),v_1(r',t-\tau(r,r')))\\
			 &\quad +J_{2}^2(r,r') f^2(r,\mu_2(r',t-\tau(r,r')),v_2(r',t-\tau(r,r')))\Big\}\lambda(r')^2 dr'+ \Lambda_1^2(r,t)\\
			\derpart{v_2}{t}(r,t) & =-2 \,v_2(r,t) + \sigma^2\int_{\Gamma}  \Big\{ J_{1}^2(r,r') f^2(r,\mu_1(r',t-\tau(r,r')),v_1(r',t-\tau(r,r')))\\
			 &\quad +J_{2}^2(r,r') f^2(r,\mu_2(r',t-\tau(r,r')),v_2(r',t-\tau(r,r')))\Big\}\lambda(r')^2 dr'+\Lambda_2^2(r,t)
		\end{cases}
	\end{equation}
\begin{remark}
	Since the different layers are driven by independent Brownian motions, the covariance between the excitatory and inhibitory population is null. This property is similar to the second point of the remark after theorem~\ref{pro:GaussianSpace}. 
\end{remark}
Similarly to the single-layer case, we consider $\Gamma=\mathbbm{S}^1$ in our numerical and analytical work. Two other types of connectivity will be dealt with in~\ref{append:BoundaryConditions}: (i)\emph{reflective} boundary conditions in which the solution is virtually evenly continued at the boundaries $0$ and $1$ and the convolution is done on $\R$ instead of $[0,1]$ and (ii) \emph{zero} boundary conditions where the convolution only occurs on $[0,1]$ (which would correspond to a convolution on $\R$ virtually considering the activity null on $\R\setminus[0,1]$). 

The study of these equations is more complex than in the previous case. In order to present analytical results, we will first analyze a particular neural field (Network I) allowing analytical investigations of the solutions and accounting for complex phenomena driven by noise in our mean-field equations. This network is characterized by the connectivity matrix:
\begin{equation*}
	w=\left(\begin{array}{ll}
		1 & -1\\
	 	1 & 1
	\end{array}\right),
\end{equation*}
$\sigma=0$, $\Lambda_1=\Lambda_2=\Lambda$ and input currents $I_1=0$ and $I_2=-1$. The main interest of this example relies in the fact that we can characterize one spatially-homogeneous fixed point of the system, $\mu_i=0$, $v_i=\Lambda^2/2$ for $i\in\{1,2\}$. This fact will be extremely useful to understand how noise interferes with the stability of this fixed point.

The second network we will consider is closer from biological networks, and inspired by the parameters proposed by Wilson and Cowan in their seminal article~\cite{wilson-cowan:72}. The synaptic weights are chosen to be: 
\begin{equation}\label{eq:J}
	J=\left(\begin{array}{ll}
		15 & -12\\
		16 & -5
	\end{array}\right),
\end{equation}
and the input currents $I_1=0$, $I_2=-3$. In this network, we do not have any trivial fixed point, can only rely on numerical bifurcation analysis. It is important to note that (i) the methodology developed here is totally independent of the synaptic weights chosen and (ii) the phenomena exhibited in the sequel are relatively robust and remain qualitatively valid in a relatively large range of values around this matrix (a wide condition is given in~\cite{wilson-cowan:72}). We will see that most of the phenomena observed in the analytical study persist in this second case. A very interesting and surprising phenomenon absent in Network I appears in the second network, corresponding to noise-induced oscillations, and the transition between stationary and oscillatory solutions will be characterized.

\subsubsection{Network (I): Analytical developments}

Spatially homogeneous solutions for Network (I) satisfy the equation:
\[
\begin{cases}
	\dot{\mu}_1 & =-\mu_1+f(\mu_1(t-\tau),v_1(t-\tau))-f(\mu_2(t-\tau),v_2(t-\tau))+I_1\\
	\dot{\mu}_2 & = -\mu_2+f(\mu_1(t-\tau),v_1(t-\tau))+f(\mu_2(t-\tau),v_2(t-\tau))+I_2\\
	\dot{v}_1 &=-2 v_1 +\Lambda^2\\
	\dot{v}_2 &=-2 v_2 +\Lambda^2\\
\end{cases}\]
The variance equations are not coupled to the mean equations, and converge towards $\Lambda^2/2$. As stated, a trivial stationary spatially homogeneous solution is given by $\mu_1=\mu_2=0$, $v_1=v_2=\Lambda^2/2$. The stability of this solution point is governed by the properties of the characteristic matrix governing the linear stability for the means $(\mu_1,\mu_2)$:
\[A(\zeta)=-(\zeta+1) Id + \frac{g}{\sqrt{2\pi(1+g^2\Lambda^2/2)}} J e^{-\zeta \tau},\]
whose eigenvalues (the \emph{characteristic roots}) are:
\[\nu_{\pm} = -(\zeta+1)+\frac{g}{\sqrt{2\pi(1+g^2\Lambda^2/2)}} e^{-\zeta\tau} (1\pm \mathbf{i})\]
and the \emph{characteristic equation} $\Delta(\zeta)\eqdef det(A(\zeta))=0$. Solutions of this equations correspond to cases where at least one of the characteristic roots vanishes, i.e. to values of $\zeta$ such that:
\begin{equation}\label{eq:CharacteristicRootNecessary}
-(\zeta+1)+\frac{g}{\sqrt{2\pi(1+g^2\Lambda^2/2)}} e^{-\zeta\tau} (1\pm \mathbf{i}) =0
\end{equation}
This equation can be solved using the complex branches of Lambert's $(W_k)_{k\in\mathbbm{Z}}$ functions (see e.g.~\cite{corless:96}):
\begin{equation}\label{eq:CharacteristicRoot}
	\zeta_{\pm}^k=-1+\frac 1 {\tau} W_{k}\left(\frac{g}{\sqrt{2\pi(1+g^2\Lambda^2/2)}} \tau e^{\tau}(1\pm \mathbf{i})\right).	
\end{equation}
 The stability of the trivial solution considered, governed by the sign of the real part of the uppermost eigenvalue, is given by the real branch $W_0$ of Lambert function, and if the argument has a real part greater than $-e^{-1}$ the root is unique. If this is not the case, two eigenvalues have the same real part (corresponding to $k=0$ or $k=-1$). 

We observe that the argument of the Lambert function in the expression~\eqref{eq:CharacteristicRoot} has a modulus that decreases towards $0$ as $\Lambda$ or $\tau$ go to infinity. For fixed values of $\tau$, the rightmost eigenvalue given by $k=0$, decreases towards $-1$ as $\Lambda$ is increased (see Fig.~\ref{fig:LambertFunction}), and there exists a value $\Lambda_c(\tau)$ such that for any $\Lambda>\Lambda_c(\tau)$ the fixed point $0$ is stable: again, noise has a stabilizing effect on this equilibrium. This stabilization appears through a Hopf bifurcation, and periodic behaviors are found for $\Lambda<\Lambda_c(\tau)$, as shown in the bifurcation diagram~\ref{fig:BifDiagZero} for a fixed value of the delays, $\tau=0.5$. For fixed values of $\Lambda$, as $\tau$ is increased, the real parts of the eigenvalues increase and might switch from positive to negative (see figure~\ref{fig:LambertFunction}). 

In order to quantitatively identify Turing-Hopf bifurcations, we use the same technique as in section~\ref{sec:OneLayer} following~\cite{hale-lunel:93,shayer-campbell:00}. Using the fact that necessarily, Hopf bifurcations correspond to purely imaginary characteristic roots $\zeta=\mathbf{i}\omega$, we obtain:
\begin{equation}\label{eq:HopfBifurcation}
	-(\mathbf{i}\omega +1) =-\frac{g}{\sqrt{2\pi(1+g^2\Lambda^2/2)}} e^{-\mathbf{i}\omega\tau} (1\pm \mathbf{i})
\end{equation}
which, taking the squared modulus of these imaginary numbers, give the equality:
\begin{equation}\label{eq:Hopfomega}
	\omega^2 = \frac{g^2}{\pi({1+g^2\Lambda^2/2})}-1
\end{equation}
The positivity of this quantity implies that:
\[\Lambda^2 \leq (\Lambda^*)^2\eqdef 2\left (\frac 1 \pi -\frac 1 {g^2}\right ).\]
This property indicates that necessarily, for noise intensities greater than $\Lambda^*$, the fixed point $(0,\Lambda^2/2)$ is stable, again pointing towards a stabilization effect of noise. Moreover, since eigenvalues have increasing real parts as delays is increased, this condition implies the existence of a vertical asymptote, which is indeed observed in the numerical computation of the characteristic roots, Figs.~\ref{fig:HopfsAnalytic} and~\ref{fig:HopfsAnalyticAccu}. 

When $\Lambda<\Lambda^*$, equating the argument of both sides of equality~\eqref{eq:HopfBifurcation}, we get:
\[\tau = \frac{-\arctan(\omega)\pm\frac{\pi}{4} + 2\,k \pi}{\omega}\]
for $k\in\mathbbm{Z}$. This relationship can be written in closed form as a function of the parameters using the expression of $\omega$ obtained in equation~\eqref{eq:Hopfomega}. The different curves of Hopf bifurcations are plotted in Figure~\ref{fig:Analytic}. We observe a cascade of Hopf bifurcation accumulation at $\Lambda=\Lambda^*$ as delays are increased. For large delays, irregular transient behaviors will arise, corresponding to the very complex landscape of the phase plane, as displayed in Fig.~\ref{fig:WeirdTransient}. We chose for instance to display the transient solution for $\tau=5$, a case where the delays are small enough so that we can resolve the presence of different limit cycles trapping the solution transiently. This case hence makes explicit the dependence on noise levels of qualitative behaviors of the system for finite populations networks. 

\begin{figure}
	\begin{center}
		\subfigure[$x\mapsto \Re(W_0(x (1+\mathbf{i})))$ ]{\includegraphics[width=.2\textwidth]{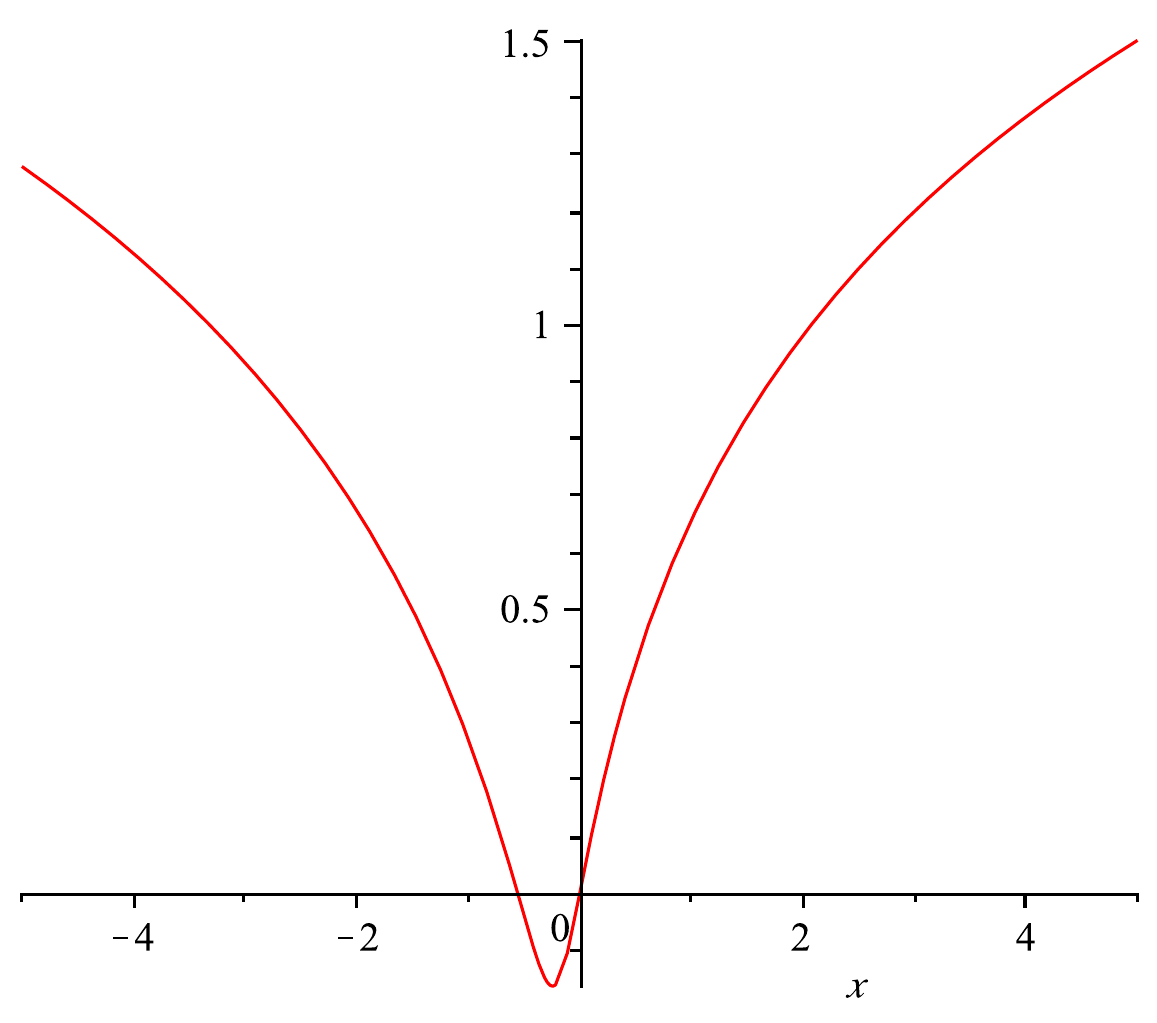}\label{fig:LambertFunction}}\quad
		\subfigure[Hopf bifurcations]{\includegraphics[width=.2\textwidth]{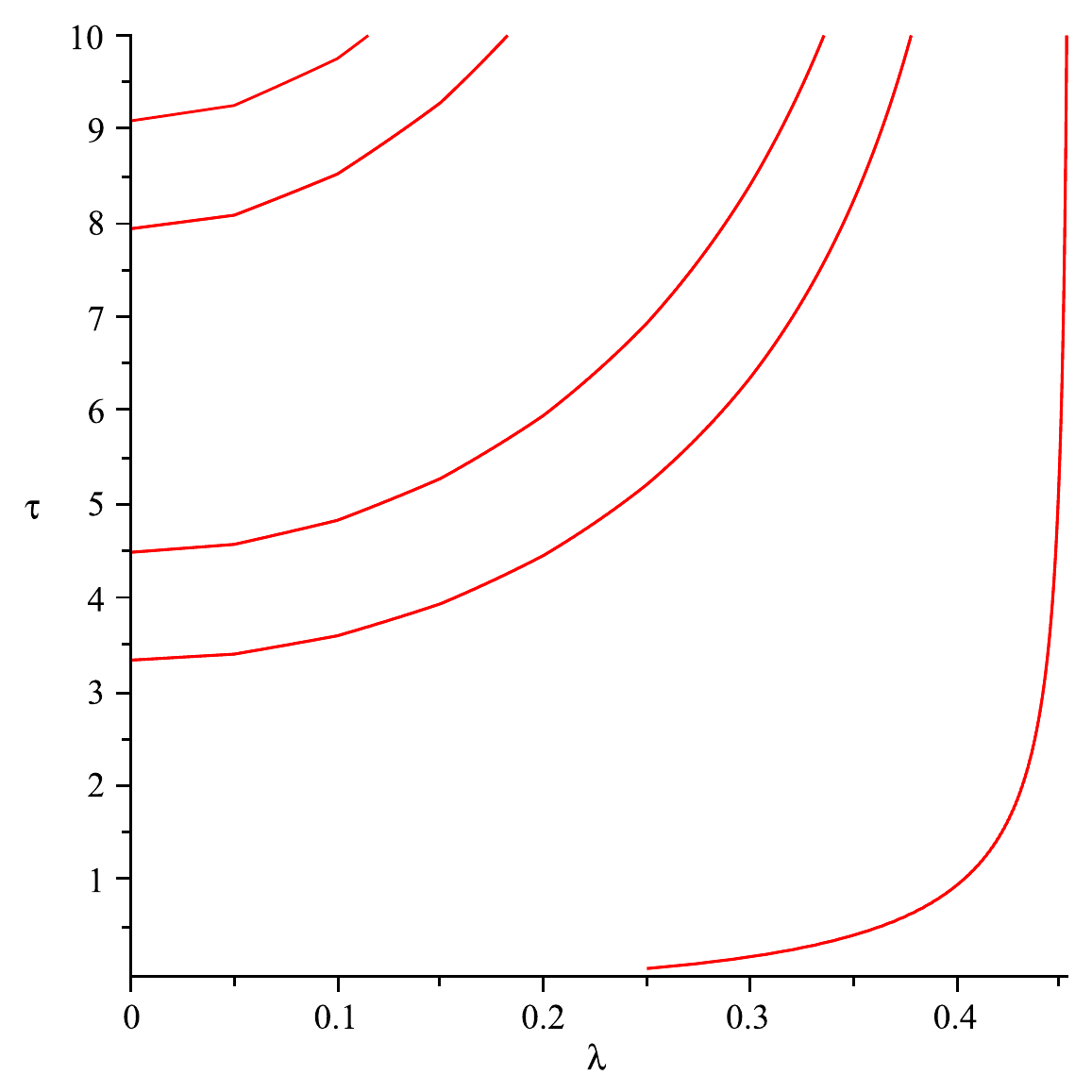}\label{fig:HopfsAnalytic.pdf}}\quad
		\subfigure[Accumulation]{\includegraphics[width=.2\textwidth]{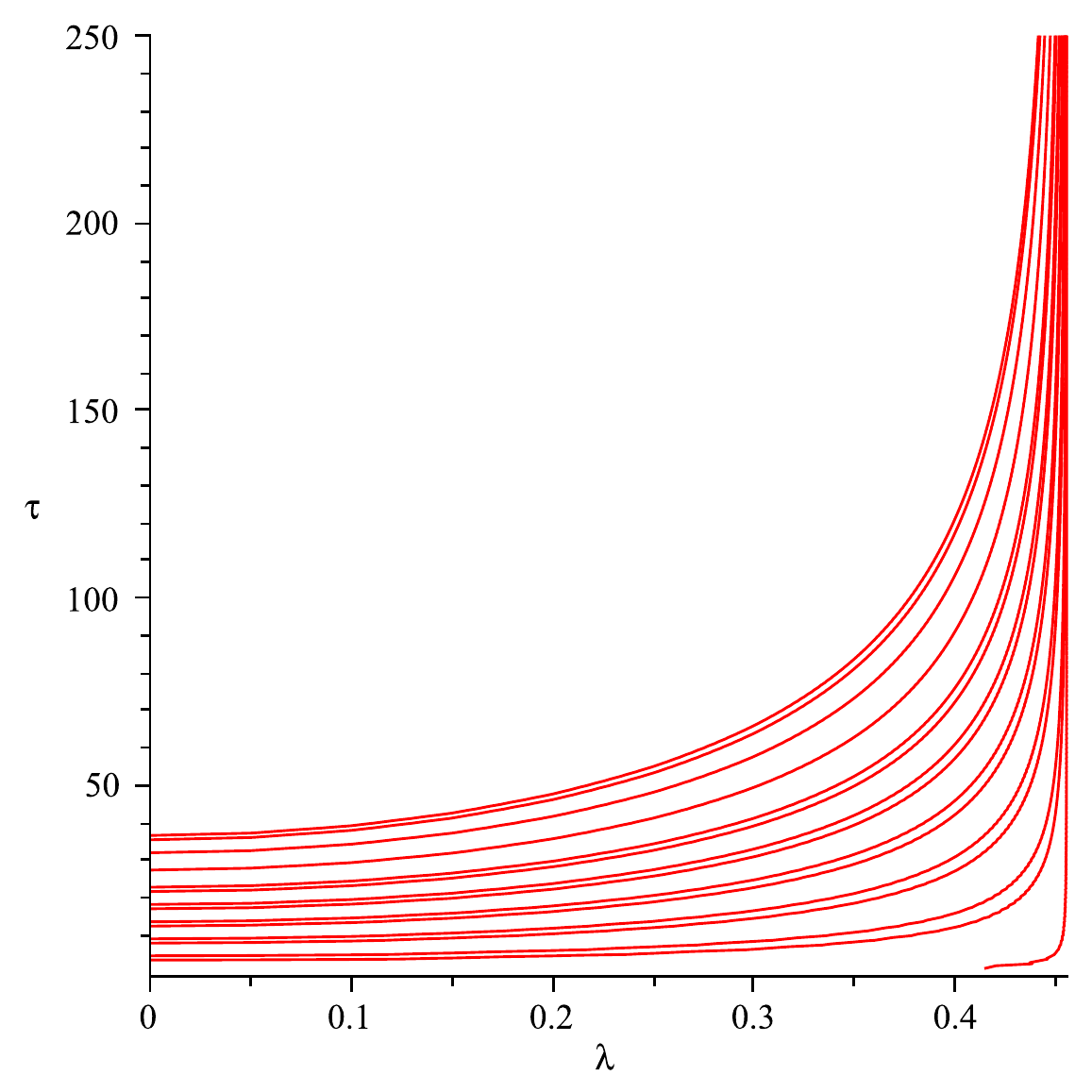}\label{fig:HopfsAnalyticAccu}}\quad
		\subfigure[Chaotic transient, $\tau=5$]{\includegraphics[width=.2\textwidth]{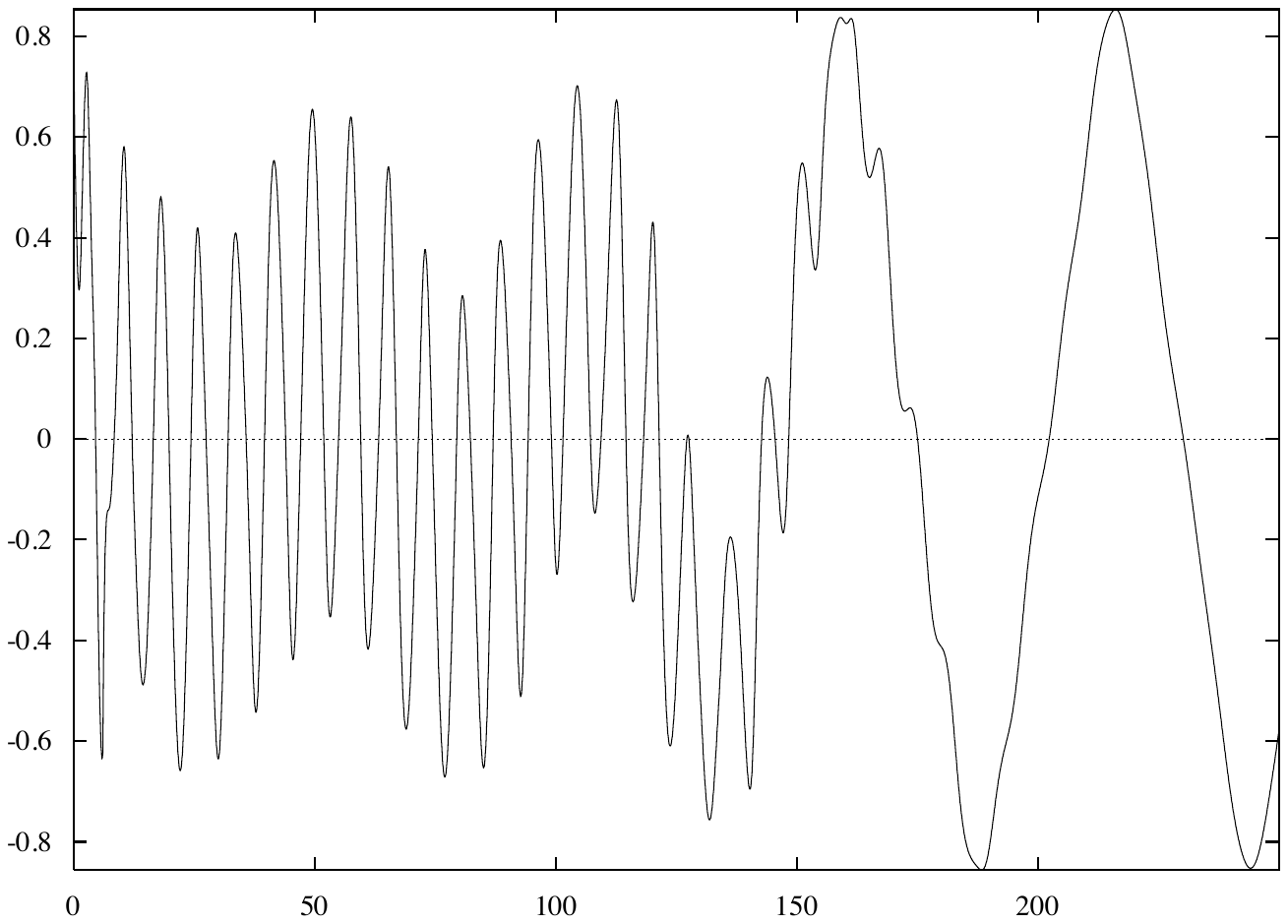}\label{fig:WeirdTransient}}\\
		\subfigure[$\tau=0.5$: stabilization by noise]{\includegraphics[width=.3\textwidth]{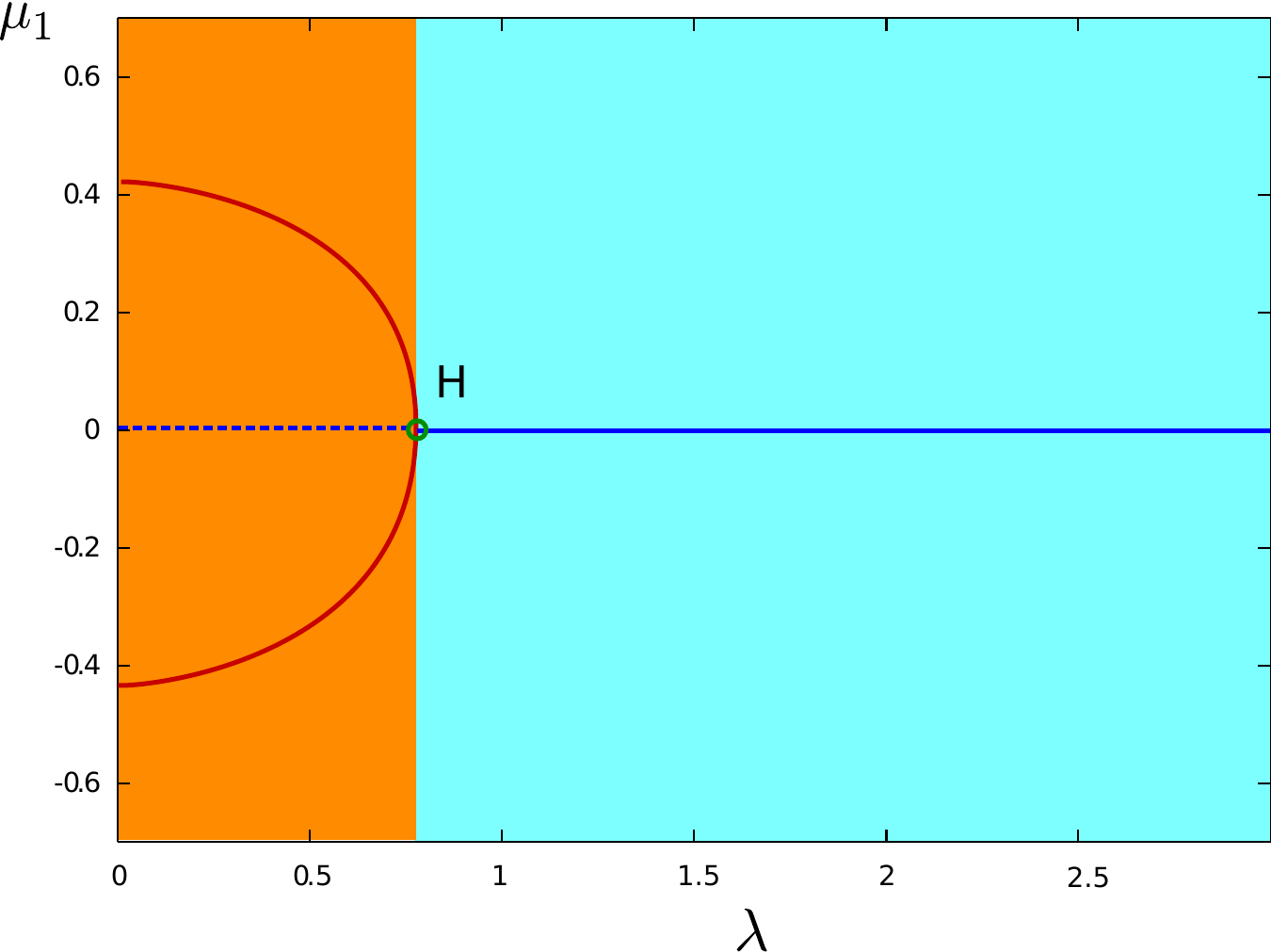}\label{fig:BifDiagZero}} \quad
		\subfigure[$\tau=0.5$, $\Lambda=0.5$: synchronized oscillations]{\includegraphics[width=.3\textwidth]{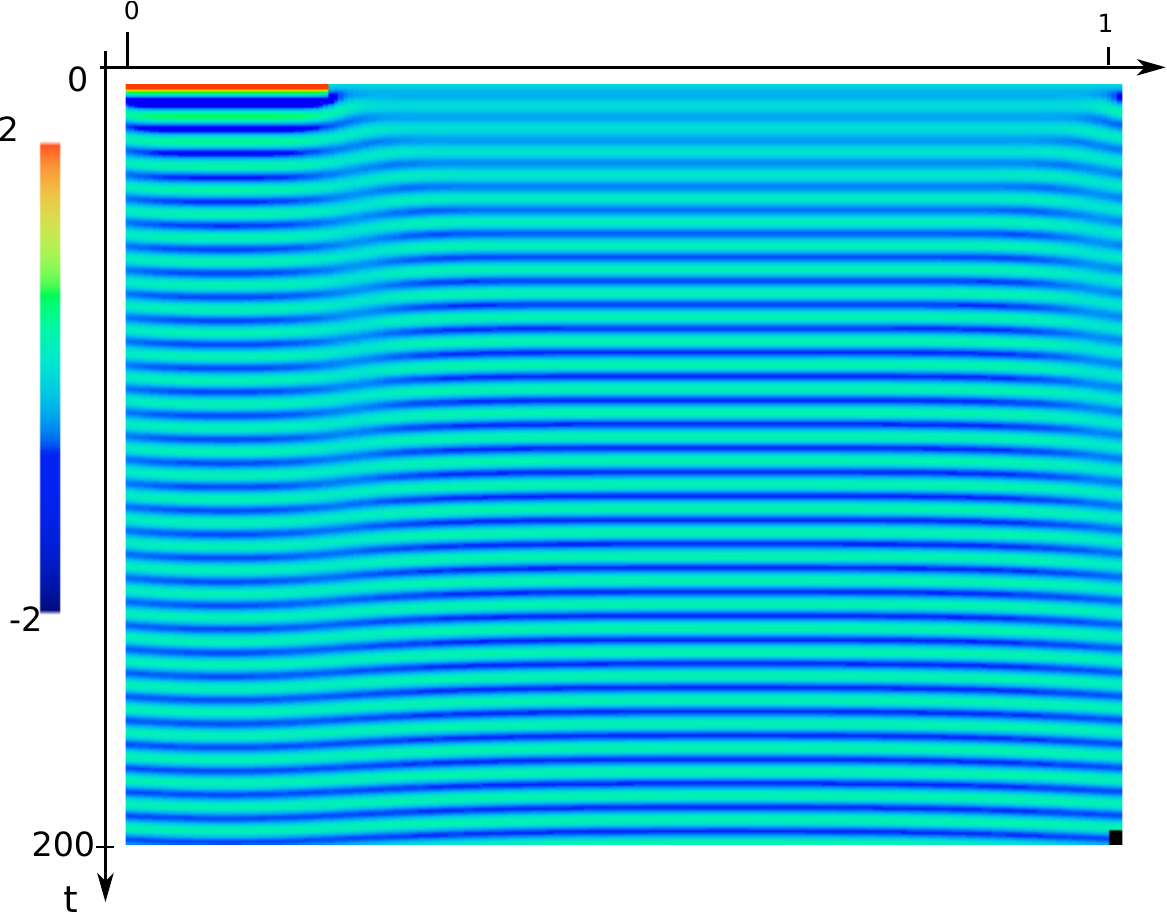}\label{fig:SpaceTimeNet1.pdf}} \quad
		\subfigure[$\tau=0.8$, $\Lambda=0.7$: stationary solution]{\includegraphics[width=.3\textwidth]{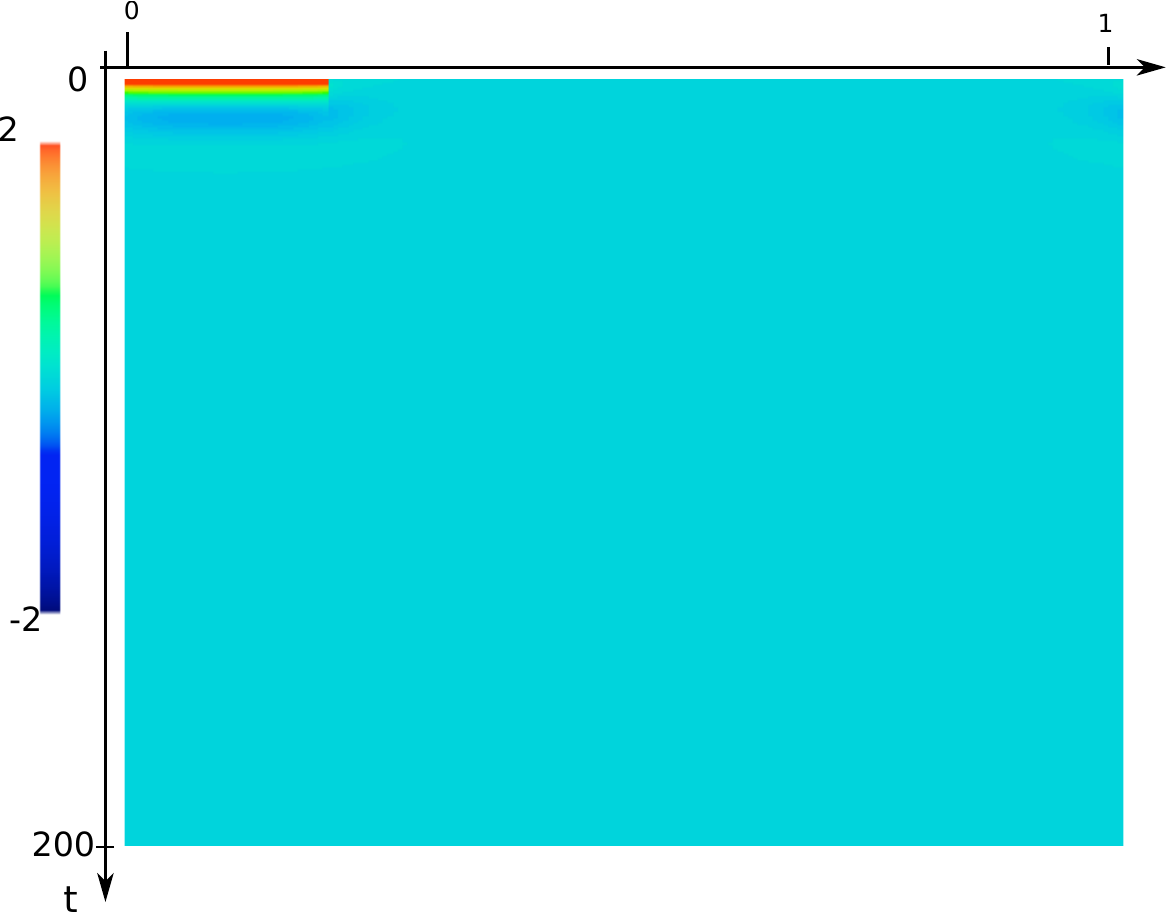}\label{fig:SpaceTimeNet1Sync.pdf}} \\
	\end{center}
	\caption{Dynamics of Network I: characteristic roots around the trivial fixed point. As delays are increased, several Hopf bifurcations arise and accumulate around the same value. (a): Shape of the Lambert function $x\mapsto \Re(W_0(x(1+\mathbf{i})))$, (b): cascade of Hopf bifurcations and (c): locus of the Hopf bifurcations for the 15 first characteristic root continuated for large values of $\tau$ (d): Transient regime for $\tau=5$. (e): Bifurcation diagram of the spatially homogeneous solutions as a function of $\Lambda$ for $\tau=0.5$. (f) and (g) illustrate the fact that spatially homogeneous solutions are recovered for non-spatially homogeneous initial conditions.  }
	\label{fig:Analytic}
\end{figure}

Simulations of the spatially extended networks with delays do not present any stable non-spatially homogeneous state: after a short transient phase, the spatially homogeneous state described by the above analysis takes over. A very similar analysis as the one performed in the one-layer case is here again possible.However, no non-spatially homogeneous solutions are found, and we always obtain the spatially homogeneous solution as permanent regime (see Figures~\ref{fig:SpaceTimeNet1} and~\ref{fig:SpaceTimeNet1Sync} representing $\mu_1(r,t)$ as a function of space (abscissa) and time (ordinate)).

This example illustrates the fact that noise can destroy oscillations: an originally oscillating state disappears as noise increases. We now turn to the study of Network (II), in which case noise will have a surprising structuring effect on the solution through the creation of regular oscillations.

\subsubsection{Network II: Noise-induced oscillations, wave and bump chaotic splitting.}\label{sec:periodicBounds}
In the case of network II, since we keep considering convolutional interactions and periodic domain, the system has spatially homogeneous solutions (Proposition~\ref{pro:Synchro}). We start by analyzing the nature of these solutions before numerically analyzing non-spatially homogeneous solutions. The parameters related to Network II do not allow computing in closed form possible spatially homogeneous equilibria, hence we will rely on numerical bifurcation analysis.

The spatially homogeneous solutions are given by equations~\eqref{eq:DDEIntegroDiff}. The bifurcation diagram of this system as a function of the noise intensity $\Lambda$ and the delay $\tau$ is given in Figure~\ref{fig:BifDiagsDDE}. It presents two branches of fixed points, and as delays are increased, one of the branches of fixed points undergoes a cascade of Hopf bifurcations as observed analytically in the case of Network I. The different curves of Hopf bifurcations accumulate on a vertical asymptote in a parameter region where the related fixed point is unstable. Similar to the case of Network I, these Hopf bifurcations arise on an unstable fixed point and are not related to rightmost eigenvalues, and hence have no effect on the number or stability of fixed points. Moreover, these are subcritical, hence associated with unstable limit cycles, the accumulation of which produces a complex landscape resulting in very irregular transient behaviors. Fixing $\tau=5$ corresponds to a case where the system displays seven Hopf bifurcations. The codimension two bifurcation diagram of the system as a function of the input current $I_1$ and the noise level $\Lambda$ is given in Figure~\ref{fig:BTDegenerate}. We observe that the saddle-node bifurcation forms a cusp, and on one of the branch of the codimension two saddle-node bifurcation curve appears a Bogdanov-Takens and two degenerate Bodganov-Takens bifurcations. When fixing a value for the delays to $\tau=0.1$ (green line of figure~\ref{fig:Codim2}), we observe in this diagram different ranges of parameter values corresponding to different asymptotic behaviors: stationary solution (blue region), bistability between a stationary and a periodic solution (yellow region), and periodic solutions (orange region, see Fig.~\ref{fig:Cod2Noise}). Oscillations in the mean-field equations correspond to phase-locked oscillations of individual neurons since they all have the same probability distribution (see Fig.~\ref{fig:Behaves15}).

Besides the stabilization by noise already discussed, this bifurcation diagram identifies a very surprising effect of noise, shaping the qualitative activity: as noise is increased, stationary solutions give place to synchronized oscillations for intermediate values of noise (see Figure~\ref{fig:Behaves15}), and as noise is further increased, these synchronized oscillations disappear in favor of another stationary behavior. This is a very counter-intuitive phenomenon, as noise generally tends to alter fine structures of the solutions. 

This phenomenon is displayed in the diagrams~\ref{fig:BifDiagsDDE}(e)-(g). The codimension two bifurcation diagram~\ref{fig:Cod2Noise} presents a Hopf, a saddle-node and a saddle-homoclinic bifurcation curves, separating the diagram into three qualitatively distinct zones: region (A) where the system features one stable and two unstable fixed points, separated by the saddle-homoclinic bifurcation curve from a bistable zone (B, yellow) where the moment equations presents an additional stable periodic orbit, (C) where the system has a periodic orbit and an unstable fixed point and (D) where the system has a unique stable fixed point. Zone (B) is separated from (C) by a saddle-node bifurcation manifold and (C) separated from (D) by the Hopf bifurcation manifold. We observe that $\sigma$ and $\Lambda$ have qualitatively the same effect on the dynamics. As examples are plotted codimension 1 bifurcation diagrams for $\Lambda=0.1$ as a function of $\sigma$ and for $\sigma=0.1$ as a function of $\Lambda$ (black lines in~\ref{fig:Cod2Noise}). 

\begin{figure}
	\begin{center}
		\subfigure[Codimension 2 bifurcations in $(\Lambda,\tau)$]{\includegraphics[width=.33\textwidth]{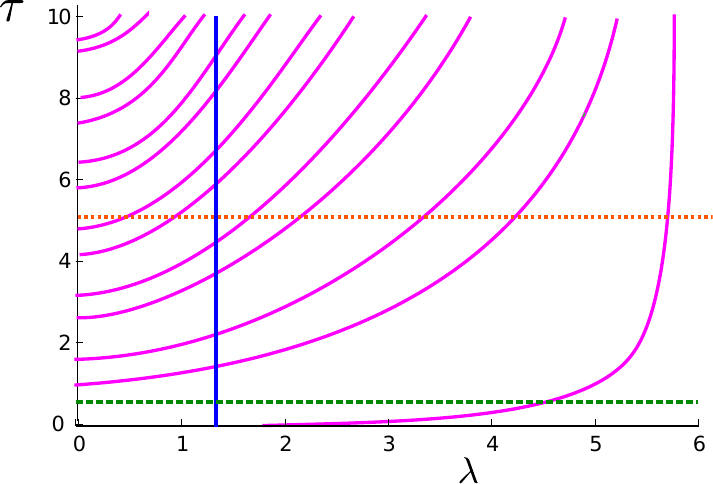}\label{fig:Codim2}}\quad 
		\subfigure[3 first Hopf bifurcations, large delays]{\includegraphics[width=.30\textwidth]{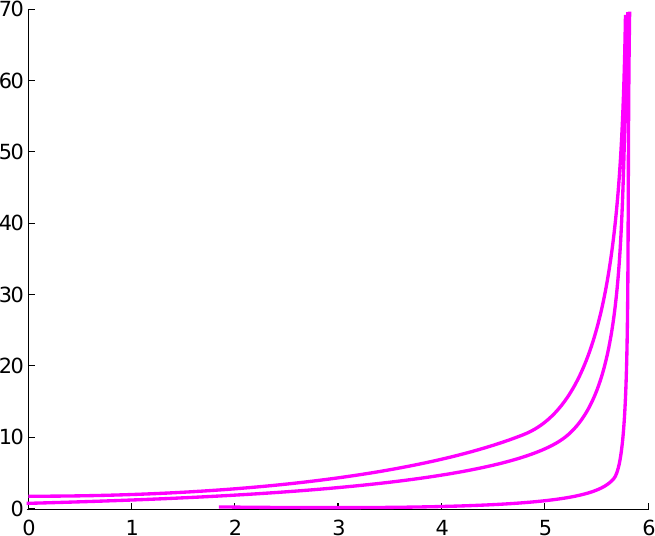}\label{fig:Accumulation}}\quad
		\subfigure[Codimension 2 bifurcations in $(I_1,\Lambda)$,  $\tau=5$]{\includegraphics[width=.30\textwidth]{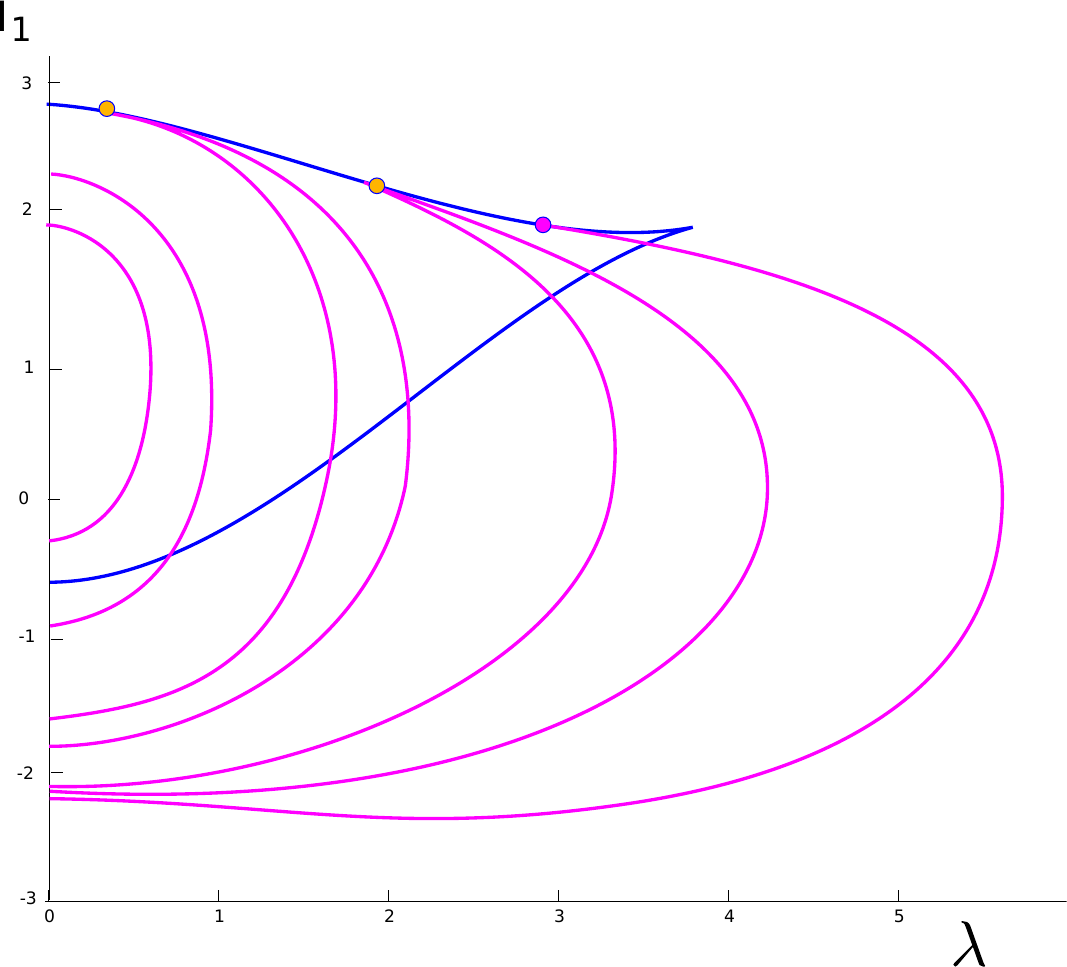}\label{fig:BTDegenerate.pdf}}\\
		\subfigure[solution, $t\mapsto \mu_1(t)$ for $\tau=5$, $\Lambda=0.1$]{\includegraphics[width=.30\textwidth]{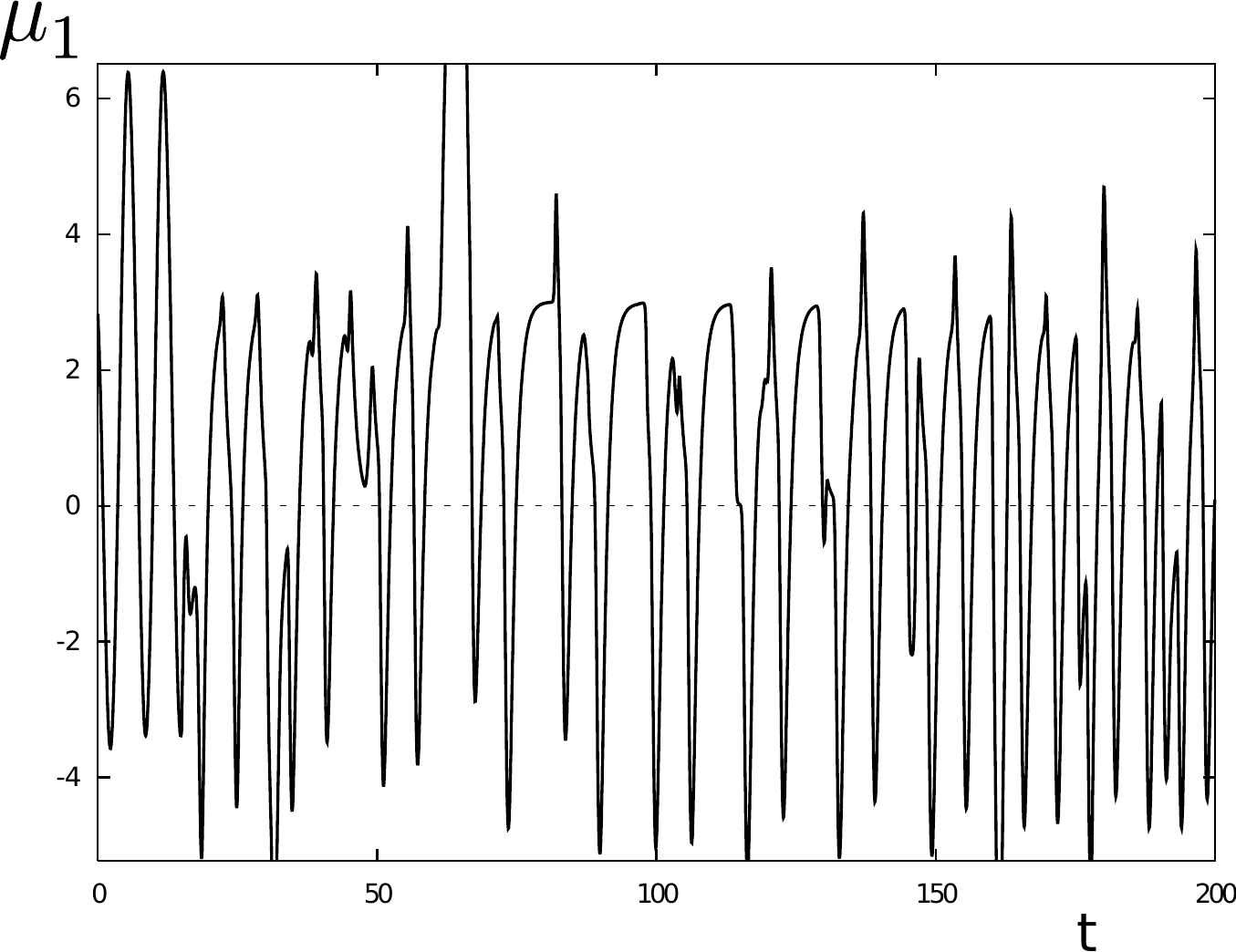}\label{fig:Chaos}}\qquad 
		\subfigure[Codimension 2 diagram in $(\Lambda,\sigma)$]{\includegraphics[width=.35\textwidth]{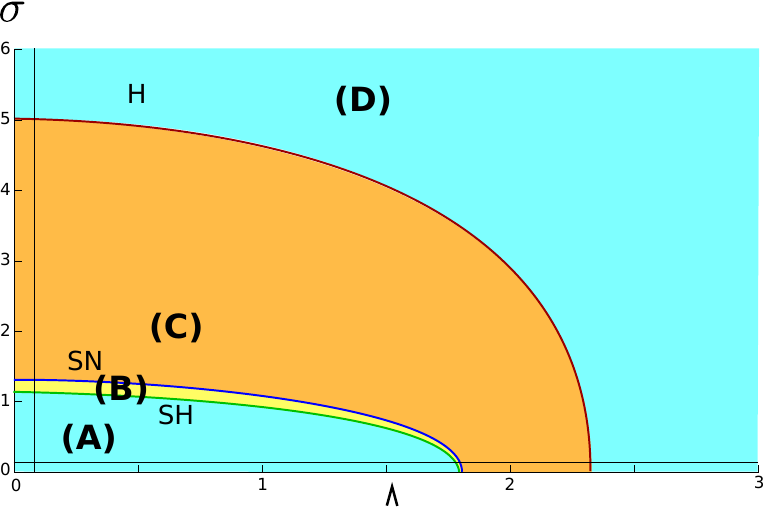}\label{fig:Cod2Noise}}\\
		\subfigure[Bifurcation diagram in $\Lambda$ with $\sigma=0.1$]{\includegraphics[width=.3\textwidth]{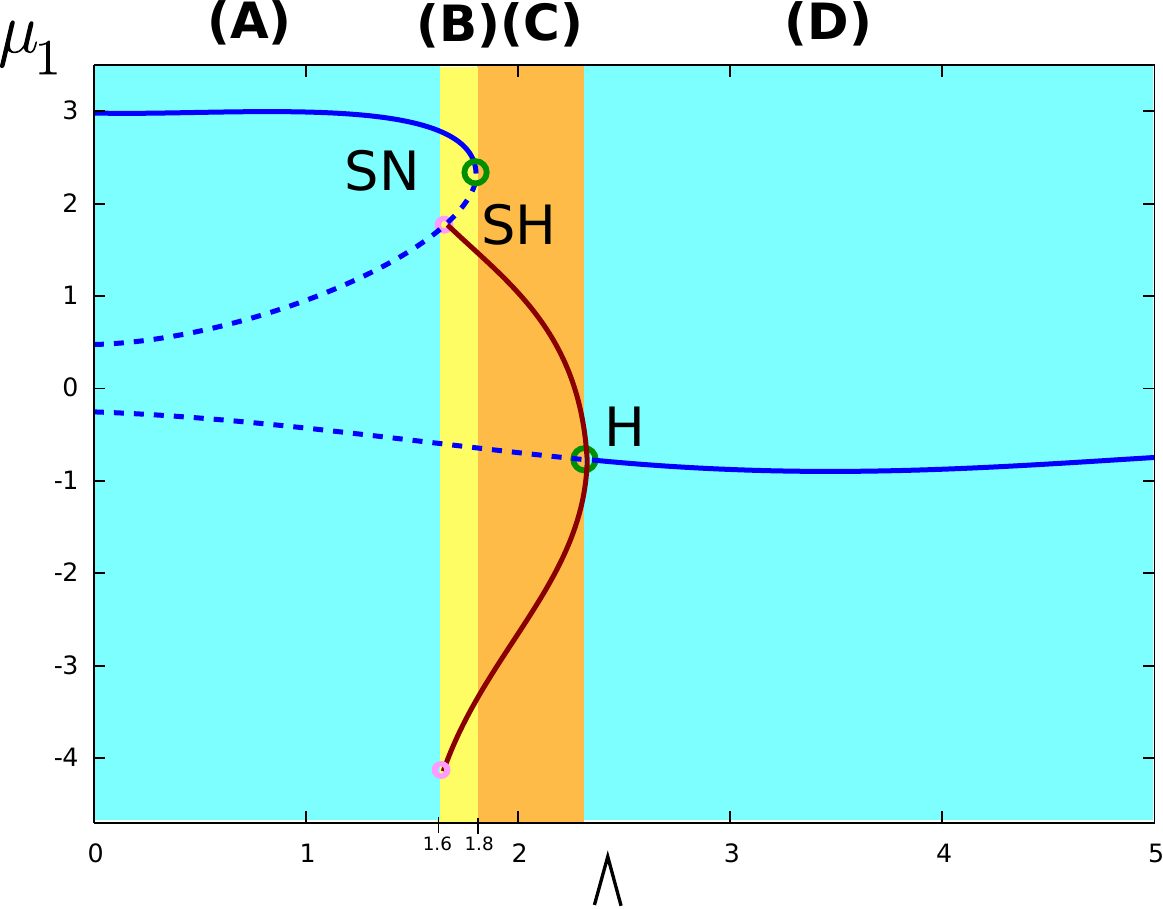}}
		\subfigure[Bifurcation diagram in $\sigma$ with $\Lambda=0.1$]{\includegraphics[width=.3\textwidth]{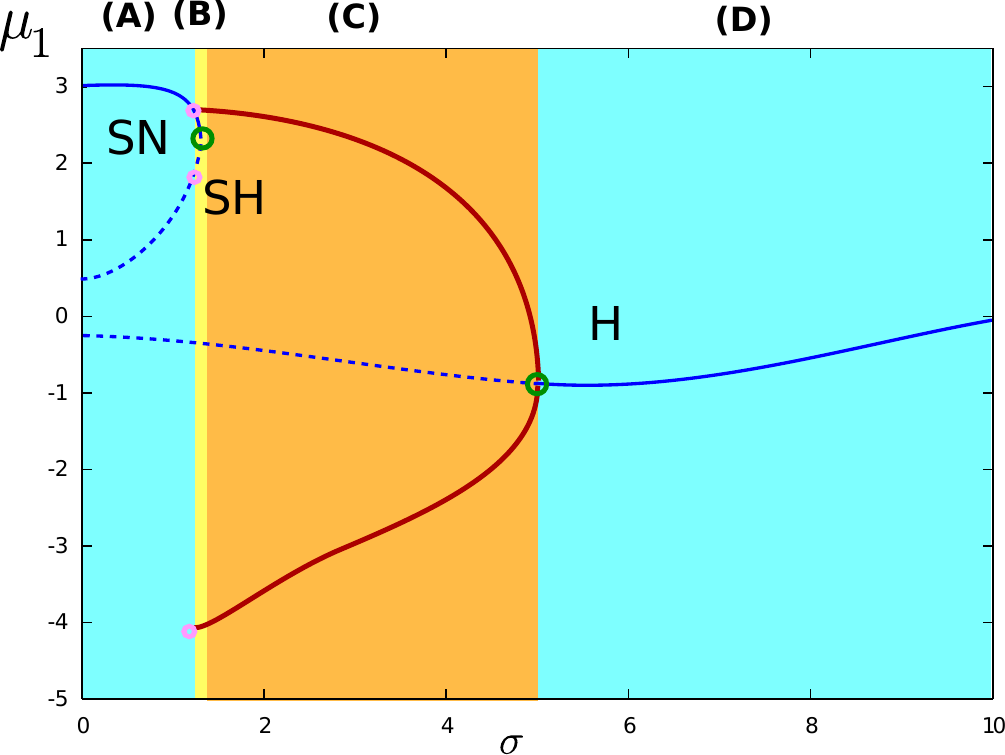}}
				\subfigure[$100$ trajectories for the network equations, $\Lambda=1.5$, $\tau=0.5$]{\includegraphics[width=.30\textwidth]{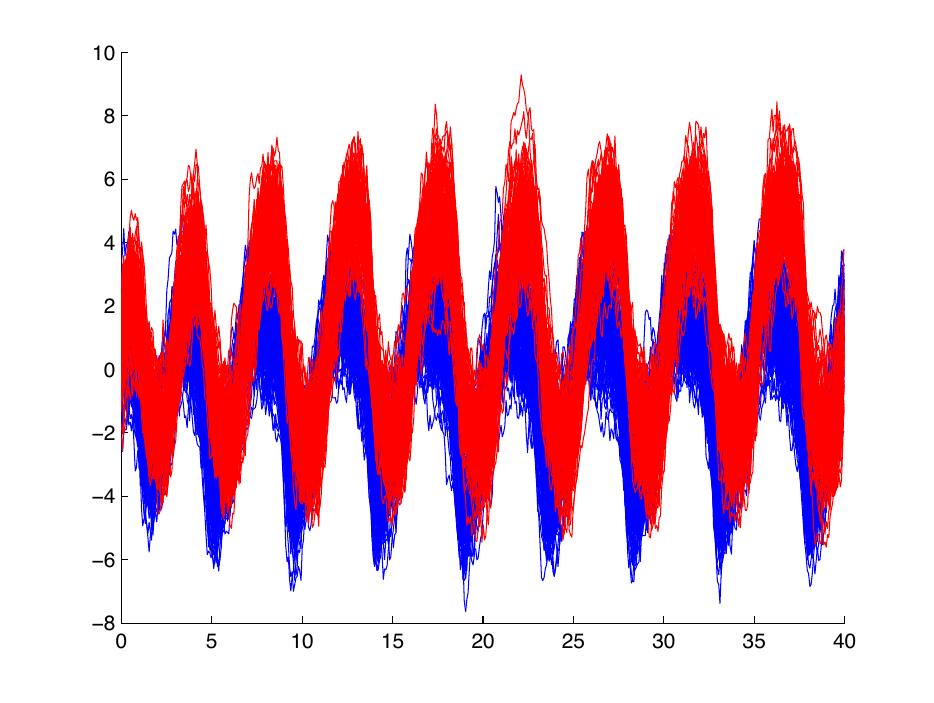}\label{fig:Behaves15}}
	\end{center}
	\caption{Dynamics and bifurcations of the spatially homogeneous equations for Network (II). (a): Codimension 2 bifurcation diagram as a function of the noise intensity $\Lambda$ and the delay $\tau$: saddle-node bifurcations (blue line) and cascade of Hopf bifurcation (pink curves) that all have a common vertical asymptote as shown in (b) for the three rightmost eigenvalues. Blue region: stationary states, orange: cycles. 
	(c): Codimension 2 bifurcation diagram as a function of $I_1$ and $\Lambda$ for $\tau=5$ (red line in diagram (a)): 2 degenerate bifurcations appear, corresponding to the tangential merging of two Hopf with the saddle-node bifurcation, and one Bogdanov-Taken bifurcation. 
	(d) Perturbed transient prior to reaching the stable oscillation or stationary state for $\tau=5$. 
	(e) Codimension 2 bifurcation diagram with respect to $\sigma$ and $\Lambda$: saddle-homoclinic (SH, green), saddle-node (SN, blue) and Hopf (H, brown) separating the bifurcation diagram into 4 zones: stationary (blue, A and D), a periodic (C, orange) and a bistable (B, yellow) zones. 
	Codimension 1 bifurcation diagram as a function of $\Lambda$ or $\sigma$ for fixed values of the other noise parameter are displayed in (f) and (g).
	(h) Synchronization of all neurons in the oscillatory region: blue (resp. red): $50$ trajectories from neurons of the excitatory (resp. inhibitory) population 1 (resp. 2)). The diagrams were obtained using DDE-BIFTOOL~\cite{DDEBif1,DDEBif2} and a specific code for the network equations.}
	\label{fig:BifDiagsDDE}
\end{figure}
 
We now analyze the dynamics of the spatially extended system and its dependence upon noise, initial datum and boundary conditions. We will distinguish between functional connectivity case where $s_1>s_2$ (inhibition more distal than excitation) and anatomical connectivity (excitation more distal than inhibition).

\paragraph{Functional Connectivity Case}
We analyze the dynamics of spatially distributed Network II in the functional connectivity case with $s_1=0.02$ and $s_2=0.0125$. 

The diagrams presented in Figure~\ref{fig:BifDiagsDDE} characterize the existence and the nature of the synchronized states. As soon as the initial condition is homogeneous, the system will present solutions that are constant in space, and their time profile is given by the solutions of the fully-synchronized system. However, these synchronized states might not be stable, and inhomogeneities in the initial condition might lead the system to different states. 
We numerically address this problem by computing the solutions of the neural-field moment equation for different values of the parameter $\Lambda$ or $\sigma$, and non-spatially homogeneous initial condition (Fig.~\ref{fig:turingPeriod1}).

\begin{figure}
	\begin{center}
		\includegraphics[width=.8\textwidth]{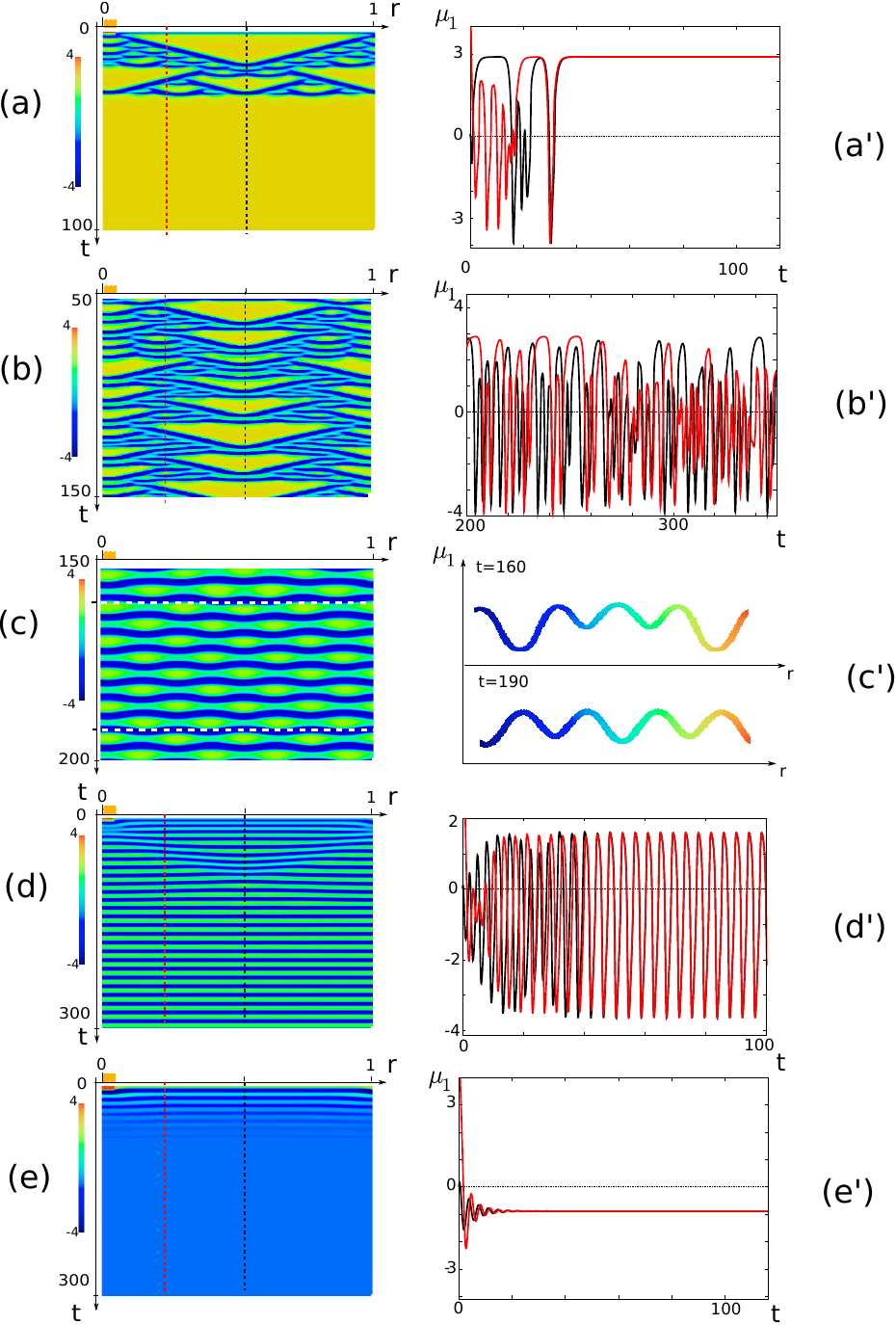}
	\end{center}
	\caption{Spatio-temporal activity (left) and sample trajectories (right) for different values of the noise parameter. (left): $\mu_1(r,t)$ as a function of $r$ (abscissa) and $t$ (ordinate). (right): (c') represents $r\mapsto \mu_1(r,t=160)$ and $r\mapsto \mu_1(r,t=190)$$\mu_1(r=0.1,t)$, other left panels represent $t\mapsto \mu_1(r=0.2,t)$ (black) and $t\mapsto \mu_1(r=0.5,t)$ (red), as shown in the left diagrams. Initial conditions are set to $\mu_1(r,0)=5$ for $r\in [0,0.05]$ and $0$ elsewhere (orange box) $\mu_2(r,0)\equiv 0$ and $v(r,0)\equiv 0$. (a): $\Lambda=0.1$, Lyapunov coefficient: $-0.997$. (b): $\Lambda=1$, Lyapunov coefficient: $0.20$. (c) $\Lambda=1.6$, Lyapunov: $0.08$. (d) $\Lambda=1.7$, Lyapunov:$-0.03$. (e) $\Lambda=3$, Lyapunov: $-0.20$. Animations of the activity are available in the supplementary material. Figures and animations were obtained using XPPAut~\cite{ermentrout:02}. }
	\label{fig:turingPeriod1}
\end{figure}
Spatially homogeneous solutions are stable in the noise regions (B) through (D). In the parameter region (A), a complex transition from stationary spatially homogeneous solutions to synchronized solutions involving chaotic patterns of activity is observed: for very small noise, the stable stationary spatially homogeneous solution appears attractive. The initial condition, strictly positive in $[0, 0.05]$ and zero otherwise, creates a bi-directional wave that travel through the neural field and splits into different secondary waves, themselves potentially splitting. All these waves interact together, and this phenomenon results in highly irregular transient behaviors (Fig.~\ref{fig:turingPeriod1} (a)). This chaotic regime becomes permanent for larger values of $\Lambda$ and the spatially homogenous solution is not recovered (Fig.~\ref{fig:turingPeriod1} (b)). As noise is further increased, the irregular wave-splitting pattern suddenly turns into a space-time quasi-periodic wave (Fig.~\ref{fig:turingPeriod1} (c)), i.e. more regular quasi-periodic spatial patterns oscillating quasi-periodically in time. These waves do not to interfere together, which explains the increased regularity observed in contrast with the patterns observed for smaller values of the noise parameter. These irregular waves progressively gain regularity as noise is further increased, and as soon as the noise parameter reaches values corresponding to the bistable parameter region (B), they turn into regular, spatially homogeneous solutions corresponding to the periodic orbit identified in the spatially homogeneous system. This stability of the spatially homogeneous state persists in the parameter regions (C) and (D) (Fig.~\ref{fig:turingPeriod1} (d) and (e)). The effect of varying of the noise parameter $\sigma$ are qualitatively the same (not shown). One difference is the synchronization of the oscillations that appears sharper due to the fact that the variance oscillates and reaches very small values. 
This phenomenon also persists when considering different boundary conditions on the neural field, as shown in ~\ref{append:BoundaryConditions}. 

\paragraph{Anatomical Connectivities: Bumps, bump-splitting and wave interference}
The anatomical case, where excitatory connections are more distal than inhibitory connections, shows clear qualitatively distinctions. Incidentally, because of the rescaling of the connectivity kernels, the bifurcation diagram of the spatially homogeneous system is the same as in the functional case (Fig.~\ref{fig:Codim1Anatomical}) and hence such neural fields show the same transition between spatially homogeneous stationary and periodic solutions as a function of noise levels. Moreover, similarly to the functional case, it appears that chaotic instabilities occur in the low noise regime (A) corresponding to the transition to synchronized oscillations. However, the nature of the transition appears to significantly depend on the kernel extension ratio $r=s_1/s_2$ (see Figure~\ref{fig:Anatom1}): for large ratios (typically larger than $0.6$ in our system), wave splitting persists as in the functional case (Fig.~\ref{fig:Anatom1}(l)). 
\begin{figure}
	\begin{center}
		\subfigure[Spatially Homogeneous]{\includegraphics[width=.3\textwidth]{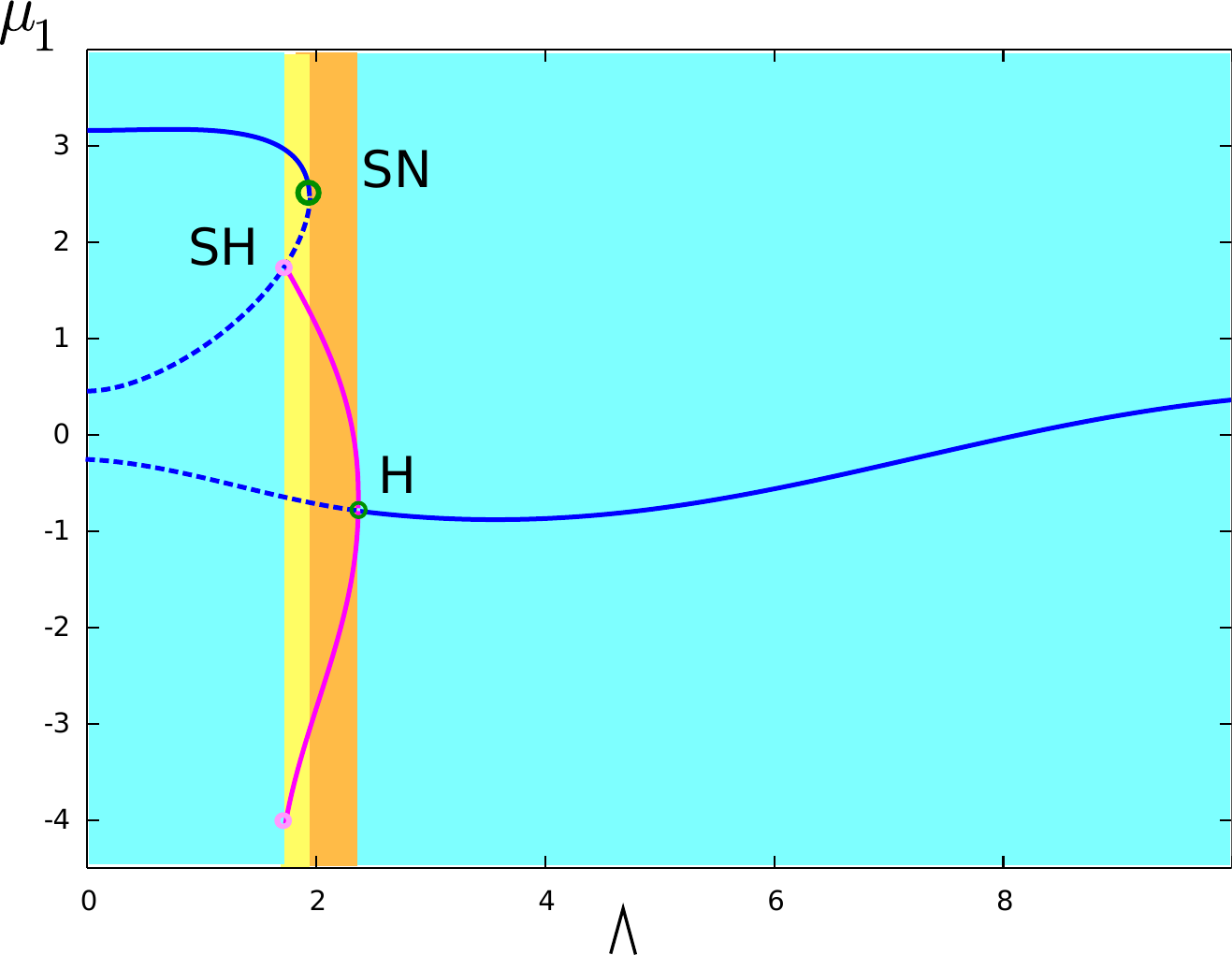}\label{fig:Codim1Anatomical}}
		\subfigure[$\Lambda=0.1$]{\includegraphics[width=.3\textwidth]{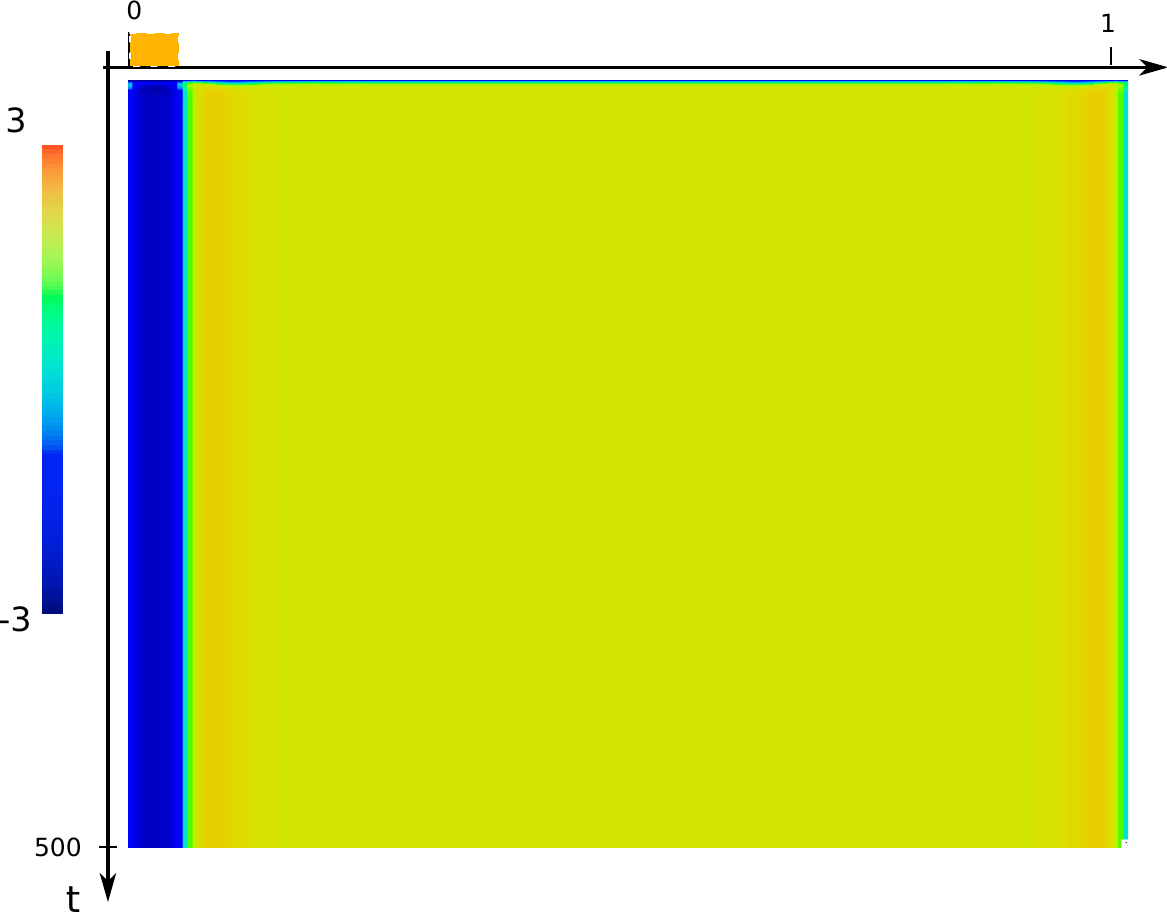}}
		\subfigure[$\Lambda=0.5$]{\includegraphics[width=.3\textwidth]{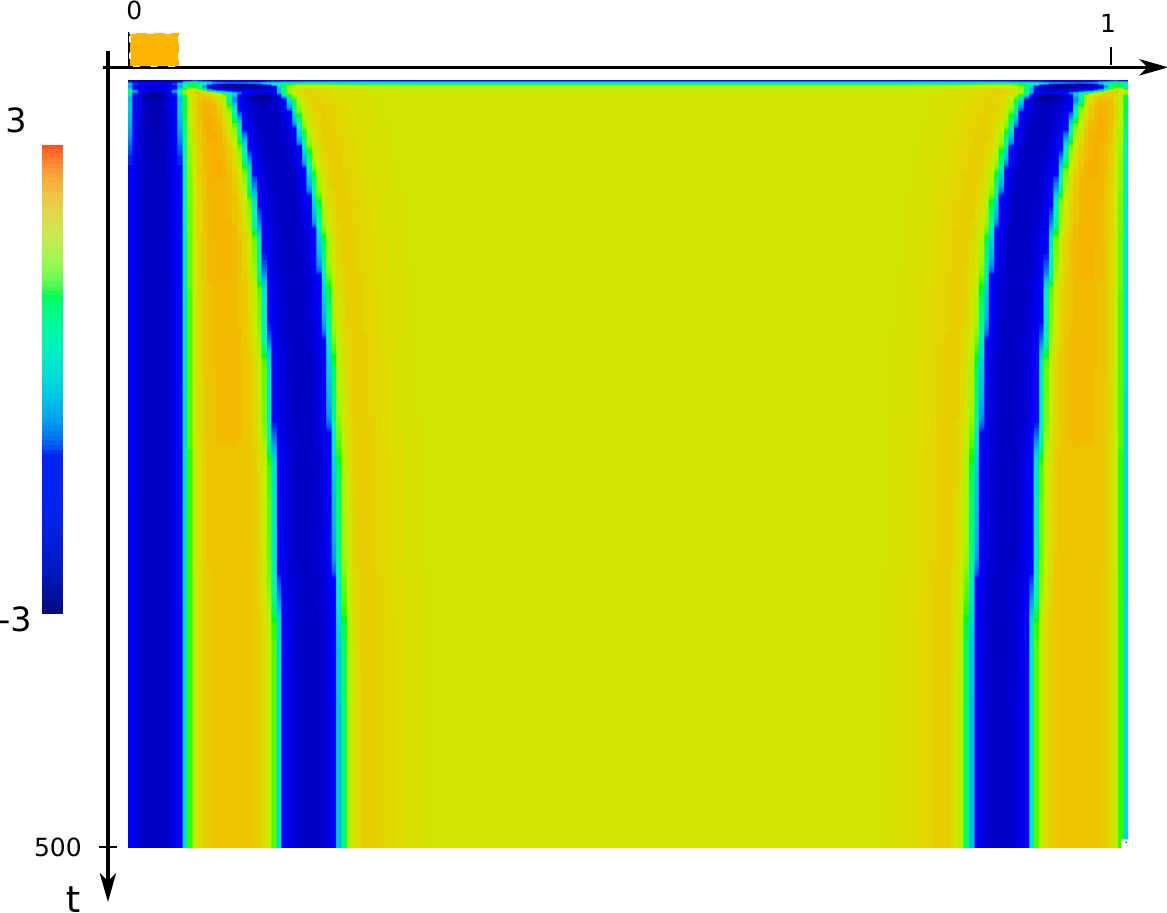}}\\
		\subfigure[$\Lambda=0.7$]{\includegraphics[width=.3\textwidth]{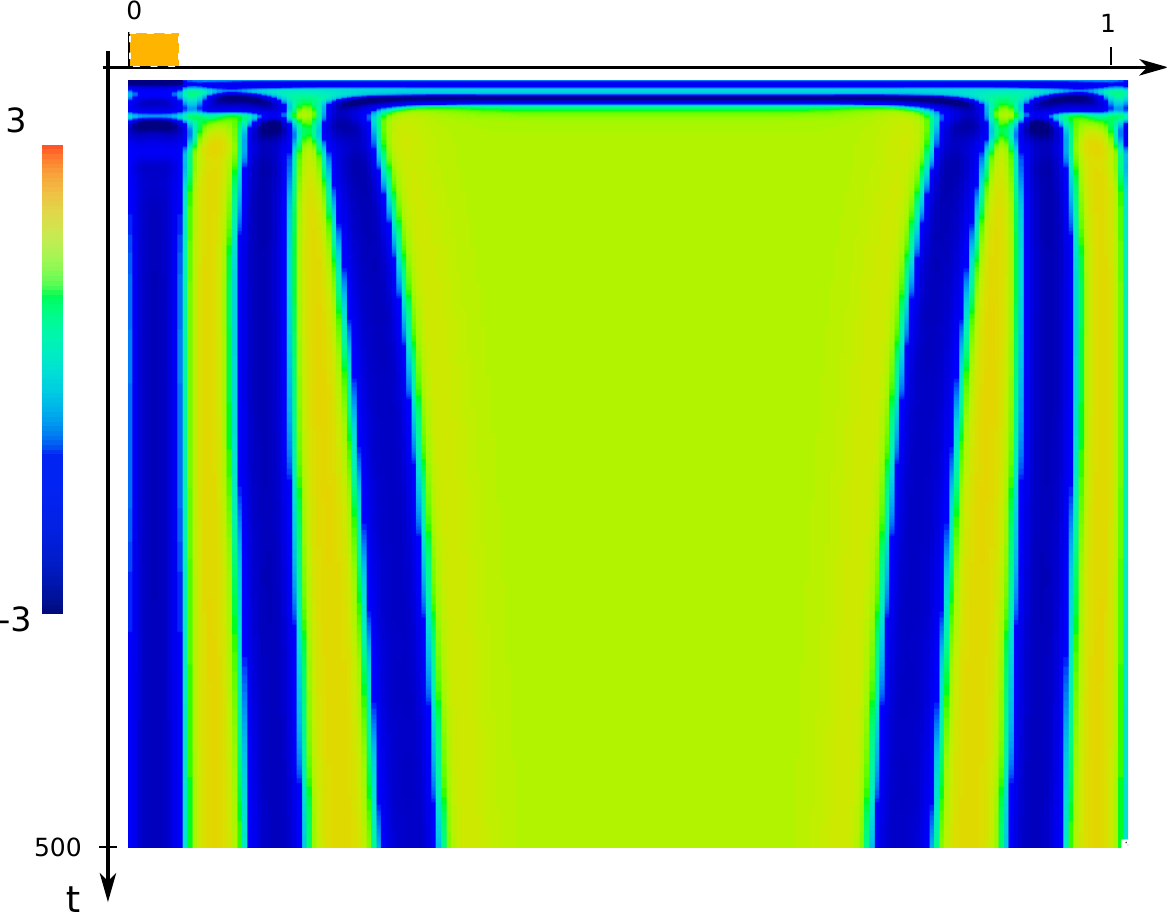}}
		\subfigure[$\Lambda=0.8$]{\includegraphics[width=.3\textwidth]{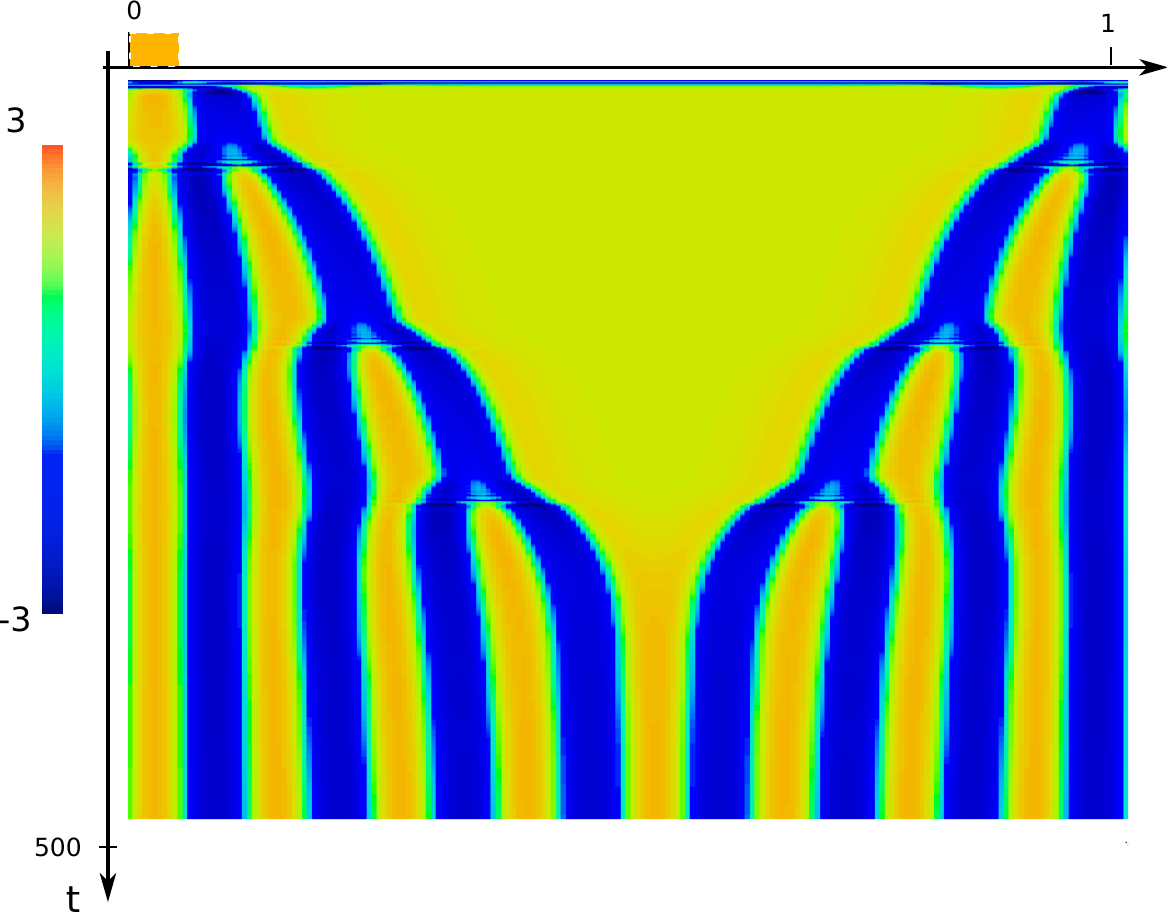}\label{fig:8Bumps}}
		\subfigure[$\Lambda=0.8$,IC2]{\includegraphics[width=.3\textwidth]{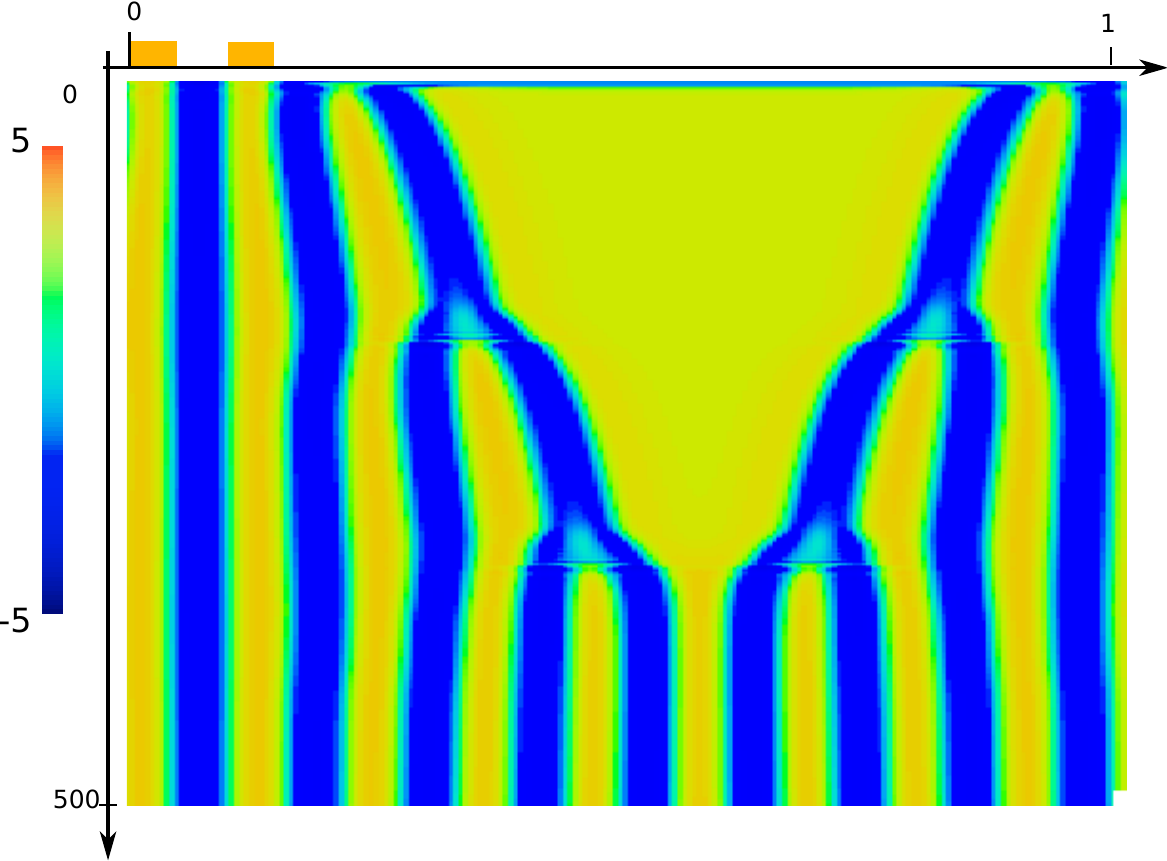}\label{fig:9Bumps}}\\
		\subfigure[$\Lambda=1$]{\includegraphics[width=.3\textwidth]{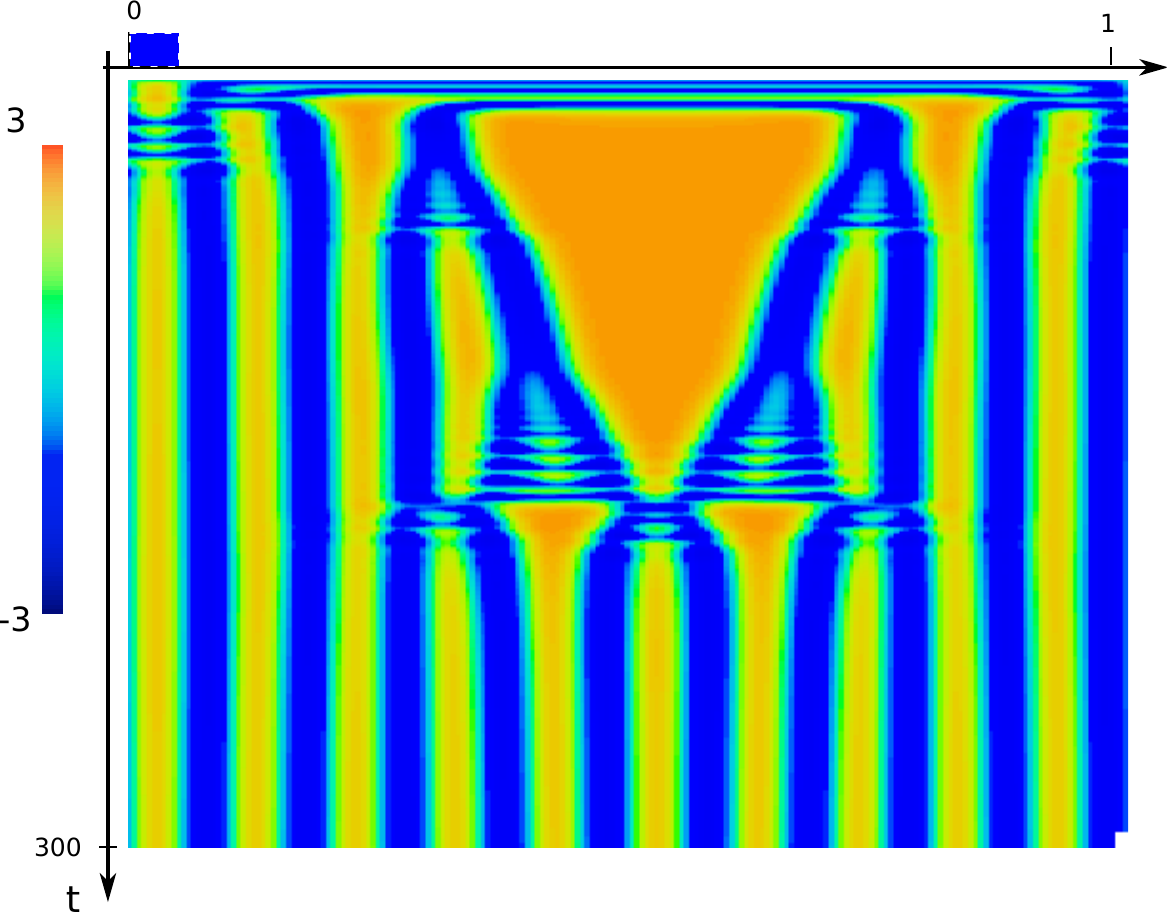}}
		\subfigure[$\Lambda=1.4$]{\includegraphics[width=.3\textwidth]{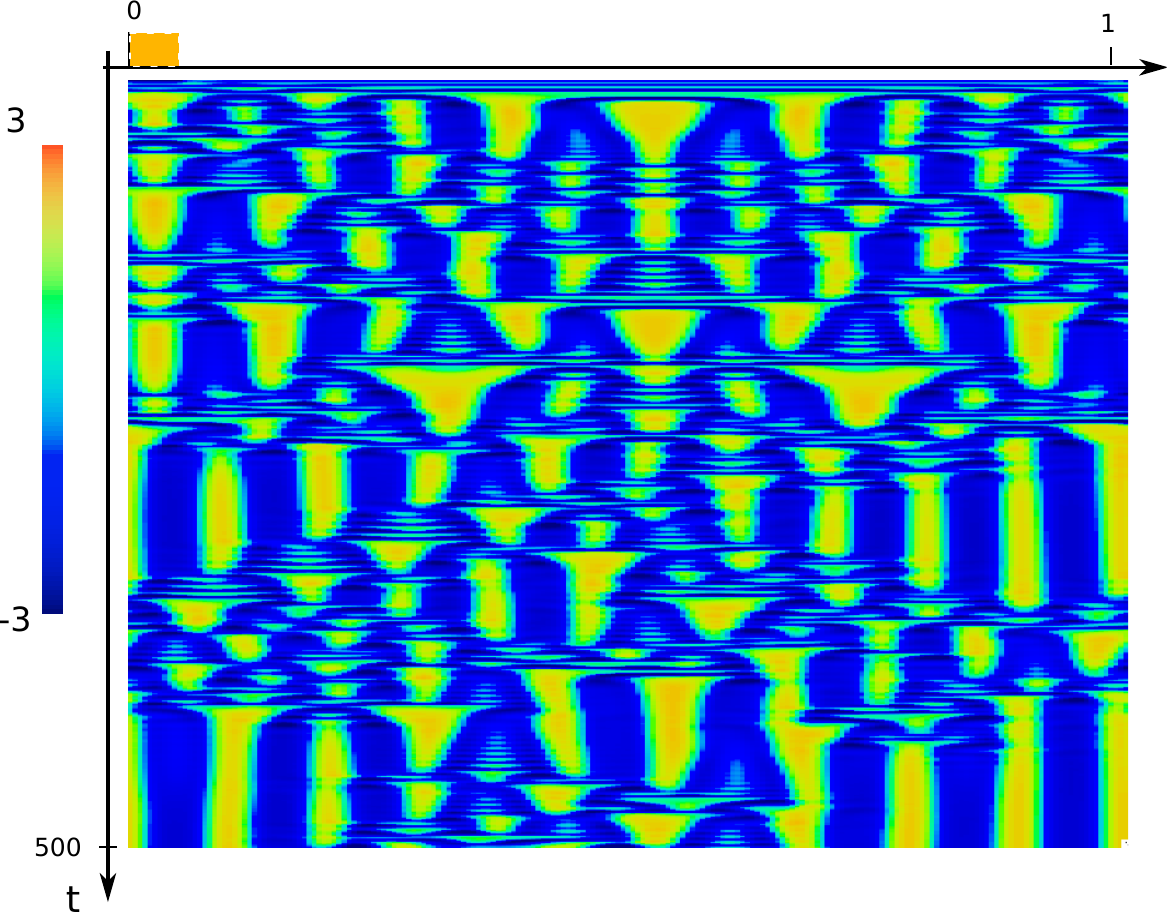}}
		\subfigure[$\Lambda=1.6$]{\includegraphics[width=.3\textwidth]{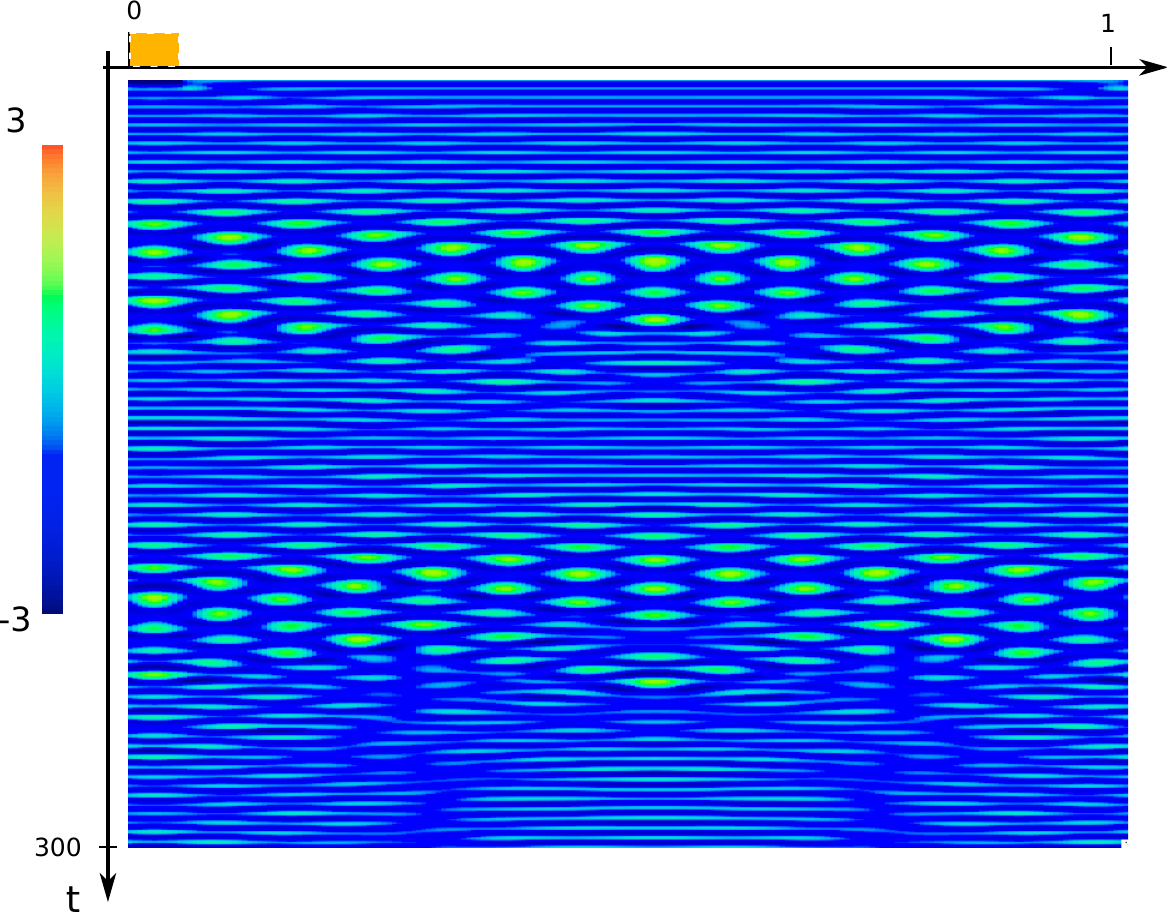}\label{fig:SpaceTimeWaveAna}}\\
		\subfigure[$\Lambda=1.7$]{\includegraphics[width=.3\textwidth]{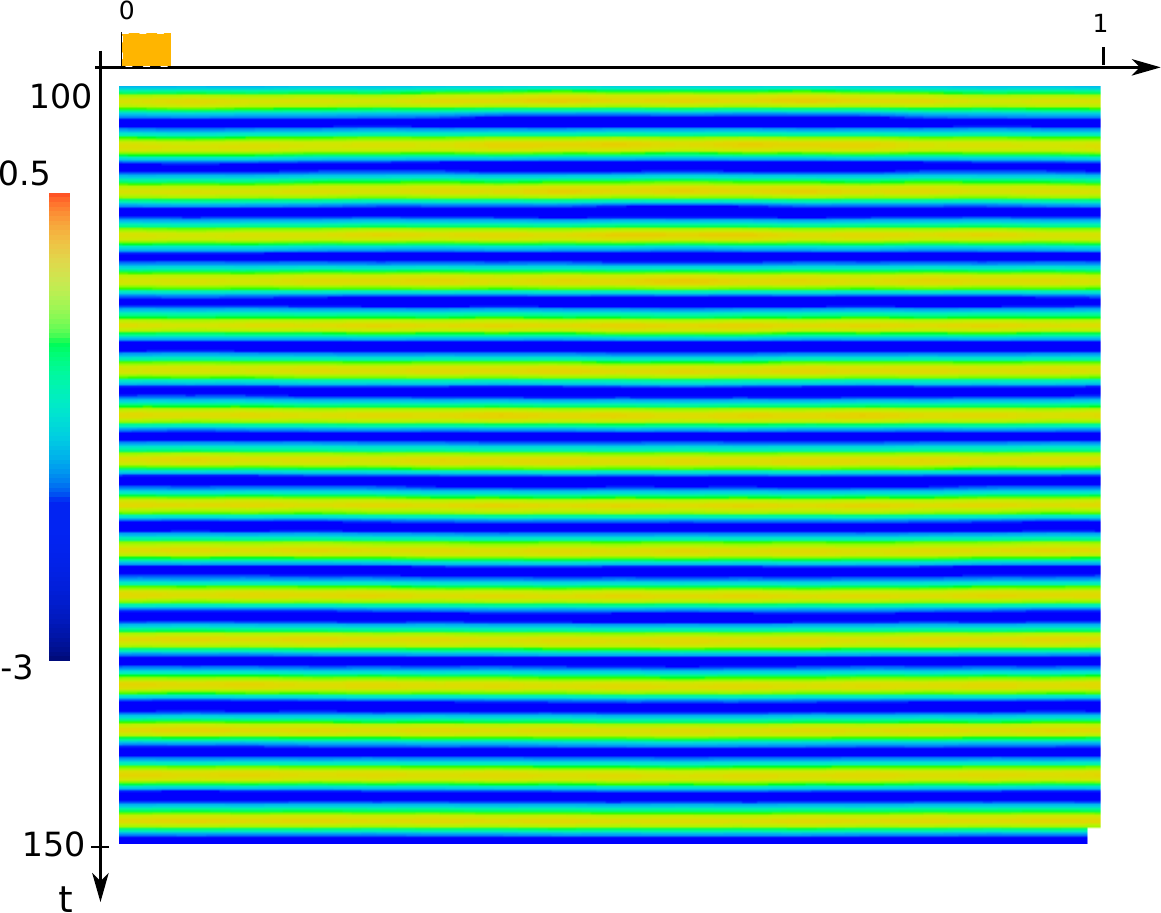}}
		\subfigure[$\Lambda=2.5$]{\includegraphics[width=.3\textwidth]{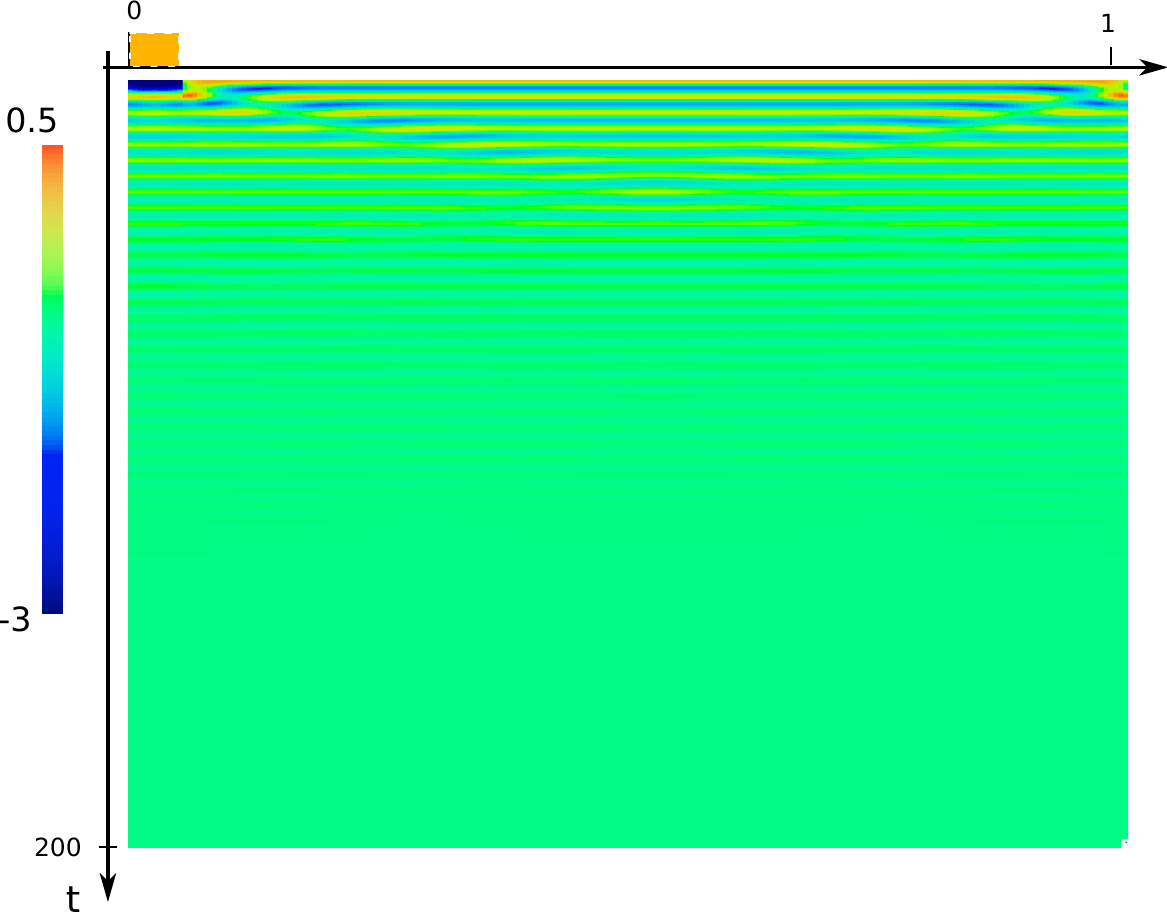}}
		\subfigure[$\Lambda=1.3$, $\frac{s_1}{s_2}>r^*$]{\includegraphics[width=.3\textwidth]{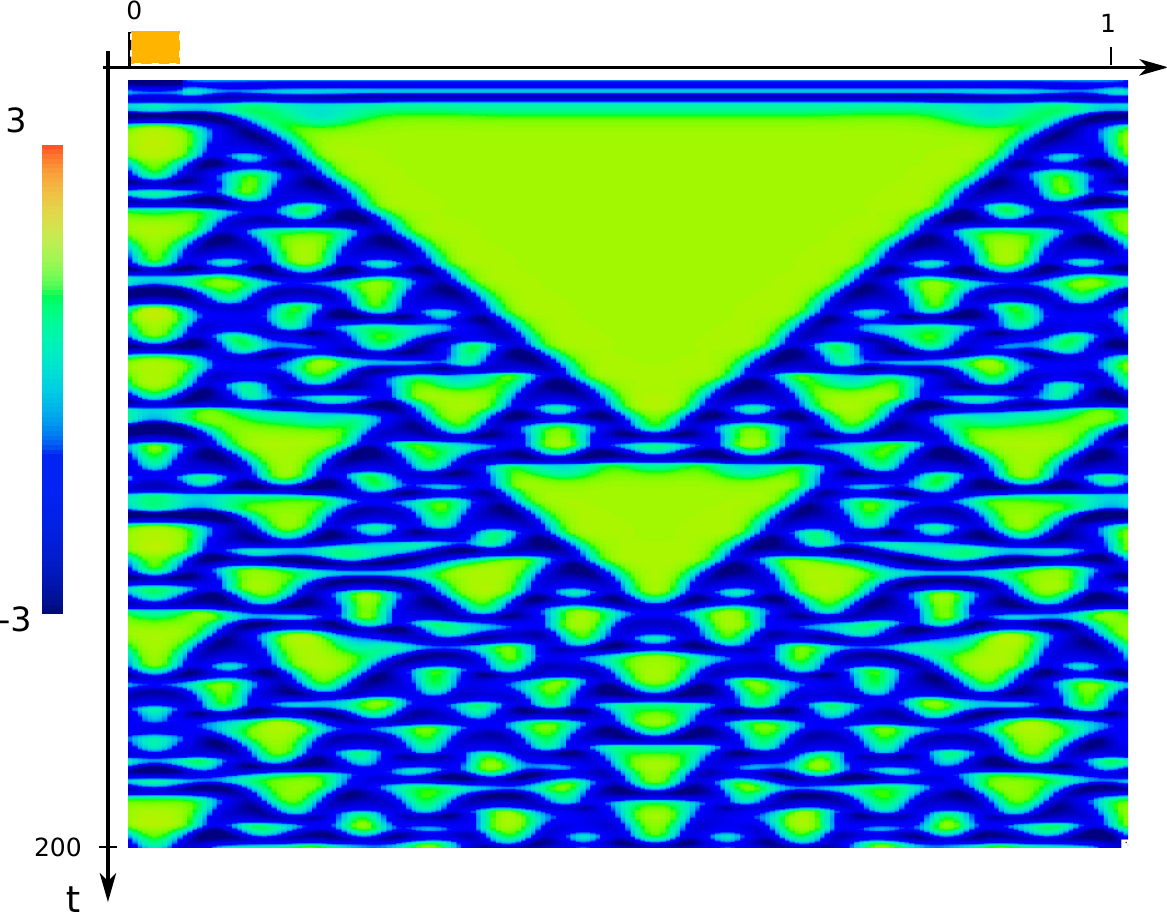}\label{fig:WaveSplit}}
	\end{center}
	\caption{Anatomical connectivity case with $\mu_1(r,t)$ represented as a function of $r\in\mathbbm{S}^1$ (abscissa) and $t$ (ordinate). (a)-(k): $s_1/s_2=0.5$ show a sequence of bump splitting as $\Lambda$ is increased in the parameter region where stationary spatially homogeneous solutions related to small values of $\Lambda$. Orange Box: initial condition $\mu_1(r,0)=5$ for $r\in[0,0.05]$ and $0$ otherwise, and blue box: $-\mu_1(r,0)$. (l): $s_1/s_2=0.65$: wave splitting phenomenon, for $\Lambda=1.3$.}
	\label{fig:Anatom1}
\end{figure}
For smaller ratios, a new type of dynamics appears in the parameter region (A) characterized by the presence localized stationary patterns of activity (bumps). The number and stability of these bumps depend on the level of noise: as noise is increased, bumps tend to split in two different bumps of the same spatial size and either stabilize, or split again. This phenomenon is strongly evocative of the patterns observed in a different context by Coombes and Owen in~\cite{coombes-owen:07}. This sequence of bump splitting either stabilizes on a stationary pattern composed of several bumps (Figs.~\ref{fig:Anatom1}(b)-(g)), or repeat indefinitely and irregularly for higher noise levels (Fig.~\ref{fig:Anatom1}(h)). It is interesting to note that we observed that for $\Lambda>1.5$, the stationary pattern found is spatially periodic, characterized by a specific wavenumber increasing as noise is increased, and depending on the type of initial condition chosen. For instance, we show in Figure~\ref{fig:8Bumps} the case of $\Lambda=1.5$ and initial condition zero except on $[0,0.05]$ where it is equal to $5$ (IC1), where the stationary behavior is characterized by a spatially periodic pattern with wavenumber $8$, and for initial condition zero except on $[0,0.05]\cup [0.1,0.15]$ where it is equal to $5$, the same kind of phenomenon appears and stabilizes on a pattern with wavenumber equal to $9$. For $\Lambda=1.6$ and initial conditions IC1, the wave number is $10$. As noise is further increased, the system starts presenting quasi-periodic spatio-temporal waves (Fig.~\ref{fig:SpaceTimeWaveAna}) similarly to the functional connectivity case. 

These bumps disappear in favor of a spatially homogeneous periodic activity when noise levels reach the bistable region, and this region turns into spatially homogeneous stationary solutions as noise is further increased, in the case corresponding to regions (B), (C) and (D).

\section{Discussion}\label{sec:Discussion}
In this article, we initiated a study of neural fields in the presence of noise, based on a microscopic model accounting for neuron's dynamics. Limits of such networks in the presence of noise, analyzed in~\cite{touboulNeuralfields:11}, show a complex interplay between noise and the dynamics, encapsulated in intricate equations on the space of stochastic processes, that appear mostly impenetrable. In this article, we provided a sufficient condition for spatially homogeneous solutions to exist in such equations. In order to precisely analyze the dynamics of these equations, we applied the formalism to networks of firing-rate neurons. In that case, solutions converge towards Gaussian processes whose mean and standard deviation satisfy a closed deterministic system of coupled integro-differential equations. Incidentally, this new set of equations is compatible, in the zero noise limit, with the usual Wilson and Cowan equations widely used in the study of neural fields. This suggests that in our modeling bridging microscopic and macroscopic states, the Wilson and Cowan system actually reflects the behavior of \emph{non-noisy} neurons rather than corresponding to averaging effects. In the stochastic model, the variance of the process nonlinearly impacts the evolution of the mean activity. 

The compatibility of the usual Wilson and Cowan system with the new set of equations directly allows identifying the effects of noise on the neural fields dynamics. An important conclusion of the present article is the significant effect of the microscopic noise levels on the macroscopic dynamics. Analyzing the bifurcations of the newly derived set of equations as a function of noise levels evidenced several non trivial qualitative effects on the dynamics. A relatively universal phenomenon observed in all the models analyzed was the stabilization by noise: we identified in all the cases treated stationary solutions, unstable in the low-noise regime, that gain stability in a high noise regime, while other solutions lose stability and disappear. This property quantify the heuristic principle that when noise exceed certain levels, it tends to dominate the dynamics and destroy fine structures of the underlying dynamical system. This is further illustrated that in all the cases treated, the stabilized stationary solution has an average close to a median value of the fixed points and periodic orbits observed for smaller noise. More surprising was the fact that noise lead to the emergence of a structure activity, for instance solutions periodic in law, in a case where the non-noisy system did not display any cycle. This phenomenon can be related to random asymmetric transitions between attractors, or to random crossings of separatrices, that become regular when the population size increase. For instance in the two-populations firing-rate model (Network II), the dynamical system related to one neuron in each population features a stable fixed point, an unstable fixed point and an hyperbolic fixed point, the stable manifold of which separates trajectories directly returning to the fixed point from trajectories making large excursions around an heteroclinic cycle. As noise is increased, the occurrence of these large excursions become increasingly frequent, accounting for the possibility of macroscopic oscillations. Let us emphasize that this route to synchronization and oscillations is distinct from phenomena documented in the neural networks literatures, such as coherence or stochastic resonance~\cite{pham-pakdaman:98,nesse-bressloff:08,lindner:04}. Indeed, beyond the absence of periodic forcing present in the stochastic resonance, the perfectly periodic behavior clearly distinguishes the present purely collective phenomenon from these more usual phenomena. 
An important point of this analysis is that low-pass filtering macroscopic signals does not cancel out the noise present at the microscopic scale and that essential qualitative features of the macroscopic signal are observed.

Another intriguing effect is the nature of the transition from stationary to periodic behaviors in the spatially extended system. We indeed showed that stationary spatially homogeneous states might lose stability in favor of irregular spatio-temporal behaviors as a function of noise levels, yielding highly irregular spatio-temporal patterns. These patterns significantly depend on the relative extension of the connectivity kernels related to the excitatory and the inhibitory populations: functional connectivity cases with excitation more local than inhibition showed a wave-splitting phenomenon whereas anatomical connectivity cases corresponding to more distal excitation were characterized by a sequence of bump splitting. These phenomena might be related to the presence of several attractors and to the phenomenon of chaotic switching. This formation of complex spatio-temporal patterns arising from a wave-splitting phenomenon strongly evokes Turing patterns as found in different reaction diffusion equations in biological mathematics. In particular, these patterns are similar to those exhibited by~\cite{meinhardt-klinger:87} and obtained from the analysis of the dynamics of reaction diffusion equations related to pattern formation on the shells of mollusks, in a case where the system induces the formation of forward and backward running interacting waves. These are also close from the results of Coombes and Owen in~\cite{coombes-owen:07}, where the authors identify self-replicating bumps, compared with dynamics observed dissipative equations such as Ginzburg-Landau's, but to our knowledge not observed in neural field equations.
 
Another interest of the present approach is the fact that one can directly infer the microscopic stochastic behavior of each cell in the mean-field limit. Indeed, the uniform propagation of chaos property demonstrated in the present manuscript ensures that any neuron in the network converges after a transient phase towards an independent Gaussian process with mean and standard deviation solutions of the dynamical system exhibited here. In particular, when oscillations take place, all neurons in the network present phase-locked oscillations, the phase being set by the choice of the initial condition. 

These different phenomena all point towards the same conclusion that noise has a significant role on shaping the activity of large network, and shed new light on currently widely debated question of the functional role of noise in the brain. Our analysis directly relates noise to the presence of synchronized oscillations. This sharp and sudden synchronization of neurons on relatively large, slow periodic orbits evokes the apparition of seizures. And incidentally, it was observed experimentally that increased variability of the post-synaptic currents (that would correspond in our model to an increased synaptic noise coefficient $\sigma$) was related to the occurrence of epileptic seizures (see~\cite{aradi-soltesz:02}). This suggests an important direction in the development of this work, consisting in fitting the microscopic model to biological measurements. This would yield a new neural assembly model for large scale areas and develop studies on the appearance of stochastic seizures and rhythmic activity in relationship with different parameters of the model, integrating the presence of noise in a mathematically and biologically relevant manner. This would also allow deriving quantitative predictions on the noise-induced transitions, that would then potentially be experimentally tested either to evaluate the level of noise in a cortical network, or to control behaviors through noise levels. This seems plausible nowadays since experimental techniques to control noise levels are now well understood: for instance, direct injection of  stochastic conductances is also be performed through dynamics clamp (see e.g.~\cite{destexhe-contreras:06}), and several techniques controlling channel noise levels are reviewed in~\cite{diba-koch:04}. Besides, the present approach allowed accounting for several collective phenomena: reliable response (in that case in probability distribution), decorrelation~\cite{ecker-berens-etal:10} and synchronized oscillations. 

The first limitation of the present study is the fact that this precise analysis is only valid in the case of firing-rate neurons, where we were able to exhibit Gaussian local equilibria. Though popular in the study of neural fields and widely used, the model does not take into account the highly nonlinear nature of several neuronal phenomena. These linear models, though less general and accurate than the nonlinear ones, yielded much greater insight, and in particular analytic treatment, of the cortical dynamics within their regimes of validity. The Gaussian nature of the solution is hence a breach we exploited to go further in analyzing the complex dynamics of the neural fields mean-field equations, with the aim of further understanding nonlinear neuron models. The analysis provided here is the first (and to our understanding, one of the only case where such an analytical study is possible) to address precisely the dynamics of such complex mean-field equations, and can also be seen as a proof of concept of the dynamics of this class of equations, in particular the effects of noise and delays in these equations. A perspective of great interest also would be to derive from the non-linear mean-field equations systems governing macroscopic variables such as the mean firing-rate. This is a complex and deep question we are currently investigating.

\appendix
\section{Existence and Uniqueness of solutions of the synchronized mean-field solution}\label{append:ExistUniqueSynchro}

In this appendix we provide the proof of Proposition~\ref{pro:Synchro}, which we recall here for completeness.

\begin{proposition}
	Assume that the distribution of the initial condition $\zeta^0_t(r)$ is chaotic and independent of $r$, and that the functions $G(r,t,x)$ and $g(r,t,x)$ do not depend on $r$. Moreover, if the law of the quantities:
	\begin{equation*}
		\begin{cases}
			B(r,x,\varphi)\eqdef\int_{\Gamma}\int_{-\tau}^0  b(r,r',x,\varphi(u)) d\eta(r,r',u)\lambda(r')dr'\quad \text{and}\\ 		
			H(r,x,\varphi)\eqdef\int_{\Gamma}\int_{-\tau}^0 \int_0^t \beta(r,r',\psi(s),\varphi(s+u))  d\mu(r,r',u)\lambda(r')dr'dB_s
		\end{cases}
	\end{equation*}
	do not depend on $r$ for any $(\psi,\varphi)$ measurable functions, then the solution of the mean-field equation~\eqref{eq:MFESpaceSummary} is spatially homogeneous in law. The common spatially homogeneous law is solution of the implicit equation:
	\begin{multline*}
		\displaystyle{X_t (r_0)= \zeta^0_0(r_0) + \int_0^t ds \Big( G(r_0,s,X_s(r_0)) + \Exp_{Z} [B(r_0,X_s(r_0),Z_{(\cdot)}(r_0))]\Big)}\\
	 \displaystyle{+\Exp_{Z} [ H(r_0,X_{(\cdot)}(r_0), Z_{(\cdot)}(r_0) )] + \int_0^t dW_s g(r_0,s,X_s(r_0)),}
	\end{multline*}
	which has a unique solution. 
\end{proposition}

\begin{proof}
The existence and uniqueness property of solutions proved in~\cite{touboulNeuralfields:11} is based on a classical contraction argument on the map $\Phi$ acting on stochastic processes:
\begin{align*}
	\Phi(X)_t(r) &= \begin{cases}
		\zeta^0_0(r) + \int_0^t ds \Big( G(r,s,X_s(r)) + \int_{\Gamma} \lambda(r')\,dr' \int_{-\tau}^0 d\eta(r,r',u) \Exp_{Z} [ b(r,r',X_s(r), Z_{s+u} (r') )] \Big) \\
				\qquad \quad+ \int_{\Gamma} \lambda(r')\,dr' \int_{-\tau}^0 d\mu(r,r',u) \int_0^t \Exp_{Z} [ \beta(r,r',X_s(r), Z_{s+u} (r') )] dB_s \\
				\qquad \quad+ \int_0^t dW_s g(r,s,X_s(r)) \qquad,\qquad  t>0\\
				 \zeta^0_t(r) \qquad ,\qquad t\in [-\tau, 0]\\
				 (Z_t)\eqlaw (X_t)  \text{ independent of $(X_t)$, $(W_t)$ and $(B_t)$}
	\end{cases}
\end{align*}
Routine fixed-point methods allows demonstrating that the unique solution of the mean-field equations is the limit of the recursion $X^{n+1}_t(r)=\Phi(X^n_t(r))$ starting from any square integrable initial process $X^0_t(r)$. Let $X^0_t(r)$ a stochastic process whose law does not depend on $r$: $X^0_t(r)\eqlaw X^0_t(r_0)$ for any $(r,r_0)\in\Gamma$. Then $\Phi(X^0)_t(r)$ does not depend on $r$, since we have\footnote{For simplicity we denoted the variable over which the integration takes place just after the integral sign.}:
\begin{align*}
	\Phi(X^0)_t(r) &= \displaystyle{\zeta^0_0(r) + \int_0^t ds \Big( G(r,s,X_s^0(r)) + \int_{\Gamma} \lambda(r')\,dr' \int_{-\tau}^0 d\eta(r,r',u) \Exp_{Z^0} [ b(r,r',X_s^0(r), Z^0_{s+u} (r') )] \Big)} \\
				&\quad \displaystyle{+ \int_{\Gamma} \lambda(r')\,dr' \int_{-\tau}^0 d\mu(r,r',u) \int_0^t \Exp_{Z^0} [ \beta(r,r',X^0_s(r), Z^0_{s+u} (r') )] dB_s + \int_0^t dW_s g(r,s,X_s^0(r))} 
				\end{align*}
				\begin{align*}
				&\displaystyle{\eqlaw \zeta^0_0(r_0) + \int_0^t ds \Big( G(r_0,s,X_s^0(r_0)) + \int_{\Gamma} \lambda(r')\,dr' \int_{-\tau}^0 d\eta(r,r',u) \Exp_{Z^0} [ b(r,r',X_s^0(r_0), Z_{s+u}^0(r_0))] \Big)} \\
				& \displaystyle{+ \int_{\Gamma} \lambda(r')\,dr' \int_{-\tau}^0 d\mu(r,r',u) \int_0^t \Exp_{Z^0} [ \beta(r,r',X_s^0(r_0), Z_{s+u}^0(r_0) )] dB_s + \int_0^t dW_s g(r_0,s,X_s^0(r_0))}\\
				&=  \displaystyle{\zeta^0_0(r_0) + \int_0^t ds \Big( G(r_0,s,X_s^0(r_0)) + \Exp_{Z^0} [B(r,X_s^0(r_0),Z_{(\cdot)}^0(r_0))]\Big)}\\
				&\qquad \displaystyle{+\Exp_{Z^0} [ H(r,X_{(\cdot)}^0(r_0), Z^0_{(\cdot)}(r_0) )] + \int_0^t dW_s g(r_0,s,X_s^0(r_0))} \\
				&\eqlaw  \displaystyle{\zeta^0_0(r_0) + \int_0^t ds \Big( G(r_0,s,X_s^0(r_0)) + \Exp_{Z} [B(r_0,X_s^0(r_0),Z_{(\cdot)}^0(r_0))]\Big)}\\
				&\qquad \displaystyle{+\Exp_{Z^0} [ H(r_0,X_{(\cdot)}^0(r_0), Z^0_{(\cdot)}(r_0) )] + \int_0^t dW_s g(r_0,s,X^0_s(r_0))} \\
				&=\Phi(X^0)_t(r_0)
\end{align*}
All the processes $X^n_t(r)$ hence have a law independent of $r$ by an immediate recursion, and so does the limit. We hence proved that the unique solution of the mean-field equations is spatially homogeneous in law, and obviously satisfies equation~\eqref{eq:SynchronizedLaw}.

The proof of existence and uniqueness of solutions for equation~\eqref{eq:SynchronizedLaw} uses also the classical fixed-point argument. Since the quantities $B(r,x,\varphi)$ and $H(r,\psi,\varphi)$ do not depend on $r$, we drop the dependence of these functions in $r$. As usually done, we transform the equation~\eqref{eq:SynchronizedLaw} into a fixed point equation on the space of stochastic processes. To this end, let us define the map $\Psi$ as follow:
\begin{align*}
	\Psi(X)_t &= \begin{cases}
		\displaystyle{\zeta^0_0 + \int_0^t ds \Big( G(s,X_s) + \Exp_{Z} [B(X_s,Z_{(\cdot)})] \Big)}  \displaystyle{+ \Exp_{Z} [ H(X_{(\cdot)}, Z_{(\cdot)} )] + \int_0^t g(s,X_s) dW_s},\;  t>0\\
				 \zeta^0_t(r) \qquad ,\qquad t\in [-\tau, 0]\\
				 (Z_t)\eqlaw (X_t)  \text{ independent of $(X_t)$, $(W_t)$ and $(B_t)$}
	\end{cases}
\end{align*}
The solutions of equation~\eqref{eq:SynchronizedLaw} are exactly the fixed points of $\Psi$. We assume here that $G$ and $g$ are $K$-Lipschitz-continuous and satisfy the linear growth condition, and $b$ and $\beta$ are $L$-Lipschitz continuous in both their variables. It is easy to show that any possible solution has a bounded second moment following~\cite{touboulNeuralfields:11}.

	\noindent {\it Existence:}\\
	Let $X^0\in \M^2(\C)$ the space of square integrable stochastic processes such that $X^0\vert_{[-\tau, 0]} \eqlaw \zeta_0$ a given stochastic process. We introduce the sequence of probability distributions $(X^k)_{k \geq 0}$ defined by induction as $X^{k+1}=(\Psi(X^k))$. We denote by $(Z^k)$ a sequence of processes independent of the collection of processes $(X^k)$ and having the same law. We analyze $X^{k+1}_t-X^k_t$ and decompose it into the sum of six elementary terms as follows:
	\begin{align*}
		X^{k+1}_t-X^k_t &=\displaystyle{\int_0^t \Big ( G(s,X^{k}_s) - G(s,X^{k-1}_s)\Big)\,ds} \\
		& \quad\quad \displaystyle{+ \int_0^t \Exp_Z\Big[ B(X^{k}_s,Z^{k}_{\cdot}) - B(X^{k-1}_s, Z^{k}_{\cdot}) \Big] \, ds} \\ 
		& \quad\quad \displaystyle{+ \int_0^t \Exp_Z\Big[ B(X^{k-1}_s,Z^{k}_{\cdot}) - B(X^{k-1}_s, Z^{k-1}_{\cdot}) \Big] \, ds} 
		\end{align*}
		\begin{align*}
		& \quad\quad \displaystyle{+\int_0^t \Big ( g(s,X^{k}_s)-g(s,X^{k-1}_s) \Big) \, dW_s}\\
		& \quad\quad \displaystyle{+ \Exp_Z\Big [ H(X^{k}_{\cdot}, Z^{k}_{\cdot}) - H(X^{k-1}_{\cdot}, Z^{k}_{\cdot}) \Big]} \\ 
		& \quad\quad \displaystyle{+ \Exp_Z\Big [ H(X^{k-1}_{\cdot}, Z^{k}_{\cdot}) - H(X^{k-1}_{\cdot}, Z^{k-1}_{\cdot}) \Big]} \\
		& \stackrel{\text{def}}{=} A_t + \tilde{B}_t + C_t + D_t + E_t + F_t
	\end{align*}
	where we simply identify each of the six terms $A_t$, $\tilde{B}_t$, $C_t$, $D_t$, $E_t$ and $F_t$ with the corresponding expression in the previous formulation. By a simple convexity inequality (H\"older) we have:
	\[\vert X^{k+1}_t-X^k_t \vert^2 \leq 6 \Big( \vert A_t\vert^2 + \vert \tilde{B}_t\vert^2 + \vert C_t\vert^2+ \vert D_t\vert^2+\vert E_t\vert^2+\vert F_t\vert^2\Big)\]
	and treat each term separately. 

	The term $A_t$ is easily controlled using Cauchy-Schwarz inequality, Fubini identity and standard inequalities and we obtain:
	\begin{equation*}
		\Exp \Big[\sup_{\sup_{s\in[0,t]}} \vert A_s \vert^2\Big] 
		 \leq K^2 \,t \, \int_0^t \Exp \Big[ \sup_{-\tau\leq u\leq s} \vert X^{k}_u - X^{k-1}_u \vert^2\Big ] \, ds
	\end{equation*}
	Similarly, the martingale term $D_t$ is bounded using the Burkholder-Davis-Gundy theorem to the $d$-dimensional martingale $(\int_0^t ( g(s,X^{k}_s)-g(s,X^{k-1}_s)) \, dW_s)$ and we obtain:	
	\begin{equation*}
		\Exp \Big[\sup_{0\leq s\leq t} \vert D_s \vert^2\Big] \leq 4 K^2 \, \int_0^t \Exp \Big[ \sup_{-\tau \leq u\leq s} \vert X^{k}_u - X^{k-1}_u \vert^2 \Big]  \, ds 
	\end{equation*}
	
	Let us now deal with the deterministic interaction terms $\tilde{B}_t$ and $C_t$. We have:
	\begin{align*}
		\vert \tilde{B}_t \vert^2 &=\left \vert \int_0^t ds \int_{\Gamma}\lambda(r')dr' \int_{-\tau}^0 d\eta(r,r',u) (\Exp_Z[b(r,r',X^{k}_s,Z^{k}_{s+u}) - b(r,r',X^{k-1}_s,Z^{k}_{s+u})]) \right\vert^2\\
		& \leq t\,\lambda(\Gamma)\,\kappa \int_0^t ds \int_{\Gamma}\lambda(r')dr' \int_{-\tau}^0 d\eta(r,r',u) \Exp_Z\left[ \vert b(r,r',X^{k}_s,Z^{k}_{s+u}) - b(r,r',X^{k-1}_s,Z^{k}_{s+u})\vert^2\right])\\
		&\leq t \lambda(\Gamma)^2 \kappa^2 L^2 \int_0^t \vert X^k_s - X^{k-1}_s\vert^2\,ds \leq t \lambda(\Gamma)^2 \kappa^2 L^2 \int_0^t \sup_{-\tau \leq u\leq s}\vert X^k_u - X^{k-1}_u\vert^2\,ds
	\end{align*}
	hence easily conclude that 
	\[\Exp[\sup_{s\in[0,t]}\vert \tilde{B}_s \vert^2] \leq t \lambda(\Gamma)^2 \kappa^2 L^2 \int_0^t \Exp[\sup_{-\tau \leq u\leq s}\vert X^k_u - X^{k-1}_u\vert^2]\,ds \]
	and similarly 
		\[\Exp[\sup_{s\in[0,t]}\vert C_s \vert^2] \leq t \lambda(\Gamma)^2 \kappa^2 L^2 \int_0^t \Exp[\sup_{-\tau \leq u\leq s}\vert X^k_u - X^{k-1}_u\vert^2]\,ds \]
		
	Eventually, the terms $E_t$ and $F_t$ are treated using Burkholder-David-Gundy (BDG) inequality instead of Cauchy-Schwarz' together with similar arguments as used for $\tilde{B}_t$ and $C_t$. Using the cylindrical nature of the Brownian motions $(B_t)$, BDG inequality yields, for the term $E_t$ (and similarly for the term $F_t$):
	\[\Exp[\sup_{s\in[0,t]}\vert \Theta_s \vert^2] \leq 4 \lambda(\Gamma)^2 \kappa^2 L^2 \int_0^t \Exp[\sup_{-\tau \leq u\leq s}\vert X^k_u - X^{k-1}_u\vert^2]\,ds \]
	for $\Theta_t$ equal to $E_t$ or $F_t$.
	Putting all these estimates together, we get:
	\begin{equation}\label{eq:Bounds}
		\Exp\Big[\sup_{s\in[0,t]} \vert X^{k+1}_s-X^k_s \vert^2 \Big] \leq 6(T+4)(K^2 + 2\lambda(\Gamma)^2\,L^2\,\kappa^2 ) \int_0^t \Exp[ \sup_{-\tau \leq u\leq s} \vert X^{k}_u - X^{k-1}_u \vert^2] ds
	\end{equation}
Moreover, since $ X^{k+1}_t \equiv X^{k}_t$ for $t\in [-\tau, 0]$ by definition, we have, noting
\[M^k_t =\Exp\Big[\sup_{-\tau \leq s\leq t} \vert X^{k+1}_s-X^k_s \vert^2 \Big],\] 
the recursive inequality $M^k_t \leq K'' \int_0^t M^{k-1}_s\,ds$ with $K''=6(T+4)(K^2 + 2 \lambda(\Gamma)^2+\kappa^2 )$, which classically allows concluding on the existence and uniqueness of solutions. 
\end{proof}

\section{Uniform propagation of chaos}\label{append:UniformPropagation}
In this appendix we prove the uniform propagation of chaos property stated in Theorem~\ref{thm:UniformPropaChaos}. In more detail, we prove the following:
\begin{theorem}[Uniform propagation of chaos]
	If the Lipschitz constant of the sigmoid is uniformly bounded by $K_S$, $\theta(r)\in[\theta_m,\theta_M]$ and
	\[K_S\leq \sqrt{\frac{3\theta_M-2\theta_m}{{3\theta_M \theta_m}}},\]
	then the convergence of the network equations towards the mean-field equations is uniform in time, i.e. there exists a constant $C>0$ such that for all $T>0$, for all $i\in\N$ a neuron at location $r\in\Gamma$,
	\[\sup_{0\leq t\leq T} \Exp[ \vert V^i_t - \bar{V}^i_t \vert ] \leq \frac {C}{\sqrt{N}}\]
	for $\bar{V}^i_t$ the a particular process, the coupled process, with law $\bar{V}_t(r)$.
\end{theorem}
\begin{proof}
	The proof of this theorem necessitates a thorough control of the difference between the process $V^i_t$ solution of the network equations and $\bar{V}^i_t$ the \emph{coupled} process, i.e. the solution of the mean-field equation built with the same Brownian motions $W^i_t$ and $B^i_t$ and with the same initial condition as $V^i_t$. It is a refinement of the main theorem of~\cite{touboulNeuralfields:11}. 

	Let us denote by $p:\N\mapsto\N$ the population function associating to a neuron index $j$ the population $\gamma$ it belongs to.  Using the expression of $V^i_t$ and $\bar{V}^i_t$ given implicitly using the variation of constant formula (and making use of the linearity of the intrinsic dynamics), we easily obtain that:
		\begin{multline}\label{eq:vimoinsvibar}
		\displaystyle{V^i_t - \bar{V}^i_t \leq \int_0^t e^{-(t-s)/\theta(r)}  \frac 1 {P(N)} \sum_{j=1}^N \frac 1 {N_{p(j)}}  S(r_{p(j)},V^j_{s-\tau(r,r_{p(j)})})} \\
		\displaystyle{- \int_{\Gamma}J(r,r') \Exp[S(r',\bar{V}_{s-\tau(r,r')}(r'))]\lambda(r')\,dr'  \,ds} \\ 
			 \displaystyle{+  \int_0^t e^{-(t-s)/\theta(r)} \frac 1 {P(N)} \sum_{j=1}^N \frac 1 {N_{p(j)}} S(r_{p(j)},\bar{V}^j_s)dB^i_s -} \\
			\displaystyle{\int_{\Gamma}\int_0^t e^{-(t-s)/\theta(r)} \Exp{[S(r',\bar{V}_{t-\tau(r,r')}(r'))]}\big)\, dB^i_s\lambda(r')\,dr'.}
		\end{multline}
		Introducing the term $\frac 1 {P(N)} \sum_{j=1}^N \frac 1 {N_{p(j)}}  S(r_{p(j)},\bar{V}^j_{s-\tau(r,r_{p(j)})})$ we get:
		\begin{multline}\label{eq:vimoinsvibarDecomp}
		\displaystyle{\vert V^i_t - \bar{V}^i_t \vert^2 \leq 4 \int_0^t e^{-(t-s)/\theta(r)}  \bigg \vert \frac 1 {P(N)} \sum_{j=1}^N \frac 1 {N_{p(j)}}  S(r_{p(j)},V^j_{s-\tau(r,r_{p(j)})}) - S(r_{p(j)},\bar{V}^j_{s-\tau(r,r_{p(j)})})\bigg \vert\,ds} \\
		\displaystyle{+ 4\int_0^t e^{-(t-s)/\theta(r)}  \bigg \vert \frac 1 {P(N)} \sum_{j=1}^N \frac 1 {N_{p(j)}}  S(r_{p(j)},\bar{V}^j_{s-\tau(r,r_{p(j)})})	- \mathcal{E}_{r'}\big[J(r,r') \Exp[S(r',\bar{V}_{s-\tau(r,r')}(r'))] \big]  \bigg \vert\,ds }\\ 
		\displaystyle{	 + 4\bigg \vert \int_0^t e^{-(t-s)/\theta(r)} \frac 1 {P(N)} \sum_{j=1}^N \frac 1 {N_{p(j)}} \Big(S(r_{p(j)},\bar{V}^j_s) - S(r_{p(j)},\bar{V}^j_s) \Big)dB^i_s}\\
		\displaystyle{	+ 4\bigg \vert \int_0^t e^{-(t-s)/\theta(r)} \frac 1 {P(N)} \sum_{j=1}^N \frac 1 {N_{p(j)}} \bigg (S(r_{p(j)},\bar{V}^j_s) - \mathcal{E}_{r'}\Big[\Exp{[S(r',\bar{V}_{t-\tau(r,r')}(r'))]}\Big]\bigg)\, dB^i_s \bigg\vert}\\
		\displaystyle{	=:A_t(r)+B_t(r)+C_t(r)+D_t(r).}
		\end{multline}
		Let us assume that $\theta(r)\in (\theta_m,\theta_M)$ with $\theta_m>0$ and $\theta_M<\infty$, and define $\tilde{\theta}^{\alpha}(r)$ the strictly positive, finite quantity such that $1/\theta(r)=1/(\theta_M+\alpha) +  1/\tilde{\theta}^{\alpha}(r)$ for an arbitrary $\alpha>0$. Let us also denote by $\tau_M$ the quantity $\max_{r,r'\in \Gamma} \tau(r,r')$. 

		The first and third terms are handled using Cauchy-Schwarz and Burkholder-David-Gundy's inequality. The expectation of the first term is readily upperbounded by:
		\[\Exp[A_t^2(r)]\leq 2(\theta_M+\alpha) K_S^2 \int_0^t e^{-2(t-s)/\tilde{\theta}^{\alpha}(r) } \sup_{u\in[s-\tau_M,s]} \sup_{j\in\N}\Exp[{\vert V_{u}^j -\bar{V}_u^j}\vert^2]\]
		and
		\[\Exp[C_t^2(r)]\leq K_S^2\int_0^t e^{-2(t-s)/\theta(r) } \sup_{u\in[s-\tau_M,s]} \sup_{j\in\N}\Exp[{\vert V_{u}^j -\bar{V}_u^j}\vert^2].\]
		The other two terms are treated exactly in the same manner as in~\cite{touboulNeuralfields:11} and yield terms bounded by $K (\mathbbm{e}(N)+1/N)$. Putting these evaluations together we obtain the following inequality on $\Delta_t:=\sup_{i\in\N}\sup_{s\in[t-\tau_M,t]} \Exp[\vert V^i_s-\bar{V}^i_s\vert^2]$:
		\[\Delta_t \leq K (\mathbbm{e}(N)+1/N) + K' e^{\tau_M/\tilde{\theta}^{\alpha}(r)}\int_0^t e^{-(t-s)/\tilde{\theta^{\alpha}(r)}} \Delta_s\,ds.\]
		The exponential term in $\tau_M$ correspond to Applying Gronwall's lemma to the quantity $\Delta_t e^{t/\tilde{\theta}^{\alpha}}$, we obtain that:
		\[\Delta_t e^{t/\tilde{\theta}^{\alpha}} \leq K (\mathbbm{e}(N)+1/N) \left(e^{t/\tilde{\theta}^{\alpha}} +K' e^{\tau_M/\tilde{\theta}^{\alpha}(r)}\int_0^t  e^{s/\tilde{\theta}^{\alpha}}e^{K'(t-s)}\,ds\right) \]
		which is uniformly bounded in $t$ as soon as $K' e^{\tau_M/\tilde{\theta}^{\alpha}(r)}<\frac 1 {\tilde{\theta}^{\alpha}(r)}$, i.e. when $K_S$ and $\tau_M$ are small enough in comparison with $1/\theta(r)$. Fixing $\alpha=\theta_M/2$, we obtain the announced property. 
\end{proof}

\section{Existence and Uniqueness of solutions for the reduced equations}\label{append:ExistenceUniqueness}

We know from~\cite{touboulNeuralfields:11} that there exists a unique solution to the mean-field equations, and that necessarily the solution starting from Gaussian chaotic initial condition is Gaussian with mean and standard deviation satisfying equations~\eqref{eq:DDEIntegroDiff}. We show here that these equations are well posed, distinguishing the finite and infinite population cases. In both cases, we will need to use a regularity property of the function $f$ with respect to $\mu$ and $v$, proved in lemma~\ref{lem:RegularityF}. Note that $f$ can become singular in some cases, for instance when the sigmoid has a singularity and the standard deviation of the solution can reach zero.

Let us first notice that in the case where $S(r,x)$ is equal to  $\erf(g(r)x)$ with $g(r)>0$ bounded, then $f(r,\mu,v)=\erf({g(r)\,\mu}/{1+g(r)^2\,v})$
which is uniformly Lipschitz continuous in $\mu$ and $v$. This was the case in all the applications of the present paper.

In the general case, we have the following property:

\begin{lemma}\label{lem:RegularityF}
	Assume that:
	\begin{itemize}
		\item the sigmoids $S(r,x)$ have derivatives in $x$ uniformly bounded
		\item the diffusion coefficient is lowerbounded: $\min_{r\in\Gamma}(\Lambda(r,t))\geq \Lambda_0>0$ for all $t\geq 0$ 
		\item the initial variance is uniformly lowerbounded by a positive quantity $v^0(r,t)\geq v_0>0$ for $t \in [-\tau,0]$ and $r \in \Gamma$
	\end{itemize}
	Then the function $f(r,\mu,v)$ is uniformly Lipschitz continuous in $\mu$ and $v$ on the trajectories.
\end{lemma}
\proof{
 Thanks to uniform lowerbound of the functions $\Lambda(r,t)$ and of the initial condition on the variances $v^0(r)$, it is easy to show using the integral version of the system~\ref{eq:DDEIntegroDiff} that $v(r,t) \geq v_m\eqdef \min(v_0,\Lambda_{0}^2\theta_m/2)$  (we recall that $\theta_m$ is the strictly positive lower bound of the characteristic times $\theta_{\alpha}$). The function $f$ writes:
\[f(r,x,y) = \frac{1}{\sqrt{2\pi y}}\int_{\R} S(r,z) e^{\frac{(z-x)^2}{2y}}= \int_{\R} S(r,z\sqrt{y} +x) \frac{e^{-z^2/2}}{\sqrt{2\pi}}\,dz.\]
Since Gaussian distribution have exponential moments, it is straightforward to show that derivative with respect to $x$ and $y$ read:
\begin{equation}\label{eq:DiffF}
	\begin{cases}
		\derpart{f(r,x,y)}{x} &= \int_{\R} \derpart{S}{x}(r,z\sqrt{y} +x) \frac{e^{-z^2/2}}{\sqrt{2\pi}}\,dz\\
		\derpart{f(r,x,y)}{y} &= \int_{\R} z\,\derpart{S}{x}(r,z\sqrt{y} +x) \frac{e^{-z^2/2}}{2\sqrt{2\pi y}}\,dz\\
	\end{cases}
\end{equation}
Using the assumption that the sigmoids have bounded derivatives (upperbounded by a quantity denoted  $\Vert \derpart{S}{x}'\Vert_{\infty}$), we have:
\begin{equation*}
	\begin{cases}
		\left \vert \derpart{f(r,x,y)}{x}\right \vert  &\leq \Vert \derpart{S}{x}\Vert_{\infty}\\
		\left \vert \derpart{f(r,x,y)}{y}\right \vert &\leq \frac{\Vert \derpart{S}{x}\Vert_{\infty}}{2\sqrt{v_0}} \int_{\R} \vert z\vert \frac{e^{-z^2/2}}{\sqrt{2\pi}}\,dz =  \frac{\Vert \derpart{S}{x}'\Vert_{\infty}}{\sqrt{v_0}} \int_{0}^{\infty} z \frac{e^{-z^2/2}}{\sqrt{2\pi}}\,dz = \frac{\Vert \derpart{S}{x}'\Vert_{\infty}}{\sqrt{2\pi v_0}} \\
	\end{cases}
\end{equation*}
 This property ensures global Lipschitz continuity of the vector field.
}

We now show existence and uniqueness of solutions of the equations~\eqref{eq:DDEIntegroDiff}. We start deal with the $P$-populations case and denote by $\C$ the Banach space of continuous functions mapping $[-\tau,0]$ into $E^{2P}$ endowed with the topology of the uniform convergence. Following Hale and Lunel~\cite{hale-lunel:93}, we consider the moment equations~\eqref{eq:DDEIntegroDiff} as ordinary differential equations on $\C$. We have the following:

\begin{theorem}\label{thm:ExistenceUniquenessFinite}
Let us assume that $t\mapsto I(r_{\alpha},t)$ and $t\mapsto \Lambda(r_{\alpha},t)$ are continuous, and that $(\mu,v)\mapsto f(r,\mu,v)$ is uniformly Lipschitz continuous,	then there exists a unique solution to the moment equations~\eqref{eq:DDEIntegroDiff} in the finite-population case, starting from an initial condition  $\mu^0(r_{\alpha}) \in C([-\tau,0], \R)$ and $v^0(r_{\alpha})$. 
\end{theorem}

This theorem is a simple application of theorems~\cite[Thm 2.1 and 2.3]{hale-lunel:93} ensuring existence and uniqueness of solutions as soon as the vector field is Lipschitz-continuous in $\C$ and continuous in time. 

Let us now deal with the spatially extended equations.
\begin{theorem}\label{thm:ExistenceUniquenessSpace}
	For the sake of simplicity, we assume here that the density function $\lambda(r)$ is upperbounded by a constant $A$. Under the assumptions that:
	\begin{itemize}
		\item $J$ is square integrable with respect to Lebesgue's measure on $\Gamma^2$, i.e. belongs to $\mathbbm{L}^2(\Gamma^2,\R)$,
		\item $\sigma^2$ is square integrable with respect to Lebesgue's measure on $\Gamma^2$,
		\item the external current $I(r,t)$ is a bounded, continuous functions of time taking values in $\L^2(\Gamma,\R)$ 
		\item the external noise $\Lambda^2(r,t)$ is a continuous functions of time taking values in $\L^2(\Gamma,\R)$, and is uniformly lowerbounded by a strictly positive constant: $\Lambda(r,t)^2 \geq \Lambda_0^2>0$ for all $(r,t)\in \Gamma\times \R^+$,
		\item The derivative of $S(r,x)$ with respect to its second variable is uniformly bounded,
	\end{itemize}
	then for any initial condition $\mu(r,t)\in C([-\tau,0],\L^2(\Gamma,\R))$ and $v(r,t)\in C([-\tau,0],\L^2(\Gamma))$ uniformly lowerbounded by a quantity $v_0>0$, there exists a unique solution to the moments mean-field equations~\eqref{eq:MFEFiringRateSpace} which moreover belongs to $ C([-\tau,T],L^2(\Gamma,\R^2))$.
\end{theorem}

\begin{proof}
	The moment mean-field equations~\eqref{eq:MFEFiringRateSpace} constitute a dynamical system in the Banach spaces of functions of $\Gamma$ with values in $\R^2$. It is well known that the space of functions in $\mathcal{B}\eqdef C([-\tau,T],L^2(\Gamma,\R^2))$ endowed with the norm:
	\[\|(\varphi_1,\varphi_2)\|_{\mathcal{B}} = \sup_{s\in[-\tau,T]}\Big(\int_{\Gamma} \vert\varphi_1(r,s)\,dr\vert^2 + \int_{\Gamma} \vert\varphi_2(r,s)\vert^2\,dr\Big)^{1/2}\]
	is a Banach space. We will show the existence and uniqueness of solutions in this space. We further define the norm up to time $t>0$ of two elements of $\mathcal{B}$ by:
	\[D_t(\varphi_1,\varphi_2)=\sup_{s\in[-\tau,t]}\Big(\int_{\Gamma} \vert\varphi_1(r,s)\,dr\vert^2 + \int_{\Gamma} \vert\varphi_2(r,s)\vert^2\,dr\Big).\]
	
	Let us start by ensuring that any possible solution is bounded in this space. We recall that:
		\begin{multline*}
			\displaystyle{\mu(r,t) =e^{-\frac{t}{\theta(r)}}\bigg(\Ex{V^0_0(r)}+\int_0^t e^{\frac{s}{\theta(r)}} \Big ( I(r,s)} \\
			 \displaystyle{+\int_{\Gamma} \lambda(r')dr' J({r,r'}) f(r',\mu(r',s-\tau(r,r')),v(r',s-\tau(r,r')))\Big) 
			ds\bigg)}
		\end{multline*}
		and hence we have:
	\begin{align*}
		\Vert \mu(\cdot,t) \Vert_{\L^2(\Gamma,\R)}^2 
		 &\displaystyle{\leq 3\,\bigg(\Vert \Ex{V^0_0(\cdot)}\Vert_{\L^2(\Gamma,\R)}^2 +\int_{\Gamma}\vert \int_0^t I(r,s)\,ds \vert^2\,dr } \\
		&\displaystyle{\qquad +\int_{\Gamma}\,dr \vert \int_0^t ds \int_{\Gamma} J({r,r'}) f(r',\mu(r',s-\tau(r,r')),v(r',s-\tau(r,r')))\vert^2\bigg) \lambda(r')dr'}\\
		&\displaystyle{\leq 3\,\bigg(\Vert \Ex{V^0_0(\cdot)}\Vert_{\L^2(\Gamma,\R)}^2 + T^2 \sup_{t\in[0,T]} \Vert I(\cdot,s)\Vert^2_{\L^2(\Gamma,\R)}} \\
		&\displaystyle{\qquad +T \lambda(\Gamma) \int_{\Gamma}  \int_0^t ds \int_{\Gamma} \lambda(r')dr' \vert J({r,r'})\vert^2 \Vert f\Vert_{\infty}^2\bigg)\,dr}\\
		&\displaystyle{\leq 3\,\bigg(\Vert \Ex{V^0_0(\cdot)}\Vert_{\L^2(\Gamma,\R)}^2 + T^2 \sup_{t\in[0,T]} \Vert I(\cdot,s)\Vert^2_{\L^2(\Gamma,\R)} +T^2 A\,\lambda(\Gamma) \Vert J({r,r'})\Vert^2_{\mathbbm{L}^2(\Gamma^2,\R)} \Vert f\Vert_{\infty}^2\bigg)}
	\end{align*}
	where $\Vert f\Vert_{\infty}$ is the uniform upperbound of $f(r,\mu,v)$ in  $\R$, which is smaller or equal to the uniform supremum of the function $x\mapsto S(r,x)$ (which exists by assumption). The same types of calculations allow proving that:
	\[\Vert v(\cdot,t) \Vert_{\L^2(\Gamma,\R)}^2 \leq 3\,\bigg(\Vert v(\cdot,0)\Vert_{\L^2(\Gamma,\R)}^2 + T^2 \sup_{t\in[0,T]} \Vert \Lambda^2(\cdot,s)\Vert^2_{\L^2(\Gamma,\R)} +T^2 (\lambda^2)(\Gamma) A^2 \Vert \sigma^2({r,r'})\Vert^2_{\mathbbm{L}^2(\Gamma^2,\R)} \Vert f\Vert_{\infty}^4\bigg).\]
	
	These bounds do not depend upon time $t$ and hence prove that any solution of the moment mean-field equations have bounded norms in the space $\mathcal{B}$. 
	
	Routine methods for this type of infinite-dimensional systems ensure existence and uniqueness of solutions as soon as the vector field of the equation is Lipschitz-continuous for this norm. In our case, lemma~\ref{lem:RegularityF} ensures the global uniform in $r$ Lipschitz-continuity of the function $(\mu,v)\mapsto f(r,\mu,v)$. 
	Let us define for $(\mu,v) \in \mathcal{B}$ the transformation $\Phi(\mu,v)$ taking values in $\mathcal{B}$ and defined by:
		\[\left (
		\begin{array}{l}
			e^{-\frac{t}{\theta(r)}}\bigg(\mu(0,r) + \int_0^t e^{\frac{s}{\theta(r)}}\Big(I(r,s)+\int_{\Gamma} \lambda(r')dr' J({r,r'}) f(r',\mu(r',s-\tau(r,r')),v(r',s-\tau(r,r')))\Big) ds \bigg)\\
			e^{-\frac{2t}{\theta(r)}}\bigg(v(0,r) + \int_0^t e^{\frac{2s}{\theta(r)}}\Big(\Lambda^2(r,s)+\int_{\Gamma} \lambda(r')^2 dr' \sigma^2({r,r'}) f^2(r',\mu(r',s-\tau(r,r')),v(r',s-\tau(r,r')))\Big) ds \bigg)
		\end{array}\right)\]
		It is clear that any solution of the moment equations are fixed points of $\Phi$ and reciprocally, fixed points of $\Phi$ are solutions of the moment equations. Since $(\mathcal{B}, \Vert\cdot \Vert_{\mathcal{B}})$ is a Banach space, showing existence and uniqueness of solutions, i.e. of fixed points of $\Phi$, amounts showing a contraction property on $\Phi$. First of all, similarly to what was done to show that any solutions of the moment equations were bounded in $\mathcal{B}$, it is very easy to show that for any $(\mu,v)\in\mathcal{B}$, we have $\Phi(\mu,v)\in\mathcal{B}$. We use the classical iteration method to show existence and uniqueness of fixed point. To this end, we fix $\varphi^0=(\mu^0,v^0)\in\mathcal{B}$ arbitrarily and define the sequence $\varphi^n=(\mu^{n}, v^{n})_{n\in\N}$ iteratively by setting $(\mu^{n+1},v^{n+1})=\Phi(\mu^n,v^n)$. We recall that $f$ is Lipschitz-continuous as shown in the proof of Theorem~\ref{thm:ExistenceUniquenessFinite}, and we denote by $L$ the uniform Lipschitz constant of $f(r,x,y)$ in its two last variables. The function $f^2(r,x,y)$ is hence also uniformly Lipschitz-continuous in its two last variables with the Lipschitz constant $2\Vert f\Vert_{\infty} L$.  Let us now show that the vector field $\Phi$ is Lipschitz-continuous on $\mathcal{B}$. Let us fix $\varphi_1=(\mu_1,v_1)$ and $\varphi_2=(\mu_2,v_2)$ two elements of $\mathcal{B}$. We have:
	\begin{align*}
		D_t(\varphi_1,\varphi_2) & = \displaystyle{\sup_{s\in[-\tau,t]} \bigg\{\int_{\Gamma}\,dr \Big\vert \int_0^t\int_{\Gamma} \lambda(r')dr' J({r,r'}) \Big(f(r',\mu_1(r',s-\tau(r,r')),v_1(r',s-\tau(r,r'))) }\\
		& \displaystyle{\qquad \qquad -f(r',\mu_2(r',s-\tau(r,r')),v_2(r',s-\tau(r,r')))\Big)ds\Big \vert ^2}\\
		& \displaystyle{\qquad + \int_{\Gamma}\,dr \Big\vert \int_0^t\int_{\Gamma} \lambda^2(r')dr' \sigma^2({r,r'}) \Big(f^2(r',\mu_1(r',s-\tau(r,r')),v_1(r',s-\tau(r,r')))} \\
		& \displaystyle{\qquad \qquad -f^2(r',\mu_2(r',s-\tau(r,r')),v_2(r',s-\tau(r,r')))\Big)ds\Big \vert ^2\bigg\}}
	\end{align*}
	The two terms of the righthand side are treated similarly, let us hence deal with the first one. We have:
	\begin{align*}
		& \int_{\Gamma}\,dr \Big\vert \int_0^t\int_{\Gamma} \lambda(r')dr' J({r,r'}) \Big(f(r',\mu_1(r',s-\tau(r,r')),v_1(r',s-\tau(r,r'))) \\
		&\qquad \qquad -f(r',\mu_2(r',s-\tau(r,r')),v_2(r',s-\tau(r,r')))\Big)ds\Big \vert ^2\\
		&\quad \leq T \int_{\Gamma}\,dr  \int_0^t \Big\vert\int_{\Gamma} \lambda(r')dr' J({r,r'}) \Big(f(r',\mu_1(r',s-\tau(r,r')),v_1(r',s-\tau(r,r'))) \\
		&\qquad \qquad -f(r',\mu_2(r',s-\tau(r,r')),v_2(r',s-\tau(r,r')))\Big)ds\Big \vert ^2\\
		&\quad \leq T \int_{\Gamma}\,dr  \int_0^t \Big(\int_{\Gamma} \lambda(r')dr' J({r,r'})^2\Big)\bigg(\int_{\Gamma} \lambda(r')dr' \Big(f(r',\mu_1(r',s-\tau(r,r')),v_1(r',s-\tau(r,r'))) \\
		&\qquad \qquad -f(r',\mu_2(r',s-\tau(r,r')),v_2(r',s-\tau(r,r')))\Big)^2 \bigg)ds\\
		\end{align*}
		\begin{align*}
		&\quad \leq 2\,L^2\,T \int_{\Gamma}\,dr  \int_0^t \Big(\int_{\Gamma} \lambda(r')dr' J({r,r'})^2\Big)\bigg(\int_{\Gamma} \lambda(r')dr' \vert \mu_1(r',s-\tau(r,r'))-\mu_2(r',s-\tau(r,r'))\vert^2 \\
		&\qquad\qquad + \vert v_1(r',s-\tau(r,r')) -v_2(r',s-\tau(r,r')))\vert^2 \bigg)ds\\
		&\quad \leq 2\,A^2\,L^2\,T \int_{\Gamma}\,dr  \int_0^t \int_{\Gamma} dr' J({r,r'})^2 \Vert \varphi_1(\cdot,s-\tau(r,r'))-\varphi_2(\cdot,s-\tau(r,r'))\Vert^2_{\L^2(\Gamma,\R^2)} \\
		&\quad \leq 2\,A^2\,L^2\,T^2 \Vert J \Vert_{L^2(\Gamma\times\Gamma)}^2 \; \int_0^t D_s(\varphi_1,\varphi_2)\,ds
	\end{align*}
	Similarly, the second term is upperbounded by $8\,A^4\,L^2 \Vert f \Vert_{\infty}^2 \,T^2 \Vert \sigma^2 \Vert_{L^2(\Gamma\times\Gamma)}^2 \; \int_0^t D_s(\varphi_1,\varphi_2)\,ds$. These two bounds do not depend upon time, hence we have:
	\[ D_t(\varphi_1,\varphi_2) \leq K' \int_0^t D_s(\varphi_1,\varphi_2)\,ds\]
	with $K'=2\,A^2\,L^2\,T^2 (\Vert J \Vert_{L^2(\Gamma\times\Gamma)}^2 + 4\,A^2 \Vert \sigma^2 \Vert_{L^2(\Gamma\times\Gamma)}^2)$. Routine methods allow showing that:
	\[D_t(\varphi^{n},\varphi^{n-1}) \leq \frac{(K't)^n}{n!}D_T(\varphi^{1},\varphi^{0})\]
	and hence $\Vert\varphi^{n}-\varphi^{n-1}\Vert_{\mathcal{B}} \leq \sqrt{(K't)^n/n!} \Vert\varphi^{1}-\varphi^{0}\Vert_{\mathcal{B}}$, readily implying that $\varphi^n$ is a Cauchy sequence in the complete space $\mathcal{B}$ and hence converging in $\mathcal{B}$ towards a fixed point of $\Phi$. Uniqueness of the solution is also classically deduced from the above inequality. 
\end{proof}

\section{Spatio-temporal patterns for one-dimensional neural fields with reflective or zero boundary conditions}\label{append:BoundaryConditions}
In this appendix, we reproduce the results obtained in section~\ref{sec:periodicBounds} with different boundary conditions, functional connectivity kernel size, and synaptic coefficients characterizing network (II).

For reflective boundary conditions,  Proposition~\ref{pro:Synchro} applies and hence spatially homogeneous solutions exist and incidentally satisfy the system given by equations~\eqref{sec:periodicBounds} and hence the dependence on noise levels is given by the bifurcation diagram of Fig.~\ref{fig:BifDiagsDDE}. For zero boundary conditions, the neural field does not satisfy Proposition~\ref{pro:Synchro}. The spatial dynamics is numerically investigated using the same procedure as done in the main text, and it appears that in both cases the phenomena observed in the periodic neural field persist. 

For reflective boundary conditions, non spatially homogeneous persistent solutions arise are in the parameter region (A) (small noise). For small values of the noise parameter, the inhomogeneous initial conditions produce large amplitude waves that interact together, creating a transient complex structure of spatio-temporal activity that stabilizes on a fully synchronized stationary solution. For slightly larger noise amplitude, this transient irregular phase takes over and produce sustained dynamic irregular activity on the neural field. As noise is further increased, a spatio-temporal periodic activity arises as the waves become faster and stop interacting. When the noise parameter reaches parameter regions (B) or (C) the spatially homogeneous oscillatory activity is recovered after a short transient. For noise levels in the parameter region (D), the whole neural field converges towards the unique spatially stationary activity.  
\begin{figure}
	\begin{center}
		\subfigure[Reflective, $\Lambda=0.3$]{\includegraphics[width=.30\textwidth]{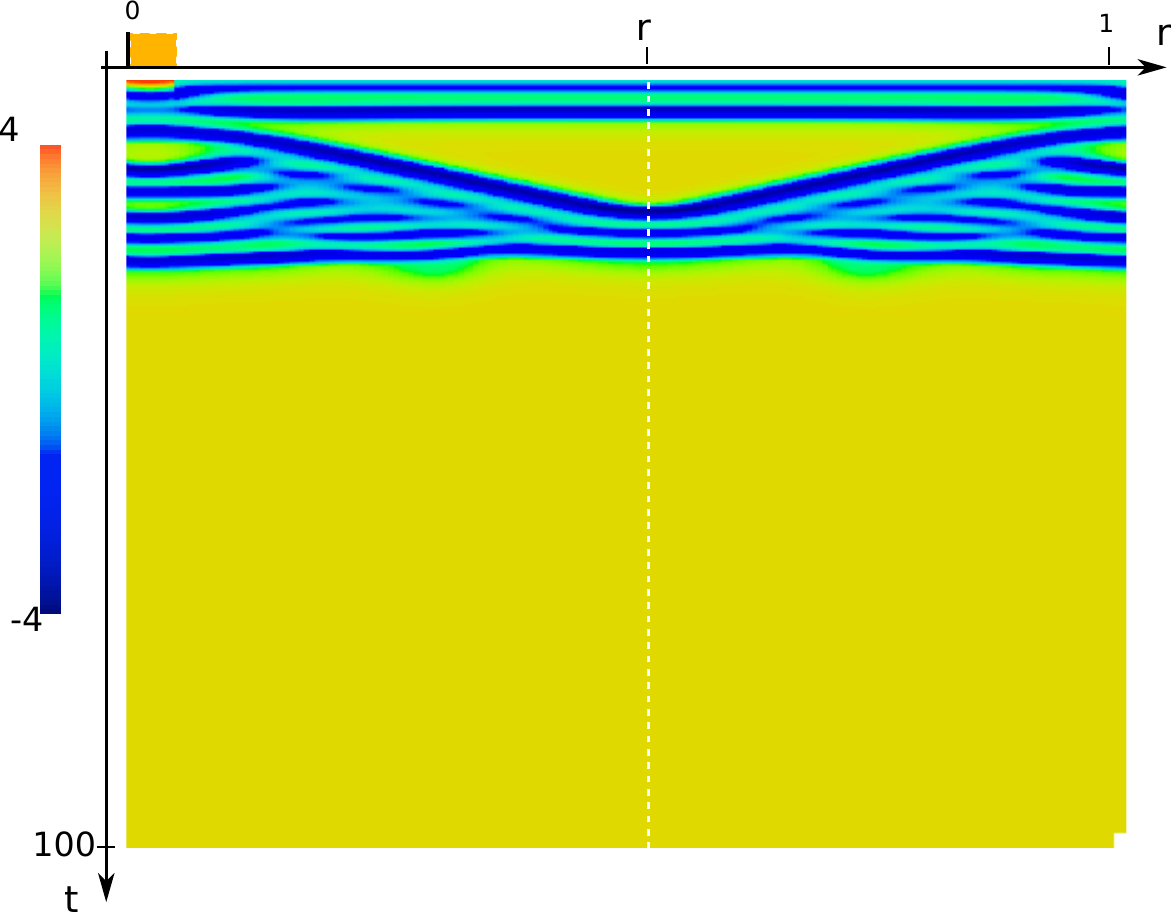}}\quad
		\subfigure[Reflective, $\Lambda=0.6$]{\includegraphics[width=.30\textwidth]{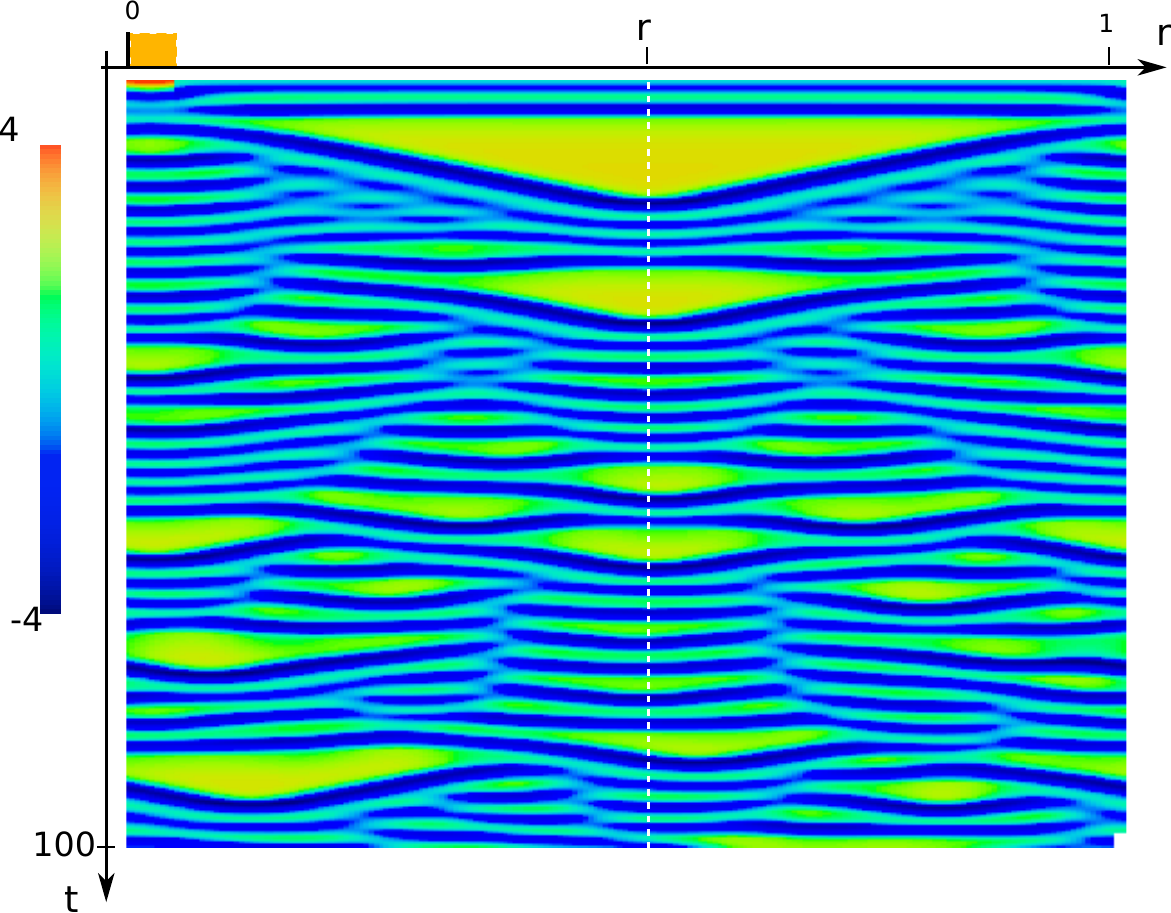}}\quad
		\subfigure[Reflective, $\Lambda=1$]{\includegraphics[width=.30\textwidth]{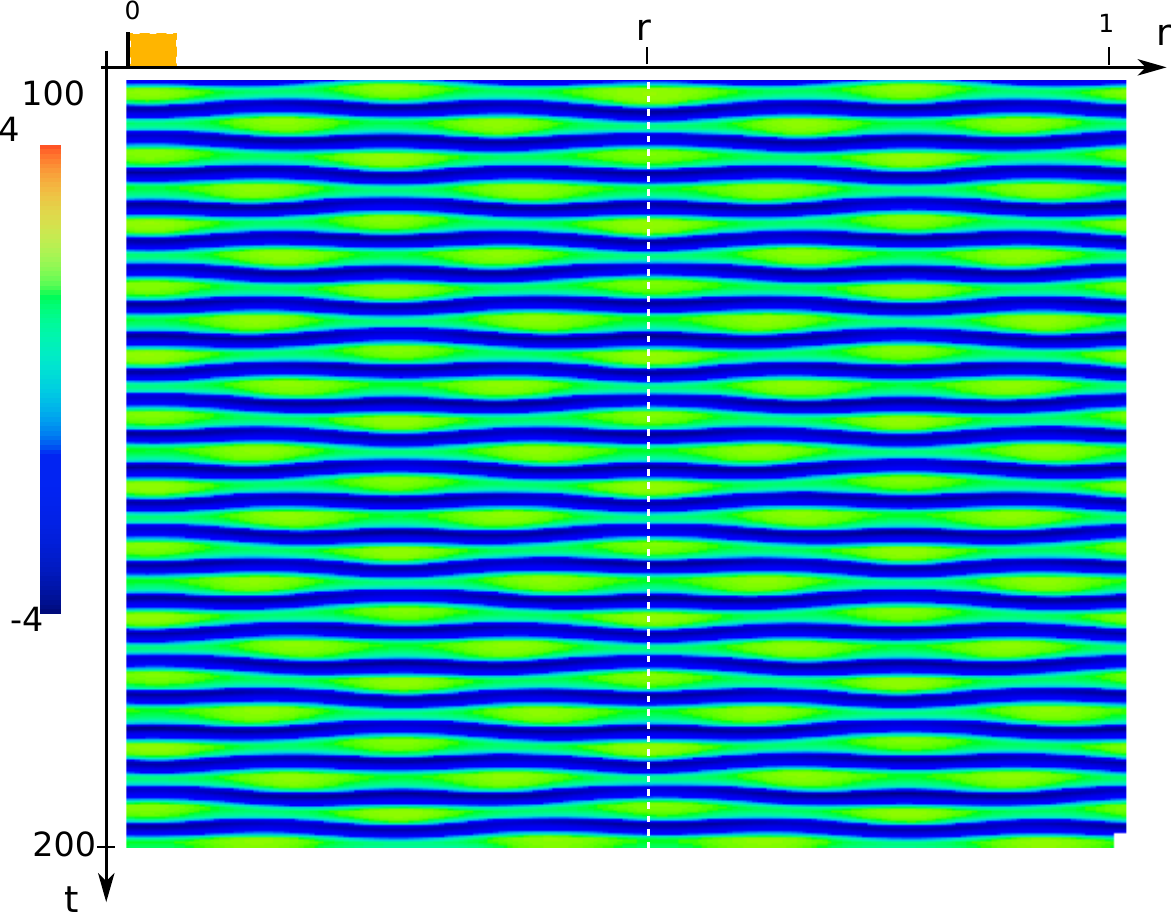}}\\
		\subfigure[Reflective, $\Lambda=1.5$]{\includegraphics[width=.30\textwidth]{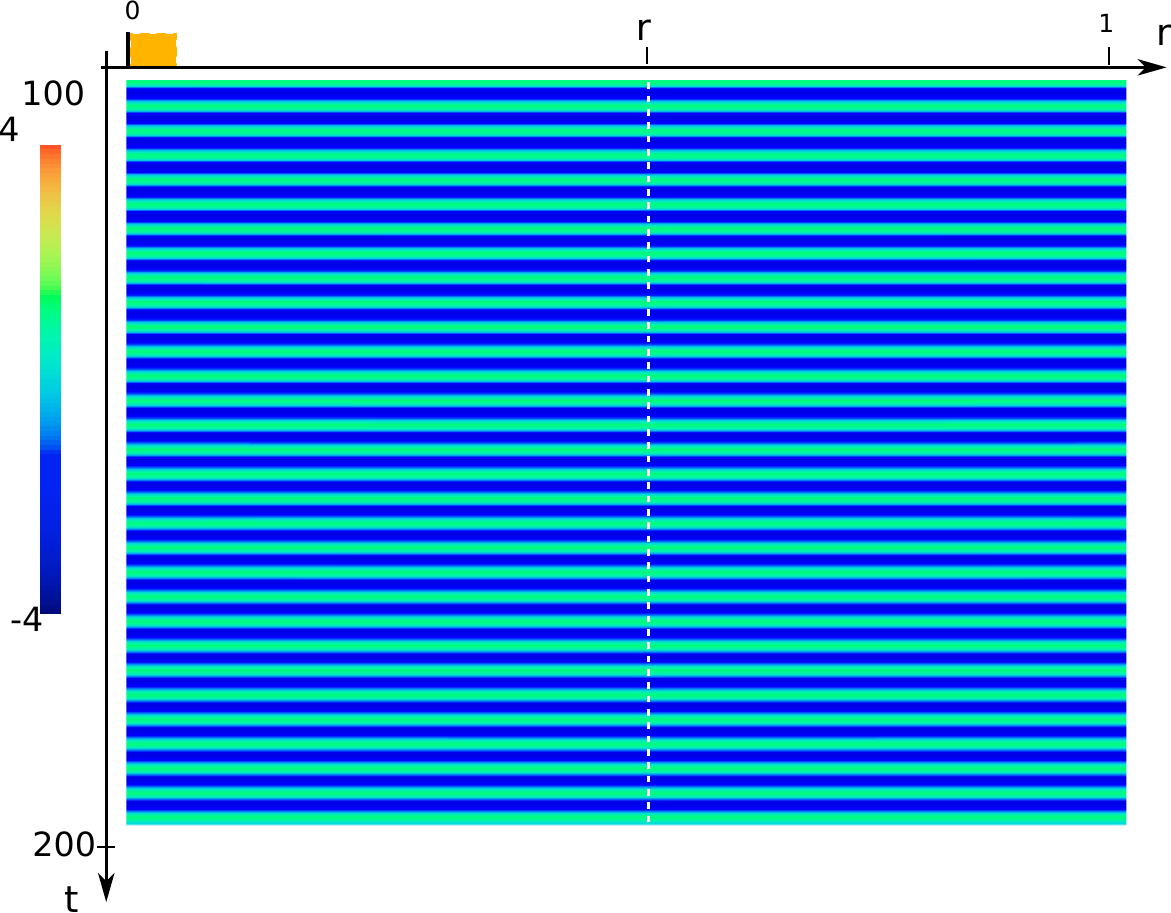}}\quad
		\subfigure[Reflective, $\Lambda=3$]{\includegraphics[width=.30\textwidth]{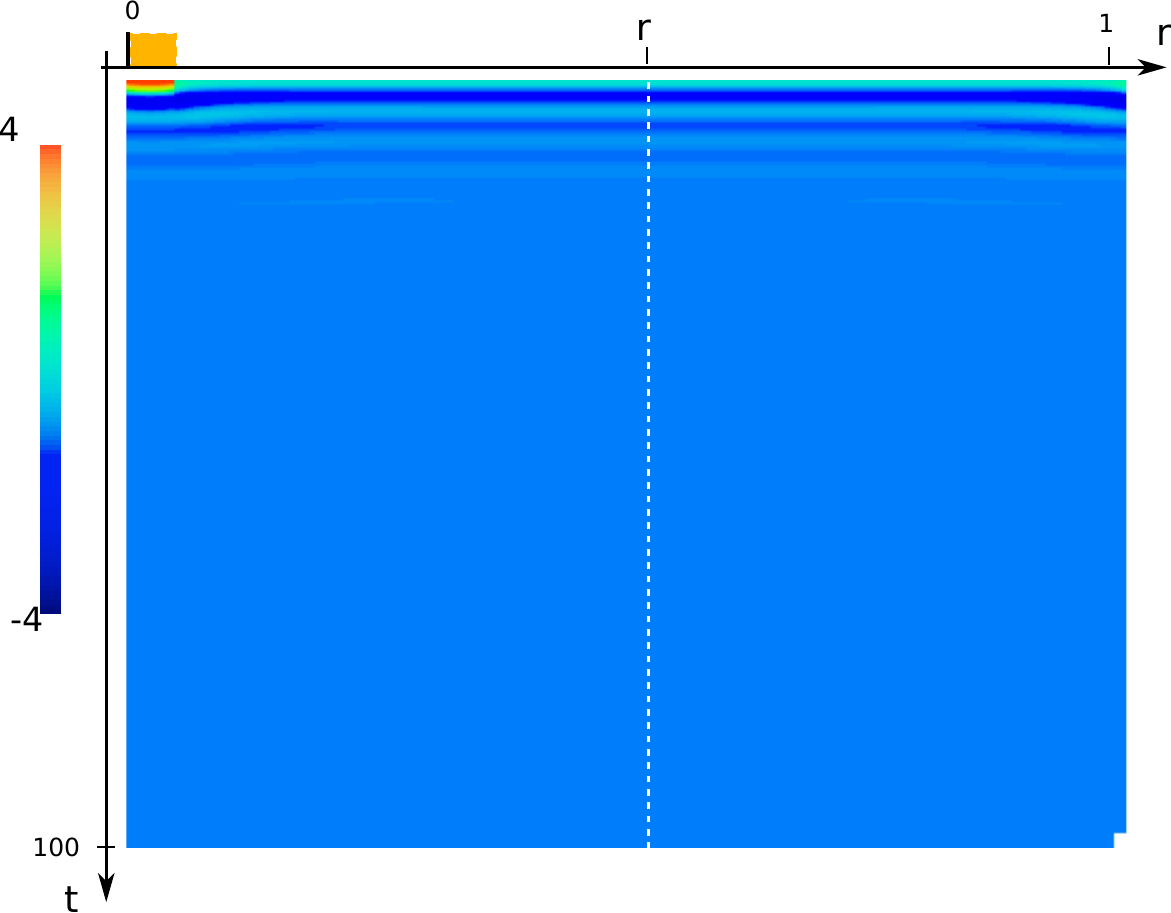}}\\
		\subfigure[Zero, $\Lambda=0.1$]{\includegraphics[width=.30\textwidth]{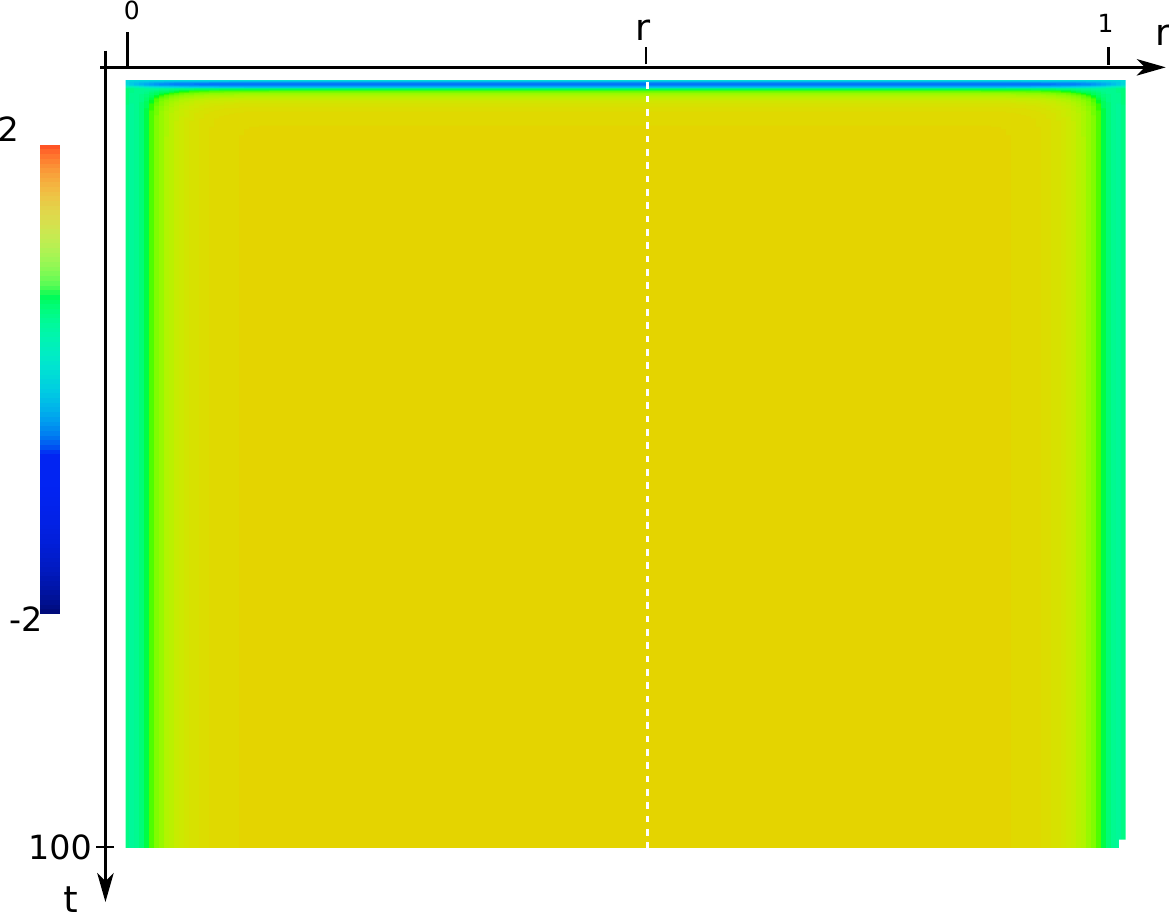}\label{fig:turing0_2Zero.pdf}}\quad
		\subfigure[Zero, $\Lambda=0.6$]{\includegraphics[width=.30\textwidth]{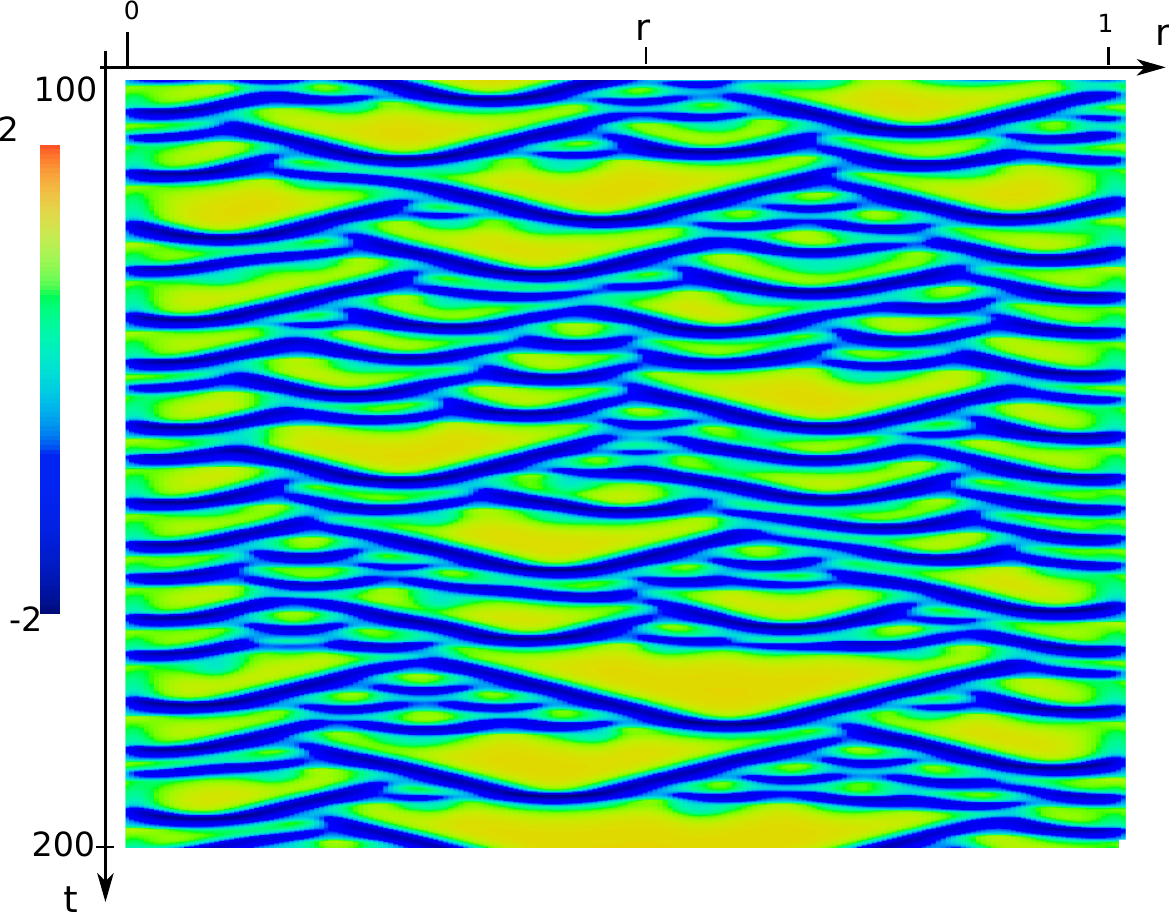}}\quad
		\subfigure[Zero, $\Lambda=1.5$]{\includegraphics[width=.30\textwidth]{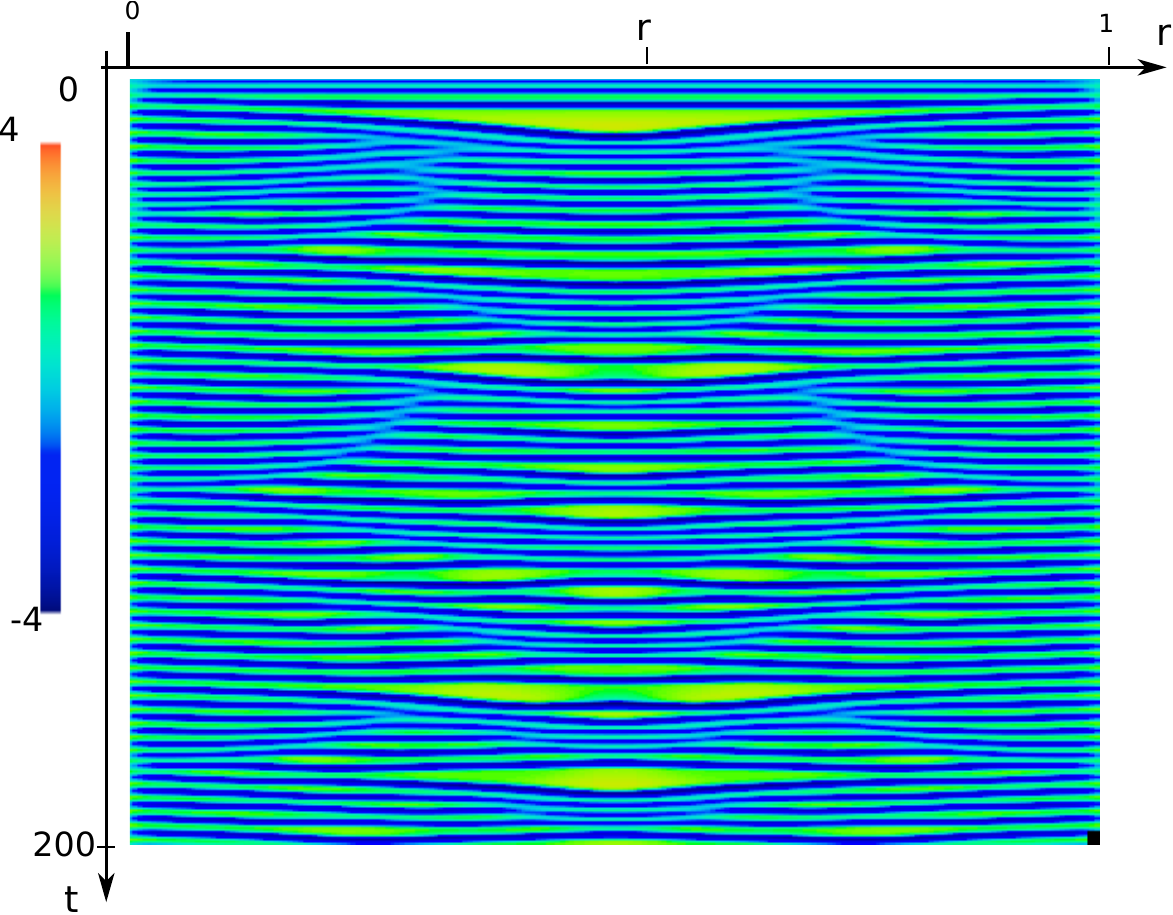}\label{fig:turing1_5Zero}}\\
		\subfigure[Zero, $\Lambda=2$]{\includegraphics[width=.30\textwidth]{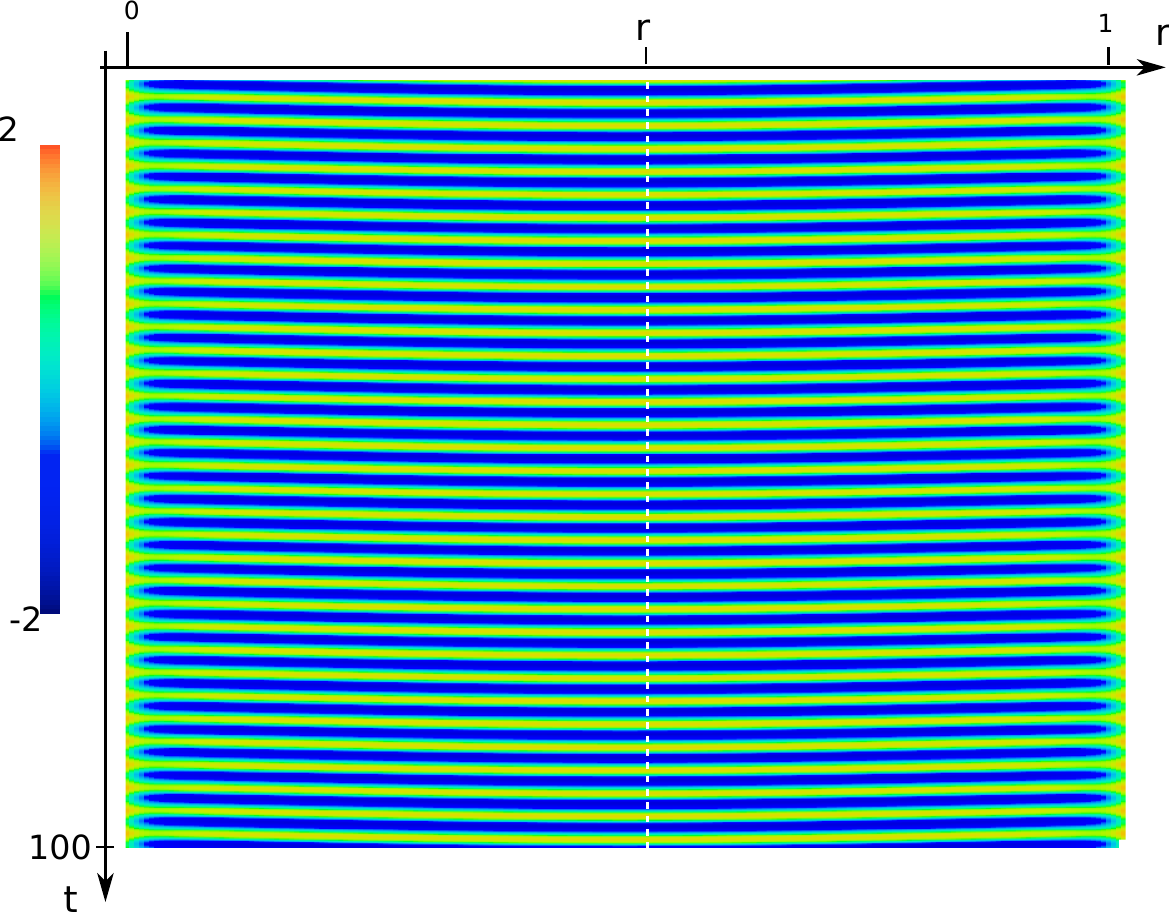}\label{fig:turing1_2Zero}}\quad
		\subfigure[Zero, $\Lambda=3$]{\includegraphics[width=.30\textwidth]{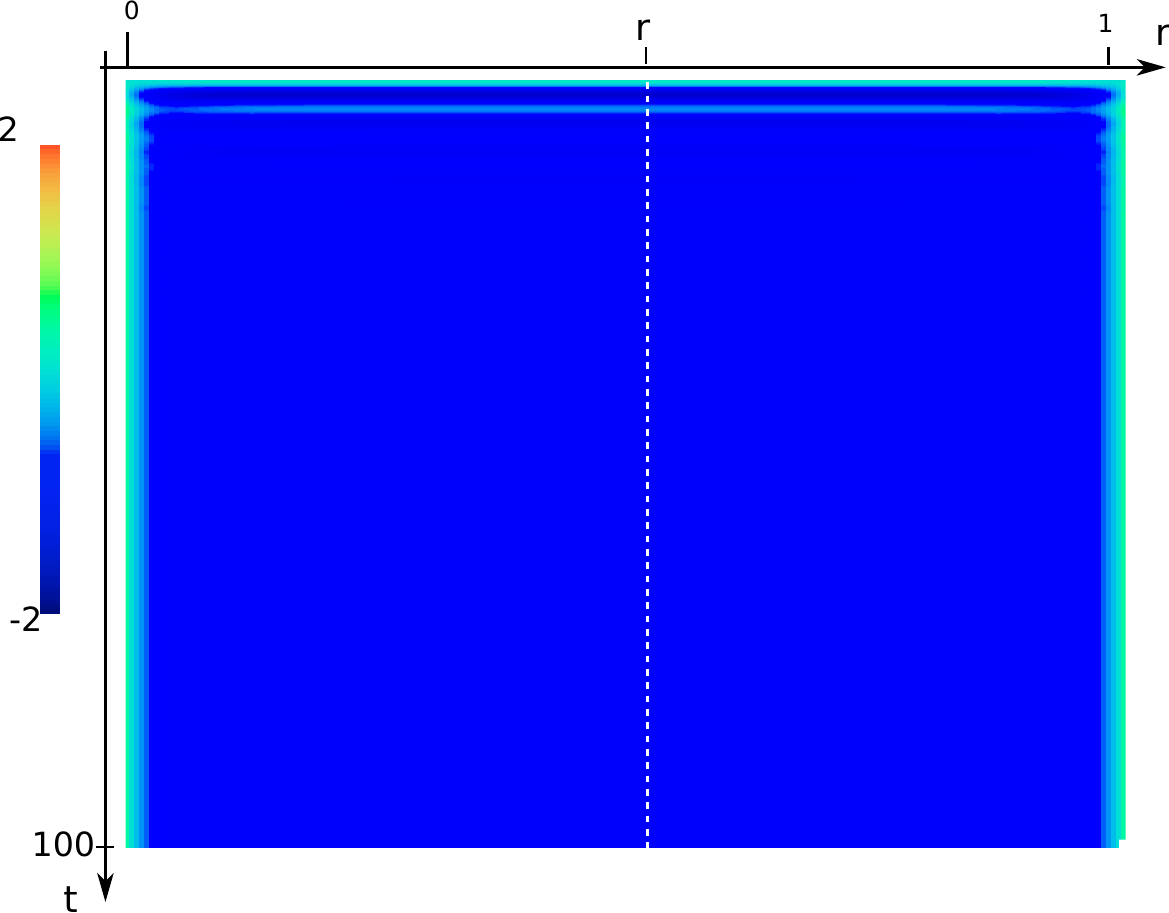}}
	\end{center}
	\caption{Reflective or zero boundary conditions: activity as the noise amplitude is increased. For reflective boundaries, initial conditions are $5$ on $[0,0.025]$ (orange box) and $0$ otherwise. For zero boundary, initial conditions are homogeneous equal to zero. Animations of the activity are available in the supplementary material.}
	\label{fig:turingReflected}
\end{figure}

In the case of zero boundary conditions, we investigate the dynamics using homogeneous initial conditions. Inhomogeneities arise for neurons close to the border of the neural field that receive less input that neurons far from the border. However, qualitative regimes observed in the periodic case are recovered in this case as noise is increased: stationary solutions, chaotic wave-splitting, synchronized oscillations (of amplitude depending on the position on the neural field), and then back to a stationary solution. No transient wave-splitting was found, and the spatio-temporal oscillations were replaced by partially synchronized solutions (Fig.~\ref{fig:turingReflected}(j))

\section*{Acknowledgements}
The author acknowledges the help of anonymous reviewers for the readability and the structure of the paper. He wants to warmly thank Romain Veltz for interesting technical discussions on the content and on relevant references, and David Colliaux for discussions on the choice of connectivity kernels, and partial funding from ERC grant \#227747.

\end{document}